\documentclass{article}

\usepackage[utf8]{inputenc}
\usepackage{lipsum}
\usepackage{url}
\usepackage{graphicx,graphbox}
\usepackage[T1]{fontenc}
\usepackage[utf8]{inputenc}
\usepackage[english]{babel}

\usepackage{times}

\usepackage[left=1.8cm, right=1.8cm, top=2.5cm, bottom=2.5cm]{geometry}
\usepackage{amsmath}
\usepackage{amssymb}
\usepackage{amsthm}
\usepackage{hyperref}
\usepackage{graphicx}
\usepackage{subcaption}
\usepackage{caption}
\usepackage{xcolor}
\usepackage{longtable}
\usepackage{verbatimbox}

\usepackage[title]{appendix}
\usepackage{bm}
\usepackage{empheq}
\usepackage{placeins}
\theoremstyle{plain}
\newtheorem{theorem}{Theorem}[section]

\newtheorem{remark}[theorem]{Remark}

\theoremstyle{definition}
\theoremstyle{remark}
\numberwithin{equation}{section}

\usepackage[
backend=biber,
style=numeric,
sorting=nty,
maxbibnames=50,
giveninits=true
]{biblatex}
\usepackage{csquotes}
\renewbibmacro{in:}{%
 \ifentrytype{article}{}{}}
\DeclareFieldFormat[article]{volume}{\textbf{#1}}
\DeclareFieldFormat[article]{title}{\textit{#1}}
\DeclareFieldFormat[article]{journaltitle}{#1}
\DeclareFieldFormat[article]{number}{\textbf{#1}}
\renewbibmacro*{volume+number+eid}{%
  \printfield{volume}%
\setunit*{\textbf{\addcolon}}
  \printfield{number}\setunit{\addcomma\space}%
  \printfield{eid}}

\usepackage{xcolor}
\hypersetup{
colorlinks,
linkcolor={red!50!black},
citecolor={blue!50!black},
urlcolor={blue!80!black}
}
\theoremstyle{plain} 

\DeclareMathOperator{\diver}{div}

\DeclareMathOperator*{\esssup}{ess\,sup}

\usepackage{mathtools}
\usepackage{pgfplots}
\pgfplotsset{/pgf/number format/use comma,compat=newest}

\newcommand{\R}{\mathbb{R}}
\newcommand{\N}{\mathbb{N}}

\newcommand{\vect}[1]{\boldsymbol{#1}}
\newcommand{\tens}[1]{\mathsf{#1}}

\newcommand{\di}{{\vect{d}}}
\newcommand{\DD}{{\vect{D}}}
\newcommand{\p}{{\vect{\Pi}}}
\newcommand{\dtilde}{{\tilde{\vect{d}}}}
\newcommand{\vv}{{\vect{v}}}
\newcommand{\zz}{{\vect{\zeta}}}
\newcommand{\uu}{{\vect{u}}}
\newcommand{\UU}{{\vect{U}}}
\newcommand{\hh}{{\vect{h}}}
\newcommand{\ww}{{\vect{w}}}
\newcommand{\ff}{{\vect{f}}}
\newcommand{\abs}[1]{\left\lvert#1\right\rvert}
\newcommand{\norm}[1]{\left\|#1\right\|}

\newcommand{\prt}[1]{\left(#1\right)}
\newcommand{\dt}{\frac{{\rm d}}{{\rm d}t}}

\def\d{{\rm d}}
\def\ddt{\frac{{\rm d}}{{\rm d}t}}

\begin{document}
\title{\textsc{
Global Solutions for Two-Phase Complex Fluids with Quadratic Anchoring in Soft Matter Physics
}}

\author{
\textsc{Giulia Bevilacqua}\\[7pt]
\small Dipartimento di Matematica\\ 
\small Università di Pisa\\ 
\small Largo Bruno Pontecorvo 5, I–56127 Pisa, Italy\\
\small \href{mailto:giulia.bevilacqua@dm.unipi.it}{giulia.bevilacqua@dm.unipi.it}
\and
\textsc{Andrea Giorgini}\\[7pt]
\small  Dipartimento di Matematica\\
\small Politecnico di Milano
\\ 
\small Via E. Bonardi 9, I-20133 Milano, Italy
\\
\small \href{mailto:andera.giorgini@polimi.it}{andrea.giorgini@polimi.it}
}

\date{}

\maketitle

\begin{abstract}
We study a diffuse interface model describing the complex rheology and the interfacial dynamics during phase separation in a polar liquid-crystalline emulsion. More precisely, the physical systems comprises a two-phase mixture consisting in a polar liquid crystal immersed in a Newtonian fluid. Such a composite material is a paradigmatic example of complex fluids arising in Soft Matter which exhibits multiscale interplay. Beyond the Ginzburg-Landau and Frank elastic energies for the concentration and the polarization, the free energy of the system is characterized by a quadratic anchoring term which tunes the orientation of the polarization at the interface. This leads to several quasi-linear nonlinear couplings in the resulting system describing the macroscopic dynamics. In this work, we establish the first mathematical result concerning the global dynamics of two-phase complex fluids with interfacial anchoring mechanism. First, we determine a set of sufficient conditions on the parameters of the system and the initial conditions which guarantee the existence of global weak solutions in two and three dimensions. Secondly, we show that weak solutions are unique and globally regular in the two dimensional case. Finally, we complement our analysis with some numerical simulations to display polarization and interfacial anchoring.
\end{abstract}

\bigskip

\textbf{Mathematics Subject Classification (2020)}: 35D30, 35Q35, 76D03, 76T05

\textbf{Keywords}: Polar liquid-crystalline emulsion, quadratic anchoring, global weak solutions, regularity

\bigskip

\maketitle

\tableofcontents

\section{Introduction}

Design and synthesis of new composite materials with enhanced physical, chemical and mechanical properties are subject of great interest for various applications in many fields of engineering, biology and medicine \cite{10.3389/fmats.2014.00011, kumar2020inverse}. Main examples are complex fluids and soft nanocomposites, which consist of mixtures of liquids or gases containing particles of different substances dispersed within them \cite{drzaic1995liquid, krishnan2010rheology, toor2016self}. 
These materials exhibit engineered
microstructures which arise from the interplay at different length scales, such as molecular assembly, mesoscopic interfacial morphology and macroscopic dynamics. The most representative example is the competition between phase separation (material's ability to spontaneously segregate into two or more different phases) and viscous or elastic response of the material \cite{mezzenga2021grand,  michieletto2022rheology, riahinasab2019nanoparticle, tanaka2022viscoelastic}. 
In the last decades, scientists have extensively developed this approach to design soft-phase materials, such as soft gels, suspensions, colloidal particles, porous materials, bio-polymers and bio-molecules, which have shown vast practical utility in electronics, energy storage, bio-sensors, bio-printing, biomedical devices, tissue engineering, drug delivery, and water/gas separation membranes
\cite{ajayan2006nanocomposite, beecroft1997nanocomposite,  bevilacqua2020faraday, elbert2011liquid, kumar2020inverse}.
Nonetheless, arresting the mixing process at the desired length scale still remains nowadays a fundamental challenge in controlling the microstructures of phase separation materials \cite{fernandez2021putting}.

A thriving method to develop novel nano-composite materials is  polymerization-induced phase separation (PIPS) \cite{elbert2011liquid, kim1993polymerization, tran2017phase}. The composite material consists of a multi-phase mixtures of liquid crystals, namely anisotropic fluids whose constituent molecules exhibit local orientational order, immersed in an isotropic polymer matrix, which are called {\em liquid-crystalline emulsions}. In such mixture, the ultimate pattern of the structured material results in a competition between phase separation and polymerization (cross-linking) (see Figure \ref{fig:PS}). 
At early stages, the cross-linking reaction activates and the mixture becomes unstable at constant temperature, in contrast to the usual thermal quench. Then, phase separation occurs leading the formation of liquid crystalline droplets or regions in the polymer phase \cite{fernandez2021putting}. Meantime, surface anchoring conditions and confinement at the interface can produce macroscopic domains with a defined optic axis, elastic deformations acting on the polymeric phase and a variety of topological defects altering their overall interactions and stability \cite{amundson1998liquid, riahinasab2019nanoparticle}.
 Lately, when polymer domains stiffen by cross-linking, the phase separation is arrested and the microstructures are stabilized. This process enhances light scattering and electro‐optic response of the liquid-crystalline emulsion offering new features suitable for many display applications \cite{doane1986field,kikuchi2002polymer}.
 
 \begin{figure}[htbp]
     \centering     \includegraphics[align=t,width=0.55\textwidth]{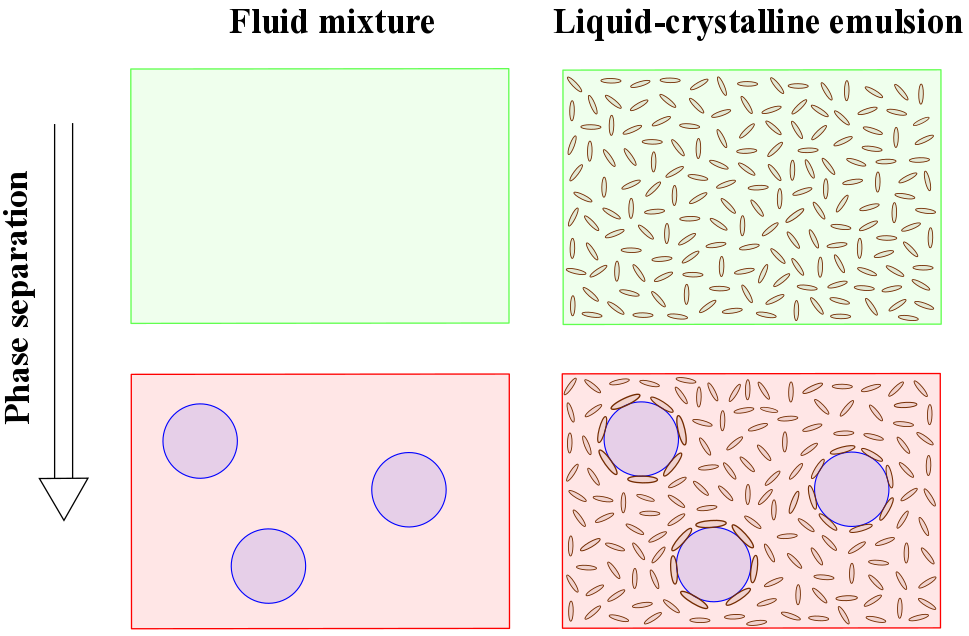}
     \caption{Phase separation in fluid mixtures vs liquid-crystalline emulsions.}
   \label{fig:PS}
 \end{figure}
 
 \subsection{Mathematical formulation}
We investigate a two-phase mixture of a polar liquid crystal and a Newtonian polymeric fluid. From the mathematical viewpoint, the above phenomenon is described through a Diffuse Interface model, which allows to treat the multi-scale interplay (visco-elastic response and interfacial dynamics) through an energy-based framework \cite{giga2017variational}. 
The order parameters of the system are the difference of the concentrations of the two phases $\phi \in [-1,1]$, which distinguishes the liquid-crystalline fluids $(\phi =+1)$ from the isotropic polymer $(\phi=-1)$, and the polarization of the liquid crystal $\di$, namely the average of the orientation of rod-like molecules in mesoscopic volume. The latter is expected to take value $|\di|=0$ in the isotropic polymer (no polarization), whereas $0\leq |\di|\leq 1$ in the liquid crystal phase. Two further state variables are the velocity of the fluid mixture $\uu$ and the pressure of the mixture $p$.

 The macroscopic dynamics of the state variables is derived from the combination of the variational energy formalism with the Least Action Principle and Maximum Dissipation Principle \cite{giga2017variational}. For liquid-crystalline emulsions, the derivation of thermodynamically consistent models has been performed in \cite{cates2018theories, yue2004diffuse}. The resulting equations of motion reads as follows
\begin{empheq}[left=\empheqlbrace]{align}
\label{NSE}
&\rho \partial_t \vect{u} + \rho (\vect{u} \cdot \nabla) \vect{u} - \diver \left(\nu(\phi)D\vect{u}\right) + \nabla p= \diver \sigma_{\rm el},\\
 \label{incompressibility}
 &\diver\,\vect{u} = 0,\\
 \label{var_phi}
 & \partial_t \phi + \vect{u} \cdot \nabla \phi = \Delta \mu,\\
 \label{var_d}
 & \partial_t \vect{d} + \left(\vect{u} \cdot \nabla\right)\vect{d} = -\vect{h},
\end{empheq}
where $\rho$ is the volume-averaged density of the mixture system and $\nu(\phi)$ is the concentration-depending viscosity of the mixture. In addition, the thermodynamic elastic stress $\sigma$, the chemical potential $\mu$ and the molecular field $\hh$ are given by
$$
\sigma_{\rm el}= \left( F(\phi,\di)- \phi \mu -\di \cdot \hh \right)\delta_{ij}-\nabla \phi \otimes \frac{\partial E_{\rm free}(\phi,\di)}{\partial \nabla \phi}
- \nabla \di \odot \frac{\partial E_{\rm free}(\phi,\di)}{\partial \nabla \di}, 
\quad 
\mu =\frac{\delta E_{\rm free}(\phi,\di)}{\delta \phi},
\quad
\hh = \frac{\delta E_{\rm free}(\phi,\di)}{\delta \vect{d}},
$$
where $\delta_{ij}$ is the Kronecker symbol, $F(\phi,\di)$ is the density of total free energy.
Moreover, $D \uu$ is the symmetric gradient $\frac12(\nabla \uu + (\nabla \uu)^T)$, the symbols $\vect{f} \otimes \vect{g}$, for $\vect{f},\vect{g} \in \mathbb{R}^n$, and $\tens{A} \odot \tens{B}$, for $\tens{A},\tens{B} \in \mathbb{R}^{n\times n}$, denote the $n\times n$ matrices whose $(i,j)$th entries are given by $\vect{f}_i \vect{g}_j$ and $(\tens{A}^T B)_{i,j}$, for $1\leq i,j\leq n$, respectively. In particular, if $\tens{A}=\nabla \vect{f}$ and $\tens{B}=\nabla \vect{g}$ for $\vect{f},\vect{g} \in \mathbb{R}^n$, then $(\nabla \vect{f}\odot \nabla \vect{g})_{i,j}= \partial_i \vect{f} \cdot \partial_j \vect{g}$.

The underlying complex rheology is described by three types of free energy: mixing energy of the interface, bulk distortion energy of the nematic phase, and the anchoring energy of the liquid crystal molecules on the interface. 
 Following the historical course, two descriptions of the total free energy have been proposed which differ from the penalization driving $|\di| \approx 0$ in the polymeric phase.  First, in the seminal work \cite{yue2004diffuse}, the authors introduced the following forms of the mixing and polarization free energies 
\begin{equation}
\label{E-mix}
E_{\rm mix}(\phi)=\int_{\Omega} 
\frac{\varepsilon}{2} \abs{\nabla \phi}^2 + \frac{1}{\varepsilon}\Psi(\phi) \, \d x,
\quad 
E^\star_{\rm pol}(\phi, \di) 
=\int_{\Omega} \frac12 (1+\phi) \left(
\frac{\kappa}{2} |\nabla \di|^2 + \frac{\kappa}{4\delta^2}
\left(|\di|^2-1\right)^2 \right)\, \d x,
\end{equation}
where the density of mixing is the Flory-Huggins potential for polymer mixtures given by
\begin{equation}
\label{pot-FH}
\Psi(s)= \textcolor{black}{F(s) - \frac{\Theta_0}{2} s^2} = \frac{\Theta}{2}\left( (1+s)\ln(1+s)+(1-s)\ln(1-s) \right) 
-\frac{\Theta_0}{2}s^2, \quad \forall \, s \in[-1,1],
\end{equation}
where $\Theta$ and $\Theta_0$ are positive parameters and as well as the anchoring energy
\begin{equation}
\label{E-anch}
E_{\rm anch}(\phi,\di)=
\int_{\Omega} \frac{\beta}{2} \abs{\nabla \phi \cdot \vect{d}}^2 \, \d x.
\end{equation}
Here, $\varepsilon$ is the interface parameter, $\kappa$ is the Frank distortion constant, $\delta$ is the relaxation parameter forcing $|\di| \approx 1$ and $\beta$ is related to the magnitude of the anchoring. The latter contribution $E_{\rm anch}$ describes the genuine tendency of liquid crystal molecules to align parallel to the interface (planar anchoring).  
The total free energy is then postulated to be as 
\begin{equation}
\label{eq:energia:Liu}
    E^\star_{\rm free}(\phi, \vect{d}) = E_{\rm mix}(\phi)
+ E^\star_{\rm pol}(\phi, \di)
+E_{\rm anch}(\phi,\di).
\end{equation}
Despite the great influence of the work \cite{yue2004diffuse} in the description of complex fluids, to the best of our knowledge, no theoretical results have been established concerning either the well-posedness of \eqref{NSE}-\eqref{var_d} with the free energy given by \eqref{eq:energia:Liu} nor the study of global minima of \eqref{eq:energia:Liu}. This is due to the degeneracy in the free energy caused by the term $(1+\phi)$. On the other hand, numerical schemes and simulations have been investigated  in \cite{ shen2014decoupled} and \cite{zhao2016decoupled} (see also \cite{guillen2021fluid, guillen2016linear, guillen2020nematic} for a more complex system including vesicles membranes).

More recently, a different form of the polarization energy has been proposed in \cite{cates2018theories}. This reads as
\begin{equation}
\label{E-pol}
E_{\rm pol}(\phi,\di) 
=\int_{\Omega} 
\frac{\kappa}{2}\abs{\nabla \vect{d}}^2
+ \frac{\alpha}{4} \abs{\vect{d}}^4
-\frac{\alpha}{2}(\phi-\phi_{\rm cr}) \abs{\vect{d}}^2 \, \d x,
\end{equation}
where $\phi_{\rm cr}$ is a constant such that $-1<\phi_{\rm cr}<1$. The corresponding total free energy is given by
\begin{equation}
\label{eq:energia}
 E_{\rm free}(\phi, \vect{d}) =  E_{\rm mix}(\phi)
+E_{\rm pol}(\phi,\di)
+E_{\rm anch}(\phi,\di),
\end{equation}
where $E_{\rm mix}(\phi)$ and $E_{\rm anch}(\phi,\di)$ are defined in \eqref{E-mix} and \eqref{E-anch}, respectively.
The selection mechanism of the minima in \eqref{E-pol} differs from the previous one: the coupling with the concentration parameter by the term $\phi-\phi_{\rm cr}$ promote $|\di|=0$ in the subdomain where $\phi<\phi_{\rm cr}$, whereas $|\di| \approx 1$ in the subdomains characterized by $\phi>\phi_{\rm cr}$.

In this paper, we consider the total free energy defined in \eqref{eq:energia}. The corresponding dynamics of the liquid-crystalline emulsion obtained from \eqref{NSE}-\eqref{var_d} reads as
\begin{empheq}[left=\empheqlbrace]{align}
\label{eq:1}
& \partial_t \vect{u} + (\vect{u} \cdot \nabla) \vect{u} - \diver \left(\nu(\phi)D\vect{u}\right) + \nabla p^\star =
- \varepsilon \diver\left(\nabla \phi \otimes \nabla \phi\right) -\kappa \diver \left(\nabla \vect{d} \odot \nabla \vect{d}\right) - \beta \diver \left((\vect{d}\cdot \nabla \phi) \nabla \phi\otimes \vect{d} \right),\\
 \label{eq:2}
 &\diver\,\vect{u} = 0,\\
 \label{eq:3}
 & \partial_t \phi + \vect{u} \cdot \nabla \phi = \Delta \mu,\\
\label{eq:4}
	&\mu = -\varepsilon \Delta \phi + \frac{1}{\varepsilon}\Psi'(\phi) -\frac{\alpha}{2} \abs{\vect{d}}^2- \beta\diver\left( (\nabla \phi \cdot \vect{d}) \vect{d}\right),\\
 \label{eq:5}
 & \partial_t \vect{d} + \left(\vect{u} \cdot \nabla\right)\vect{d} = -\vect{h},\\
 \label{eq:6}
	& \vect{h} = -\kappa \Delta \vect{d} +\alpha \abs{\vect{d}}^2 \vect{d} - \alpha (\phi-\phi_{\rm cr})\vect{d} + \beta \left(\vect{d}\cdot \nabla \phi\right)\nabla \phi.
\end{empheq}
Here, we have set $\rho=1$ and the redefined pressure is  
$$
p^\star = p- \left(\frac{\varepsilon}{2} \abs{\nabla \phi}^2 + \frac{1}{\varepsilon}\Psi(\phi) +\frac{\kappa}{2}\abs{\nabla \vect{d}}^2
+ \frac{\alpha}{4} \abs{\vect{d}}^4
-\frac{\alpha}{2}(\phi-\phi_{\rm cr}) \abs{\vect{d}}^2 + \frac{\beta}{2}|\nabla \phi \cdot \di|^2\right) 
+\phi \mu 
+ \di \cdot \hh.
$$
The total energy associated with \eqref{eq:1}-\eqref{eq:6} is
\begin{equation}
\label{Etot-Ekin}
E_{\rm tot}(\uu,\phi,\di)= E_{\rm kin}(\uu)+E_{\rm free}(\phi,\di), \quad \text{where} \quad E_{\rm kin}(\uu)=\int_\Omega \frac12 |\uu|^2 \,\d x.
\end{equation}
For simplicity, we will study the system \eqref{eq:1}-\eqref{eq:6} subject to periodic boundary conditions on $\mathbb{T}^n=(-\pi,\pi)^n$, with $n=2,3$.

The system \eqref{eq:1}-\eqref{eq:6} might be regarded as a combination of two famous models: the Ericksen-Leslie system \cite{ lin2000existence, lin1995nonparabolic, wu2013general} for nematic liquid crystals and the Model H for binary fluids \cite{abels2009diffuse, giorgini2019uniqueness} in presence of surface tension effects. However, due to the novel couplings  resulting from the anchoring of the liquid crystal molecules at the interface, \eqref{eq:1}-\eqref{eq:6} is highly non-linear than the two above-mentioned separated models. For instance, while the classical convective Cahn-Hilliard equation is a semilinear parabolic equation (\eqref{eq:3}-\eqref{eq:4} with $\alpha=\beta=0$), the modified Cahn-Hilliard system in \eqref{eq:3}-\eqref{eq:4} turns into a quasi-linear parabolic equation. Indeed, the nonlinear term $\diver((\nabla\phi \cdot \di) \di)$ has same order of derivatives of the diffusive term $\varepsilon \Delta \phi$. To the best of our knowledge, no theoretical results have been discussed yet regarding  system \eqref{eq:1}-\eqref{eq:6}.  On the contrary, numerical simulations have been presented in \cite{chen2020numerical, sui2021second, zhao2016numerical}.

Notwithstanding its complicated nonlinear nature, it is worth mentioning that the system \eqref{eq:1}-\eqref{eq:6} is a simplified version of the model resulting from \cite{cates2018theories}. For example, the motion of the polarization \eqref{eq:5} usually takes into account rigid rotation of molecules and stretching of molecules by the flow, which are modeled by $-\Omega \cdot \di$, where $\Omega$ is the anti-symmetric part of $\nabla \uu$, and $D \uu \cdot \di$, respectively. We also mention that  the polar order $\di$ can be replaced by a traceless second rank tensor $\tens{Q}$ following the Beris-Edwards theory for nematic liquid crystals (see \cite[Sections 9.1 and 10.2]{cates2018theories}).
Also, an alternative of the planar anchoring energy \eqref{E-anch} is the homeotropic anchoring promoting different alignment at the interface then tangential direction (see \cite{yue2004diffuse}), which is defined by
$$
E^{\rm H}_{\rm anch}(\phi,\di)=
\int_{\Omega} \frac{\beta}{2} \left( |\di|^2 |\nabla \phi|^2 - \abs{\nabla \phi \cdot \vect{d}}^2 \right) \, \d x.
$$
Moreover, it would be interesting to generalize the system \eqref{eq:1}-\eqref{eq:6} for the case of liquid-crystalline emulsions with unmatched densities (i.e. $\rho_1\neq \rho_2$, thereby $\rho\neq 1$) in the spirit of \cite{abels2012thermodynamically}.
Finally, the study of system \eqref{eq:1}-\eqref{eq:6} is a first step towards more complex models describing active liquid-crystalline emulsions in biological systems (e.g. rod-like bacteria and rod-like self-propelled colloids) proposed in \cite[Section 12]{cates2018theories}. For instance, a primary example is motility-induced phase separation: an active version of \eqref{eq:1}-\eqref{eq:6} has shown great potential to describe cell crawling dynamics based on F-actin polymerization and actomyosin contractility \cite{tjhung2015minimal}. 

\subsection{Main result}
\label{sub:mr}

Let us first state our main assumptions:
\begin{itemize}
    \item[(A.1)] The domain is $\Omega=(-\pi,\pi)^n$ for $n=2,3$. 
    
    \item[(A.2)] The parameters $\alpha$, $\beta$, $\varepsilon$, $\kappa$ are positive constants. The relaxation parameter $\phi_{\rm cr}\in (-1,1)$ and the density $\rho=1$. The two parameters related to the temperature $\Theta$ and the critical temperature $\Theta_0$ satisfy the conditions $0\leq \Theta <\Theta_0$.

    \item[(A.3)] The viscosity function $\nu \in W^{1,\infty}(\mathbb{R})$ such that 
$0<\nu_\ast \leq\nu(s)\leq\nu^\ast$
for all $s\in\mathbb{R}$.

    \item[(A.4)] The homogeneous density of mixing $\Psi(\phi)$ is given by \eqref{pot-FH}.

    \item[(A.5)] The spatial average of $\uu_0$ is null, i.e. $\overline{\uu_0}=\mathbf{0}$. 
\end{itemize}

Before proceeding with the main result, 
let us consider the function $g: [-1,1]\times (0,\infty)$ defined by
\begin{equation}
    \label{eq:g}
    g(s, w):= \frac{1}{\varepsilon} \left(\frac{\Theta}{2}\left( (1+s)\ln(1+s)+(1-s)\ln(1-s) \right) 
-\frac{\Theta_0}{2}s^2\right) - \frac{\alpha}{2} (s- \phi_{\rm cr}) w^2 + \frac{\alpha}{4} w^4.
\end{equation}
Here, $w$ must be interpreted as $\abs{\di}$.
We notice that $g(s,w)\geq E_0$, where the constant $E_0$ might be negative, and $g$ attains the value $E_0$ in at least one point for any choice of the {\color{black} parameters satisfying ${\rm (A.2)}$}. Then, it follows that $E_{\rm tot}(\uu,\phi,\di)- (2\pi)^n E_0\geq 0$ for any $(\uu,\phi,\di) \in D(E_{\rm tot})=\lbrace (\vv, \psi, \vect{p}) \in \dot{L}^2_\sigma(\Omega)\times H^1(\Omega)\times H^1(\Omega): 
\| \psi\|_{L^\infty(\Omega)}\leq 1\rbrace$.
In general, the structure and the cardinality of local and/or global minima are related to the choice of the parameters (see Figure \ref{fig:minimi}). 
\begin{figure}[htbp]
		\centering
		\includegraphics[width=0.95\textwidth]{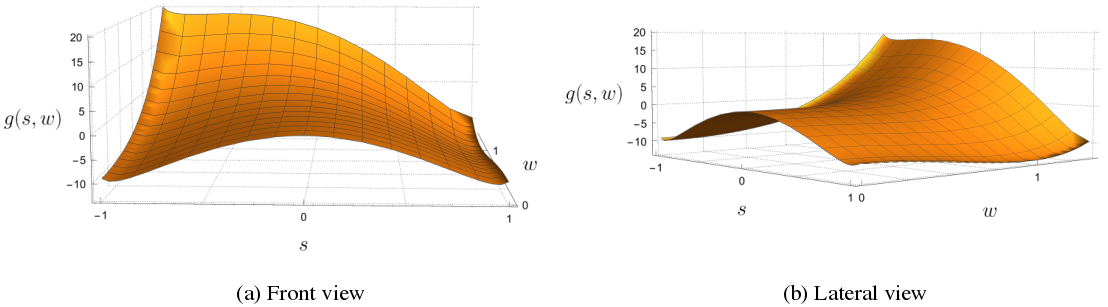}
	\caption{Plot of the function $g(s,w)$ defined in \eqref{eq:g} with parameters $\varepsilon = 0.05, \alpha = 15, \Theta = 1.5, \Theta_0 = 3$ and $\phi_{\rm cr} = 0$: (a) front view and (b) lateral view. In this case, there exists a global minimum $\left(\overline{s},\sqrt{\overline{s}}\right)$, where $\frac{\Theta}{2\varepsilon}\ln(\frac{1+\overline{s}}{1-\overline{s}})= \left( \frac{\alpha}{2}+\frac{\Theta_0}{\varepsilon}\right)\overline{s}$, whereas $(-s_{\rm CH},0)$ is a local minimum, where $\pm s_{\rm CH}$ are the minima of $g(s,0)=\Psi(s)$ (cf. \eqref{pot-FH}).}
 \label{fig:minimi}
\end{figure}

We now present the main result of this work. We refer the reader to Section \ref{Notation} for the notation.

\begin{theorem}
\label{MR}
Let the assumptions $\rm{(A.1)-(A.5)}$ hold. Suppose that $\uu_0 \in \dot{L}^2_\sigma(\Omega)$, $\phi_0 \in H^1(\Omega)\cap L^\infty(\Omega)$ with {\color{black} $\| \phi_0\|_{L^\infty(\Omega)}\leq 1$} and $|\overline{\phi_0}|<1$, $\di_0 \in H^1(\Omega)\cap L^\infty(\Omega)$, and 
\begin{equation}
\label{Key-ass}
\min \lbrace \varepsilon, \kappa \rbrace
> 3^\frac34\beta C_\Omega^2 \mathcal{D}_\infty^\frac32  \quad \text{if } \, n=2,3, \quad \text{or} \quad
\min \lbrace \varepsilon, \kappa \rbrace
> 3^\frac34\beta \widetilde{C_\Omega}^2 \mathcal{D}_\infty 
        \frac{\left(E_{\rm tot}(\uu_0, \phi_0, \di_0)-(2\pi)^n E_0 \right)^\frac12}{\varepsilon^\frac14\kappa^\frac14} 
        \quad \text{if } \, n=2,
\end{equation}
where $C_\Omega$ and $\widetilde{C_\Omega}$ are defined in \eqref{GN} and \eqref{LADY}, respectively, and $\mathcal{D}_\infty= \max \left\lbrace \| \di_0\|_{L^\infty(\Omega)}, \sqrt{1-\phi_{\rm cr}} \right\rbrace$.
Then, there exists a global weak solution $(\uu, \phi, \di, \mu, \vect{h})$ such that:
\begin{itemize}
    \item The following regularity hold
    \begin{equation}
    \label{W:regularity}
    \begin{split}
        & \uu \in L^\infty(0,\infty;\dot{L}_\sigma^2(\Omega))\cap L^2(0,\infty;\dot{H}^1_\sigma(\Omega)), \quad
        \partial_t \uu \in L_{\rm uloc}^\frac{4}{d}([0,\infty); \dot{H}^1_\sigma(\Omega)'),\\
        & \phi \in L^\infty(0,\infty; H^1(\Omega)) \cap L_{\rm uloc}^2([0,\infty);H^2(\Omega)),
        \quad 
        \partial_t \phi \in L^2(0,\infty; H^1(\Omega)'),\\
        & \phi \in L^\infty(\Omega \times (0,\infty)) 
        \text{ such that } |\phi(x,t)|<1 \ \text{a.e. in } \Omega \times (0,\infty),\\
        &\Psi'(\phi)\in L_{\rm uloc}^2([0,\infty); L^p(\Omega)), \text{ where } {\color{black} p=6 \text{ if } n=3, \text{ any } \, 2 \leq p <\infty \text{ if } n=2},\\
        &\mu \in L_{\rm uloc}^2([0,\infty);H^1(\Omega)), \quad \nabla \mu \in L^2(\Omega \times (0,\infty)),\\
        &\di \in L^\infty(0,\infty; H^1(\Omega)) \cap L_{\rm uloc}^2([0,\infty);H^2(\Omega)),\\
        &\di \in  L^\infty(\Omega \times (0,\infty))
        \text{ such that } |\di(x,t)|\leq \mathcal{D}_\infty \ \text{a.e. in } \Omega \times (0,\infty),\\
        &\vect{h} \in L^2(0,\infty;L^2(\Omega)).
    \end{split}
    \end{equation}
    
    \item The solution $(\uu, \phi, \di, \mu, \vect{h})$ {\color{black} solves \eqref{eq:1}-\eqref{eq:6}} in the following weak sense:
    \begin{align}
        \label{W:weak-s:u}
        &
        \langle \partial_t \uu, \vv \rangle_{\dot{H}^1_\sigma(\Omega)}
        + \left( (\uu\cdot \nabla) \uu, \vv \right)
        + (\nu(\phi)D \uu, D\vv)
        = -(\phi \nabla \mu, \vv)
        + ( (\nabla \di)^T \hh ,\vv),
        \\
        \label{W:weak-s:phi}
        &
        \langle \partial_t \phi, v \rangle_{H^1(\Omega)}
        + \left( \uu\cdot\nabla \phi, v\right) + (\nabla \mu,\nabla v)=0,
    \end{align}
    for all $\vv \in \dot{H}^1_{\sigma}(\Omega)$, 
    $v \in H^1(\Omega)$ and almost everywhere in $(0,\infty)$, as well as
        \begin{align}
        \label{W:weak-s:mu}
        &\mu=
        -\varepsilon \Delta \phi
        + \frac{1}{\varepsilon}
        \Psi'(\phi)
        - \frac{\alpha}{2} |\di|^2
        - \beta\diver \left( \left( \nabla \phi \cdot \di \right) \di \right),
        \quad \text{a.e. in } \Omega \times (0,\infty),
        \\
        \label{W:weak-s:d}
         &\partial_t \di + \left( \uu\cdot \nabla \right) \di = \hh, 
        \quad 
        \hh = \kappa \Delta \di -\alpha \abs{\di}^2 \di + \alpha (\phi-\phi_{\rm cr})\di -\beta (\di \cdot \nabla \phi)\nabla \phi 
        \quad \text{a.e. in } \Omega \times (0,\infty).
    \end{align}
    
    \item The solution satisfies $\uu(\cdot, 0)=\uu_0(\cdot)$, $\phi(\cdot, 0)=\phi_0(\cdot)$, $\di(\cdot,0)=\di_0(\cdot)$ in $\Omega$.
    
    \item The energy inequality
    \begin{equation}
        \label{MR:energy:in}
        \begin{split}
        &E_{\rm tot}(\uu(t), \phi(t), \di(t))+
        \int_s^t \left\| \sqrt{\nu(\phi(\tau))} D\uu(\tau)\right\|_{L^2(\Omega)}^2 
        +\| \nabla \mu(\tau)\|_{L^2(\Omega)}^2
        +\| \vect{h} (\tau) \|_{L^2(\Omega)}^2 
        \, \d \tau 
        \\
        &\quad \leq E_{\rm tot}(\uu(s), \phi(s), \di(s))
        \end{split}
    \end{equation}
holds for all $t \in [s,\infty)$ and almost all $s \in [0,\infty)$, including $s=0$. If $n=2$, \eqref{MR:energy:in} holds as equality. 
\end{itemize}
Furthermore, we have the following additional results: 
\begin{itemize}
    \item If $n=2$, the weak solutions are unique. 

    \item If $n=2$, assume that $\uu_0 \in \dot{H}^1_\sigma(\Omega)$, $\phi_0 \in H^2(\Omega)$ such that {\color{black} $|\overline{\phi_0}|< 1$ } and 
    $-\varepsilon\Delta \phi_0 
    +\frac{1}{\varepsilon}\Psi'(\phi_0)-\beta \diver\left((\nabla \phi_0\cdot \di_0)\di_0\right) \in H^1(\Omega)$,
    and $\di_0 \in H^2(\Omega)$. Then, the weak solution satisfies the additional regularity
   \begin{equation}
    \label{SS:regularity}
    \begin{split}
        &\uu \in L^\infty(0,\infty;\dot{H}^1_\sigma(\Omega))\cap L_{\rm uloc}^2([0,\infty); \dot{H}^2_\sigma(\Omega)), \quad
        \partial_t \uu \in L_{\rm uloc}^2([0,\infty); \dot{L}^2_\sigma(\Omega)),\\
        &\phi \in L^\infty(0,\infty; H^2(\Omega)),
        \quad 
        \partial_t \phi \in L^2(0,\infty; H^1(\Omega)),\quad
        \Psi'(\phi)\in L^\infty(0,\infty; L^p(\Omega)), \\
        &\mu \in L^\infty(0,\infty;H^1(\Omega)),\\
        &\di \in L^\infty(0,\infty; H^2(\Omega)) \cap L_{\rm uloc}^2([0,\infty);W^{2,p}(\Omega)), \quad \partial_t \di \in L^\infty(0,\infty;L^2(\Omega))\cap L^2_{\rm uloc}([0,\infty); H^1(\Omega)),\\
        &\vect{h} \in L^\infty(0,\infty;L^2(\Omega))
        \cap L^2_{\rm uloc}([0,\infty); H^1(\Omega)),
    \end{split}
    \end{equation}
    for any $2 \leq p <\infty$.
\end{itemize}
\end{theorem}

\begin{remark}
Since the total mass of the velocity is conserved by \eqref{eq:1}, namely $\overline{\uu}(t)=\overline{\uu_0}$ for any $t\geq 0$, the conclusion of Theorem \ref{MR} remains true if $\overline{\uu_0}\neq \mathbf{0}$ by replacing $\uu$ with $\uu-\overline{\uu_0}$.
\end{remark}

Theorem \eqref{MR} is the first theoretical result in the literature addressing the dynamics of liquid-crystalline emulsions described by \eqref{eq:1}-\eqref{eq:6}. This ensures that {\it weak} solutions (in the above sense) exists globally in time under a criteria on the parameters of the system (or on the initial total energy if $n=2$) in both two and three dimensions. 
As already highlighted, in addition to the difficulties inherited from the Ericksen-Leslie system and the Model H  as in \cite{abels2009diffuse,giorgini2019uniqueness, lin2000existence, lin1995nonparabolic}, the fundamental mathematical challenges in the analysis of \eqref{eq:1}-\eqref{eq:6} originate from the anchoring mechanism which introduces a quasi-linear coupling term in the chemical potential \eqref{eq:4} of the Cahn-Hilliard equation and a further coupling term in \eqref{eq:5}-\eqref{eq:6} that is quadratic in $\nabla \phi$.
For this reason, aiming to show the existence of weak solutions via compactness method relying on  approximating problems, the global (uniform) estimates obtained from the energy equation (cf. \eqref{MR:energy:in}) are not sufficient (cf., e.g., the term $(\di \cdot \nabla \phi)\nabla \phi$). To overcome this issue, we aim to deduce (uniform) Sobolev bounds on $\phi$ and $\di$ from the $L^2$-bounds of $\nabla \mu$ and $\hh$ given by the energy equation (cf. \eqref{MR:energy:in}). To this end, we first show the global $L^\infty$-bound of $\di$ (cf. \eqref{Ex:d:Linf}), and then an estimate of $\phi$ and $\di$ in $L_{\rm uloc}^2([0,\infty);H^2(\Omega))$. The latter arises from a combined estimate exploiting the equations of $\mu$ and $\hh$, which allows to rewrite the sum of the unsigned coupling terms as a positive term and a remainder to be handled (see \eqref{E:g_z:3}-\eqref{E:g_z:8}, in particular \eqref{crucial}). 
In this step, the main assumption \eqref{Key-ass} comes into play to control the remainder term. Roughly speaking, \eqref{Key-ass} requests that either the magnitude of the anchoring is small compared to $\varepsilon$ and $\kappa$ or the initial datum has small initial total energy compared to the other parameters of the system. In order to make the above procedure rigorous, we construct approximated solutions to a regularized model with artificial viscosity and extra gradient integrability in the Cahn-Hilliard equation via nested fixed-point theorems and the semi-Galerkin scheme for the velocity. Next, in the same framework of parameters (see \eqref{Key-ass}), we demonstrate that weak solutions are unique in two dimensions by non-trivially extending the technique devised in \cite{giorgini2019uniqueness}. Lastly, we show that {\it strong} solutions exist globally in two dimensions under additional regularity of the initial data. It is important to point out that no more restrictive assumptions on the parameters than \eqref{Key-ass} are requested. Also in this step, positive terms are recovered from the anchoring couplings by combining two estimates 
for the Cahn-Hilliard and the Ericksen-Leslie models up to lower order remainder terms.
For the sake of brevity, we perform the {a priori} estimate omitting the approximation procedure. We conclude by mentioning that the last result can be easily extended to the three dimensional case on local intervals.

\section{Notation and function spaces}
\label{Notation}
Let $\Omega$ be the $n$-dimensional torus $\mathbb{T}^n=(-\pi, \pi)^n$ with $n = 2,3$. 
Any real-valued functions $f$ defined on $\Omega$ can be written as
$$
f(x) = \sum_{k \in \mathbb{Z}^{n}} \hat{f}_k e^{i k \cdot x}, 
\quad \text{where}
\quad 
\hat{f}_k = \frac{1}{\left(2 \pi\right)^n} \int_{\mathbb{T}^n} e^{-ik\cdot x} f(x)\, \d x
\quad \text{such that}
\quad 
\hat{f}_k = \overline{\hat{f}_k}.
$$
If $\vect{f}$ is an $m$-component
vector-valued function the coefficients $\hat{\vect{f}}_k$ are elements of $\mathbb{C}^m$ such that $\hat{\vect{f}}_k = \overline{\hat{\vect{f}}_k}$. We recall that 
$\{e^{ik \cdot x}\}$ forms an orthogonal basis in $L^2(\Omega)$. For any $s \in \mathbb{N}_0$ and $p \in [1,\infty]$, the symbol $W^{s,p}(\Omega)$ denotes the Sobolev space of functions whose generalized derivative of order $s$ is $p$-integrable. For simplicity of notation,  we will use the notation $W^{s,p}(\Omega)$ also for vector-valued functions with value in $\mathbb{R}^m$ with $m\in \mathbb{N}$. In both cases, we indicate the norms by $\| \cdot \|_{W^{s,p}(\Omega)}$. In particular, we have
$$
\norm{f}_{L^2(\Omega)}^2 = \left(2 \pi\right)^n \sum_{k \in \mathbb{Z}^n}\abs{\hat{f}_k}^2, 
\quad
\norm{f}_{H^1(\Omega)}^2 = \left(2 \pi\right)^n \sum_{k \in \mathbb{Z}^n}\left(1 + \abs{k}^{2 }\right)\abs{\hat{f}_k}^2, 
\quad 
\norm{f}_{H^2(\Omega)}^2 = \left(2 \pi\right)^n \sum_{k \in \mathbb{Z}^n}\left(1 + \abs{k}^{4}\right)\abs{\hat{f}_k}^2.
$$
We also introduce the solenoidal spaces 
$$
\dot{L}^2_\sigma(\Omega)
= \lbrace \vect{f} \in L^2(\Omega): \diver \vect{f}=0, \ \hat{\vect{f}}_0=0\rbrace,
\quad 
\dot{H}^1_\sigma(\Omega)=H^1(\Omega)\cap \dot{L}^2_\sigma(\Omega),
\quad 
\dot{H}^2_\sigma(\Omega)=H^2(\Omega)\cap \dot{L}^2_\sigma(\Omega).
$$
In addition, setting $\overline{f}=\hat{f}_0$ (the spatial average), we recall the useful inequalities
\begin{align}
\label{H1:equiv}
&\| f\|_{H^1(\Omega)}
\leq \left( 2\pi\right)^\frac{n}{2}
\left| \overline{f} \right| + 
\sqrt{2} \| \nabla f\|_{L^2(\Omega)}, \quad \forall \, f \in H^1(\Omega),
\\
\label{H2:equiv}
&\| f\|_{H^2(\Omega)}
\leq \left( 2\pi\right)^\frac{n}{2}
\left| \overline{f} \right| + 
\sqrt{2} \| \Delta f\|_{L^2(\Omega)}, \quad \forall \, f \in H^2(\Omega).
\end{align}
We will also make use of the following interpolation (Gagliardo-Nirenberg and Ladyzhenskaya) inequalities
\begin{align}
\label{GN}
    &\| f\|_{W^{1,4}(\Omega)}
    \leq C_\Omega \| f\|_{L^\infty(\Omega)}^\frac12 
    \| f\|_{H^2(\Omega)}^\frac12, \quad \forall \, f \in H^2(\Omega), \, n=2,3,
    \\
\label{LADY}  
    &\| f\|_{L^4(\Omega)}
    \leq \widetilde{C_\Omega} \| f\|_{L^2(\Omega)}^\frac12 
    \| f\|_{H^1(\Omega)}^\frac12, \quad \forall \, f \in H^1(\Omega), \, n=2.
\end{align}

Let $I \subset [0,\infty)$ be a measurable set and let $X$ be a Banach space. The space $L^q(I;X)$ denotes the set of all strongly measurable $q$-integrable functions with values in the Banach space $X$.  
If $I=[0,\infty)$, we also define the space $L^q_{\rm uloc}([0,\infty);X)$ as the set of all strongly measurable functions $f:[0,\infty)\to X$ such that
$$
\norm{f}_{L_{\rm uloc}^q([0, \infty);X)} = \sup_{t \geq 0} \norm{f}_{L^q(t,t+1;X)} < \infty.
$$
The space $BC(I;X)$ is the Banach space of all bounded and continuous functions $f: I \to X$ equipped with the supremum norm. 
The space $BC_{\rm w}(I;X)$ is the topological vector space of all bounded and weakly continuous functions $f:I\to X$.

\section{Global existence of weak solutions}
\label{sec:exist}

This section is devoted to the proof of the existence of global weak solutions to system \eqref{eq:1}-\eqref{eq:6}. We first study a regularized problem depending on a parameter $\gamma$, then we carry out suitable estimates which are uniform with respect to $\gamma$ allowing the final passage to the limit as $\gamma \to 0$ by compactness methods.

\subsection{Global existence of weak solutions to a regularized problem}
\label{subsec:gamma}

First of all, for $\gamma>0$, we consider the following regularization of the two-phase complex fluids system \eqref{eq:1}-\eqref{eq:6}
\begin{empheq}[left=\empheqlbrace]{align}
\label{eq:NSE}
& \partial_t \vect{u} + (\vect{u} \cdot \nabla) \vect{u} - \diver \left(\nu(\phi)D\vect{u}\right) + \nabla p =
\nonumber\\
& \qquad 
- \varepsilon \diver\left(\nabla \phi \otimes \nabla \phi\right) - \gamma \diver \left(\abs{\nabla \phi}^2 \nabla \phi \otimes \nabla \phi\right)-\kappa \diver \left(\nabla \vect{d} \odot \nabla \vect{d}\right) - \beta \diver \left((\vect{d}\cdot \nabla \phi) \nabla \phi\otimes \vect{d} \right),\\
 \label{eq:incompressibility}
 &\diver\,\vect{u} = 0,\\
 \label{eq:var_phi}
 & \partial_t \phi + \vect{u} \cdot \nabla \phi = \Delta \mu,\\
\label{eq:def_mu}
	&\mu = -\diver\left((\varepsilon + \gamma \abs{\nabla \phi}^2) \nabla \phi\right) + \frac{1}{\varepsilon}\Psi'(\phi) -\frac{\alpha}{2} \abs{\vect{d}}^2- \beta \diver\left( (\nabla \phi \cdot \vect{d}) \vect{d}\right),\\
 \label{eq:var_d}
 & \partial_t \vect{d} + \left(\vect{u} \cdot \nabla\right)\vect{d} = -\vect{h},\\
 \label{eq:def_h}
	& \vect{h} = -\kappa \Delta \vect{d} +\alpha \abs{\vect{d}}^2 \vect{d} - \alpha (\phi-\phi_{\rm cr})\vect{d} + \beta \left(\vect{d}\cdot \nabla \phi\right)\nabla \phi,
\end{empheq}
in $\Omega \times (0,\infty)$, which is subject to periodic boundary conditions on the state variables $(\uu, \phi, \mu, \di, \hh)$.
The total energy of the system \eqref{eq:NSE}-\eqref{eq:def_h} is given by
\begin{equation}
\label{eq:energy}
	E_{\rm tot}^\gamma = E_{\rm tot}^\gamma(\vect{u}, \phi, \vect{d}) = E_{\rm kin}(\vect{u}) + E_{\rm free}^\gamma(\phi, \vect{d}),
\end{equation}
where
\begin{align}
\label{eq:energia_free}
&E_{\rm free}^\gamma(\phi, \vect{d}) := \int_{\Omega} \frac{\gamma}{4}\abs{\nabla \phi}^4 +\frac{\varepsilon}{2} \abs{\nabla \phi}^2 + \frac{1}{\varepsilon}\Psi(\phi) 
+ \frac{\kappa}{2}\abs{\nabla \vect{d}}^2 + \frac{\alpha}{4} \abs{\vect{d}}^4  -\frac{\alpha}{2}(\phi-\phi_{\rm cr}) \abs{\vect{d}}^2 + \frac{\beta}{2} \abs{\nabla \phi \cdot \vect{d}}^2\, \d x.
\end{align}
By direct computation, we observe that
\begin{equation}
\label{USEFUL}
\begin{split}
&- \varepsilon \diver\left(\nabla \phi \otimes \nabla \phi\right)=
-\varepsilon \nabla \left( \frac12 |\nabla \phi|^2 \right) -\varepsilon \Delta \phi \nabla \phi, \\
&- \gamma \diver \left(\abs{\nabla \phi}^2 \nabla \phi \otimes \nabla \phi\right)=
-\gamma \nabla \left( \frac14 |\nabla \phi|^4 \right)
-\gamma \diver \left( |\nabla \phi|^2 \nabla \phi \right),\\
&-\kappa \diver \left(\nabla \vect{d} \odot \nabla \vect{d}^T\right) 
= -\kappa \nabla \left( \frac12 |\nabla \di|^2 \right)- \kappa (\nabla \di)^T \Delta \di, \\
&- \beta \diver \left((\vect{d}\cdot \nabla \phi) \nabla \phi\otimes \vect{d} \right)= -\beta \diver\left( (\di \cdot \nabla \phi) \di\right) \nabla \phi - \beta \left(\di \cdot \nabla \phi\right) \left( \di \cdot \nabla \right) \nabla \phi. 
\end{split}
\end{equation}
In light of \eqref{eq:def_mu} and \eqref{eq:def_h}, we then find
\begin{equation}
\label{RHS-NS}
\begin{split}
&- \varepsilon \diver\left(\nabla \phi \otimes \nabla \phi\right) - \gamma \diver \left(\abs{\nabla \phi}^2 \nabla \phi \otimes \nabla \phi\right)-\kappa \diver \left(\nabla \vect{d} \odot \nabla \vect{d}\right) - \beta \diver \left((\vect{d}\cdot \nabla \phi) \nabla \phi\otimes \vect{d} \right)\\
&= \nabla \left( -\frac{\varepsilon}{2} |\nabla \phi|^2 -\frac{\gamma}{4} |\nabla \phi|^4 -\frac{\kappa}{2} |\nabla \di|^2 \right) + \mu \nabla \phi 
-\frac{1}{\varepsilon} \Psi'(\phi) \nabla \phi + \frac{\alpha}{2} |\di|^2 \nabla \phi \\
&\quad +(\nabla \di)^T \left(\vect{h} -\alpha |\di|^2\di +\alpha (\phi -\phi_{cr}) \di - \beta (\di \cdot \nabla \phi) \nabla \phi\right) 
- \beta \left(\di \cdot \nabla \phi\right) \left( \di \cdot \nabla \right) \nabla \phi \\
&= \nabla \left(\underbrace{ -\frac{\varepsilon}{2} |\nabla \phi|^2 -\frac{\gamma}{4} |\nabla \phi|^4 -\frac{1}{\varepsilon} \Psi(\phi) -\frac{\kappa}{2} |\nabla \di|^2- \frac{\alpha}{4}|\di|^4+ \frac{\alpha}{2}(\phi-\phi_{cr}) |\di|^2 -\frac{\beta}{2} |\di \cdot \nabla \phi|^2}_{\text{density of  } -E^\gamma_{\rm free}(\phi,\vect{d})} \right) + \mu \nabla \phi + (\nabla \di)^T \vect{h}.
\end{split}
\end{equation}
Thus, we can rewrite \eqref{eq:NSE} as follows
\begin{equation}
    \label{eq:identity_RHS_NSE}
     \partial_t \vect{u} +(\vect{u} \cdot \nabla) \vect{u} - \diver \left(\nu(\phi)D\vect{u}\right) + \nabla p^\ast =
    \mu \nabla \phi +(\nabla \di)^T \, \vect{h},
\end{equation}
where the generalized pressure $p^\ast= p^\star + \frac{\gamma}{4}|\nabla \phi|^4$. 
We now formally derive the energy equation for \eqref{eq:NSE}-\eqref{eq:def_h}. The same argument remains true for the system \eqref{eq:1}-\eqref{eq:6} by taking $\gamma=0$.
First,  multiplying \eqref{eq:identity_RHS_NSE} by $\uu$ and integrating over $\Omega$, we have
\begin{align}
\label{eq:firsteq_integrale}
\frac{1}{2} \ddt \int_\Omega \abs{\vect{u}}^2 \, \d x  + \int_\Omega \nu(\phi) \abs{D\vect{u}}^2\, \d x 
&= \int_\Omega \mu \nabla \phi\cdot \uu \, \d x +
\int_\Omega (\nabla \di)^T \hh \cdot \uu \, \d x.
\end{align}
Second, multiplying \eqref{eq:var_phi} by $\mu$, integrating over $\Omega$ and using the definition of $\mu$ in \eqref{eq:def_mu}, we obtain
\begin{align}
\label{eq:var_ph_integrated}
	&\ddt \left[\int_\Omega\left(\frac{\varepsilon}{2} \abs{\nabla \phi}^2 + \frac{1}{\varepsilon}\Psi(\phi) + \frac{\gamma}{4}\abs{\nabla \phi}^4 \right)\, \, \d x\right] + \int_\Omega \mu \vect{u} \cdot \nabla \phi \, \d x +\int_\Omega \abs{\nabla \mu}^2 \, \d x 
 \nonumber \\
	&\quad =\int_\Omega \partial_t \phi \frac{\alpha}{2} \abs{\vect{d}}^2\, \d x +\int_\Omega \beta\partial_t \phi \diver\left((\nabla \phi \cdot \vect{d})\vect{d}\right)\, \d x. \nonumber
\end{align}
Rewriting the terms on the right-hand side, the above equation simplifies into
\begin{equation}
\label{eq:second_eq_integrale}
\begin{split}
	&\dt \left[\int_\Omega\left(\frac{\varepsilon}{2} \abs{\nabla \phi}^2 + \frac{1}{\varepsilon}\Psi(\phi) + \frac{\gamma}{4} \abs{\nabla \phi}^4 - \frac{\alpha}{2} (\phi-\phi_{\rm cr}) \abs{\vect{d}}^2+ \frac{\beta}{2} \abs{\nabla\phi \cdot \vect{d}}^2\right)\, \d x\right]+ \int_\Omega \mu \vect{u} \cdot \nabla \phi \, \d x + \int_\Omega \abs{\nabla \mu}^2 \, \d x  \\
	&\quad = -\int_\Omega \frac{\alpha}{2}(\phi- \phi_{\rm cr}) \partial_t\abs{\vect{d}}^2\, \d x + \beta \int_\Omega (\nabla \phi \cdot \partial_t \vect{d}) (\nabla \phi \cdot \vect{d})\, \d x.
 \end{split}
\end{equation}
Finally, multiplying \eqref{eq:var_d} by $\vect{h}$ defined in \eqref{eq:def_h}, integrating over $\Omega$, we find
\begin{align}
	 \label{eq:third_eq_integrale}
	 &\dt \left[\int_\Omega\left(\frac{\kappa}{2} \abs{\nabla \vect{d}}^2 + \frac{\alpha}{2} \abs{\vect{d}}^4\right)\, \d x\right] 
	 +\int_\Omega \abs{\vect{h}}^2 \, \d x
	 \nonumber\\
	&\quad 
	- \int_\Omega \alpha(\phi-\phi_{\rm cr}) \vect{d}\cdot \partial_t\vect{d}\, \d x 
	 +\int_\Omega \beta (\vect{d}\cdot \nabla \phi) \, (\nabla \phi \cdot\partial_t \vect{d}) \, \d x - \int_\Omega\left(\vect{u} \cdot \nabla\right)\vect{d}\cdot \vect{h}\, \d x = 0.
\end{align}
Summing up \eqref{eq:firsteq_integrale}, \eqref{eq:second_eq_integrale} and \eqref{eq:third_eq_integrale}, we arrive at 
\begin{align}
\label{eq:energia_final}
	&\frac{{\rm d}}{{\rm d}t} E_{\rm tot}^\gamma(\vect{u}, \phi, \vect{d})+\int_\Omega \nu(\phi) \abs{D\vect{u}}^2\, \d x + \int_\Omega \abs{\nabla \mu}^2 \, \d x +\int_\Omega \abs{\vect{h}}^2 \, \d x=0.
\end{align}
Furthermore, we recall the mass conservation for the difference of fluid concentrations that reads
\begin{equation}
    \label{eq:cons_massa}
    |\overline{\phi(t)}| = \frac{1}{|\Omega|} \left|\int_\Omega \phi(x) \, \d x\right| = |\overline{\phi(0)}|,\quad \forall\, t \in [0, +\infty).
\end{equation}

In this section, we study the existence of weak solutions to the system \eqref{eq:NSE}-\eqref{eq:def_h}. We prove the following result
\begin{theorem}
\label{MR:gamma-p}
Let the assumptions $\rm{(A.1)-(A.5)}$ hold. Suppose that $\uu_0 \in \dot{L}^2_\sigma(\Omega)$, $\phi_0 \in W^{1,4}(\Omega)$ with {\color{black} $\| \phi_0\|_{L^\infty(\Omega)}\leq 1$} and $|\overline{\phi_0}|<1$, $\di_0 \in H^1(\Omega)\cap L^\infty(\Omega)$.
Then, there exists a global weak solution $(\uu, \phi, \di, \mu, \vect{h})$ such that:
\begin{itemize}
    \item The following regularity hold
    \begin{equation}
    \label{W:reg}
    \begin{split}
        & \uu \in L^\infty(0,\infty;\dot{L}_\sigma^2(\Omega))\cap L^2(0,\infty;\dot{H}^1_\sigma(\Omega)), \quad
        \partial_t \uu \in L_{\rm uloc}^\frac{4}{d}([0,\infty); \dot{H}^1_\sigma(\Omega)'),\\
        & \phi \in L^\infty(0,\infty;W^{1,4}(\Omega)) \cap L_{\rm uloc}^4([0,\infty);H^2(\Omega)),
        \quad 
        \partial_t \phi \in L^2(0,\infty; H^1(\Omega)'),\\
        & \phi \in L^\infty(\Omega \times (0,\infty)) 
        \text{ such that } |\phi(x,t)|<1 \ \text{a.e. in } \Omega \times (0,\infty),\\
        &\Psi'(\phi)\in L_{\rm uloc}^2([0,\infty); L^p(\Omega)), \text{where } {\color{black} p=6 \text{ if } n=3, \text{ any } \, 2 \leq p <\infty \text{ if } n=2},\\
        &\mu \in L_{\rm uloc}^2([0,\infty);H^1(\Omega)), \quad \nabla \mu \in L^2(\Omega \times (0,\infty)),\\
        &\di \in L^\infty(0,\infty; H^1(\Omega)) \cap L_{\rm uloc}^2(0,\infty;H^2(\Omega)),\\
        &\di \in  L^\infty(\Omega \times (0,\infty))
        \text{ such that } |\di(x,t)|\leq \mathcal{D}_\infty \ \text{a.e. in } \Omega \times (0,\infty),\\
        &\vect{h} \in L^2(0,\infty;L^2(\Omega)).
    \end{split}
    \end{equation}
    
    \item The solution $(\uu, \phi, \di, \mu, \vect{h})$ solves \eqref{eq:NSE}-\eqref{eq:def_h} in the following weak sense:
    \begin{align}
        \label{W:weak:u}
        &
        \langle \partial_t \uu, \vv \rangle_{\dot{H}^1_\sigma(\Omega)}
        + \left( (\uu\cdot \nabla) \uu, \vv \right)
        + (\nu(\phi)D \uu, D\vv)
        = -(\phi \nabla \mu, \vv)
        + ( (\nabla \di)^T \vect{h},\vv),
        \\
        \label{W:weak:phi}
        &
        \langle \partial_t \phi, v \rangle_{H^1(\Omega)}
        + \left( \uu\cdot\nabla \phi, v\right) + (\nabla \mu,\nabla v)=0,
    \end{align}
    for all $\vv \in \dot{H}^1_{\sigma}(\Omega)$, 
    $v \in H^1(\Omega)$ and almost everywhere in $(0,\infty)$, as well as
        \begin{align}
        \label{W:weak:mu}
        &\mu=
        -\diver \left( \left( \varepsilon+ \gamma |\nabla \phi|^2 \right) \nabla \phi \right)
        + \frac{1}{\varepsilon}
        \Psi'(\phi)
        - \frac{\alpha}{2} |\di|^2
        - \beta\diver \left( \left( \nabla \phi \cdot \di \right) \di \right),
        \quad \text{a.e. in } \Omega \times (0,\infty),
        \\
        \label{W:weak:d}
         &\partial_t \di + \left( \uu\cdot \nabla \right) \di = \hh, 
        \quad 
        \hh = \kappa \Delta \di -\alpha \abs{\di}^2 \di + \alpha (\phi-\phi_{\rm cr})\di -\beta (\di \cdot \nabla \phi)\nabla \phi 
        \quad \text{a.e. in } \Omega \times (0,\infty).
    \end{align}
    
    \item The solution satisfies $\uu(\cdot, 0)=\uu_0(\cdot)$, $\phi(\cdot, 0)=\phi_0(\cdot)$, $\di(\cdot,0)=\di_0(\cdot)$ in $\Omega$.
    
    \item The energy inequality
    \begin{equation}
        \label{MR:en:in}
        \begin{split}
        &E_{\rm tot}^\gamma(\uu(t), \phi(t), \di(t))+
        \int_\tau^t \left\| \sqrt{\nu(\phi(s))} D\uu(s)\right\|_{L^2(\Omega)}^2 
        +\| \nabla \mu(s)\|_{L^2(\Omega)}^2
        +\| \vect{h} (s) \|_{L^2(\Omega)}^2 
        \, \d s
        \\
        &\quad \leq E_{\rm tot}^\gamma(\uu(\tau), \phi(\tau), \di(\tau))
        \end{split}
    \end{equation}
holds for all $t \in [\tau,\infty)$ and almost all $\tau \in [0,\infty)$, including $\tau=0$. If $n=2$, \eqref{MR:en:in} holds as equality. 
\end{itemize}
\end{theorem}
\begin{proof}
The proof is based on the application of nested fixed point arguments through the Schauder theorem. We first consider the generalized Cahn-Hilliard/Ericksen-Leslie ($\hbox{CH-EL}$) subsystem with a given incompressible velocity field, and then we prove the existence of a weak solution to the full system. 
\medskip

\noindent
{\bf Step 1: Existence of a weak existence to the  $\hbox{CH-EL}$ system.} 
We fix $\vect{v} \in C([0,T];\vect{V}_m)$ where $\vect{V}_m = \hbox{Span}\, \{ \zz_1,\dots,\zz_m\}$, being $\zz_i$ the eigenfunctions of the Stokes operator (see e.g. \cite[Theorem 2.24]{robinson2016three}). The following inverse Sobolev embedding inequalities hold 
\begin{equation}
\label{Rev-SI}
\| \vv\|_{H^1(\Omega)}\leq C_m \| \vv\|_{L^2(\Omega)},\quad
\| \vv\|_{H^2(\Omega)}\leq C_m \| \vv\|_{L^2(\Omega)},\quad
\| \vv \|_{H^3(\Omega)}\leq C_m \| \vv\|_{L^2(\Omega)}, \quad
\forall \, \vv \in \vect{V}_m.
\end{equation} 
We consider the generalized Cahn-Hilliard/Ericksen-Leslie ($\hbox{CH-EL}$) system with divergence-free drift $\vect{v}$
\begin{equation}
    \label{eq:CH_LC}
    \begin{cases}
    \partial_t\phi + \vect{v} \cdot \nabla \phi = \Delta \mu,\\
    \mu = -\diver\left((\varepsilon + \gamma \abs{\nabla \phi}^2)\nabla \phi\right)+ \frac{1}{\varepsilon}\Psi'(\phi) -\frac{\alpha}{2}\abs{\di}^2 -\diver\left(\beta (\nabla \phi \cdot \di)\di\right),\\
    \partial_t \di + (\vect{v} \cdot \nabla)\di = -\vect{h},\\
    \vect{h} = -\kappa \Delta \di +\alpha \abs{\di}^2 \di - \alpha (\phi-\phi_{\rm cr})\di +\beta (\di \cdot \nabla \phi)\nabla \phi,
    \end{cases}
    \quad \text{in } \Omega \times (0,T),
\end{equation}
which is supplemented with periodic boundary conditions. 
We aim to show the existence of a global weak solution to \eqref{eq:CH_LC} through a fixed point argument. To this end, we fix $\delta \in (0,1]$ and a function $\dtilde \in C([0,T];H^1(\Omega))\cap L^\infty(0,T; L^\infty(\Omega)) \cap L^2(0,T;W^{1,4}(\Omega))$  such that
\begin{equation}
    \label{eq:stime_dtilda}
    \begin{aligned}
    &\norm{\dtilde}_{L^\infty(0,T; L^\infty(\Omega))}\leq \mathcal{D}_\infty, &&\norm{\dtilde}_{C([0,T];H^1(\Omega))}\leq \widetilde{C_1}, &&\norm{\dtilde}_{L^2(0,T;W^{1,4}(\Omega))}\leq \widetilde{C_2},
    \end{aligned}
\end{equation}
where $\widetilde{C_i}$ with $i=1,2$ will be chosen later on. We first consider the following viscous generalized Cahn-Hilliard system
\begin{empheq}[left=\empheqlbrace]{align}
\label{eq:VCH_1}
&\partial_t\phi + \vect{v} \cdot \nabla \phi = \Delta \mu,\\
\label{eq:VCH_2}
    &\mu = \delta \,\partial_t \phi-\diver\left((\varepsilon + \gamma \abs{\nabla \phi}^2)\nabla \phi\right)+ \frac{1}{\varepsilon}\Psi'(\phi) -\frac{\alpha}{2}\abs{\dtilde}^2 -\diver\left(\beta (\nabla \phi \cdot \dtilde)\dtilde\right).
\end{empheq}
The existence of a weak solution on $\Omega \times (0,T)$ to the system \eqref{eq:VCH_1} - \eqref{eq:VCH_2} is proven in Appendix \ref{app:galerkin}.
More precisely, Theorem \ref{vp-CH:ws} guarantees that there exist a pair $(\phi,\mu)$ such that
\begin{align}
    \label{VCH:reg1}
    & \phi \in L^\infty(0,T; W^{1,4}(\Omega))\cap L^2(0,T;H^2(\Omega)),\\
    \label{VCH:reg2}
    &\phi \in L^\infty(\Omega \times (0,T)) \, \text{with }
    |\phi(x,t)|<1 \, \text{a.e. in } \Omega \times (0,T),\\
    \label{VCH:reg3}
    &\partial_t \phi \in L^2(0,T; L^2(\Omega)),
    \quad |\nabla \phi|^2 \nabla \phi \in L^2(0,T;H^1(\Omega;\R^n)), \quad 
    F'(\phi)\in L^2(0,T;L^2(\Omega)),
    \\\label{VCH:reg4}
    &\mu \in L^2(0,T;H^1(\Omega)),
\end{align}
{\color{black}where F is the convex part of $\Psi$ as defined in \eqref{pot-FH}}, which satisfies
\begin{align}
\label{VCH:wp1}
    &( \partial_t \phi, \xi )
    - (\phi \vv, \nabla \xi)
    + (\nabla \mu, \nabla \xi)
    =0, \quad \forall \, \xi \in H^1(\Omega), \text{ a.e. in }(0,T),
    \\
    \label{VCH:wp2}
    &\mu=
    \delta \partial_t \phi
    - \diver \left( (\varepsilon+ \gamma |\nabla \phi|^2) \nabla \phi \right) 
    + \frac{1}{\varepsilon} \Psi'(\phi)
    -\frac{\alpha}{2}  |\dtilde|^2 
    - \beta \diver \left( (\nabla \phi \cdot \dtilde) \dtilde \right), \quad \text{a.e. in } \Omega \times (0,T),
\end{align}
as well as $\phi(\cdot,0)=\phi_0$ in $\Omega$. 
Moreover, we have
\begin{equation}
    \label{A:WS_satisfies_th}
    \begin{aligned}
    &\esssup_{t\in [0,T]}
\int_\Omega \abs{\nabla \phi(t)}^4 + \abs{\nabla \phi(t)}^2 + \Psi(\phi(t))\, \d x
+  \int_0^T \| \nabla \mu(s)\|_{L^2(\Omega)}^2 \, \d s
\\
& \quad 
+\int_0^T \| \partial_t \phi(s)\|_{L^2(\Omega)}^2 \, \d s
+  \int_0^T \| \Delta \phi(s)\|_{L^2(\Omega)}^2 \, \d s
+ \int_0^T \int_\Omega
\left| \nabla\left(\abs{\nabla \phi}^2\right) \right|^2 \, \d x
\\
&\leq
 \overline{C} \left(\int_\Omega \abs{\nabla \phi_0}^4 +  \abs{\nabla \phi_0}^2 + \Psi(\phi_0)\, \d x
+  \int_0^T \|\nabla \dtilde(s) \|_{L^4(\Omega)}^2 \, \d s + T\right)
\\
& \quad 
\times \mathrm{exp}\left(\overline{C}T + \overline{C}\int_0^T \| \nabla \dtilde(s)\|_{L^4(\Omega)}^2 + \|\vv(s) \|_{L^2(\Omega)}^2 \, \d s \right),
    \end{aligned}
\end{equation}
where the positive constant $\overline{C}$ depends on $\alpha, \beta, \gamma, \delta, \varepsilon, \Theta_0, n, m, \overline{\phi_0}, D_\infty$, but is independent of $\dtilde$ and $\vv$.

Next, we consider the Ericksen-Leslie $(\hbox{EL})$ system 
\begin{empheq}[left=\empheqlbrace]{align}
\label{eq:LC_1}
&\partial_t \di + (\vv \cdot \nabla)\di = -\vect{h},\\
\label{eq:LC_2}
    &\vect{h} = -\kappa \Delta \di+ \alpha \abs{\di}^2 \di -\alpha(\phi-\phi_{\rm cr}) \di +\beta (\di \cdot \nabla \phi)\nabla \phi.
\end{empheq}
The existence of a solution $\di$ defined on $\Omega \times (0,T)$ to the system \eqref{eq:LC_1}-\eqref{eq:LC_2} can be obtained through 
classical methods, such as the semigroup theory, due to the regularity given in \eqref{VCH:reg1}-\eqref{VCH:reg4}. More precisely, there exists $\di$ such that
\begin{equation}
\label{Ex:d:reg1}
    \di \in C([0,T];H^1(\Omega))\cap L^2(0,T;H^2(\Omega)) \cap
    H^1(0,T; L^2(\Omega)),
\end{equation}
which solves \eqref{eq:LC_1}-\eqref{eq:LC_2} almost everywhere in $\Omega \times (0,T)$.
Multiplying \eqref{eq:LC_1} by $\di$, integrating over $\Omega$ and using the definition of $\vect{h}$ in \eqref{eq:LC_2}, we obtain
$$
\begin{aligned}
&\frac{1}{2}\ddt \norm{\di}_{L^2(\Omega)}^2 + 
\underbrace{\int_\Omega \left(\vv  \cdot \nabla \right) \di \cdot \di\, \d x}_{=0} + \kappa \norm{\nabla \di}_{L^2(\Omega)}^2 + \alpha \norm{\di}_{L^4(\Omega)}^4 
 +\underbrace{\beta\int_\Omega \abs{\di \cdot \nabla \phi}^2 \, \d x}_{\geq 0}
 \\
&\quad = \alpha \int_\Omega (\phi-\phi_{\rm cr})\abs{\di}^2\, \d x 
\leq \alpha \norm{\phi-\phi_{\rm cr}}_{L^\infty(\Omega)} \norm{\di}_{L^2(\Omega)}^2 \leq 2 \alpha \norm{\di}_{L^2(\Omega)}^2 \leq \frac{\alpha}{2}\norm{\di}_{L^4(\Omega)}^4 + \alpha |\Omega|.
\end{aligned}
$$
Then, an integration in time gives
\begin{equation}
\label{eq:thrid_est_esistenza}
\norm{\di}_{L^\infty(0,T;L^2(\Omega))} + \sqrt{\kappa} \norm{\nabla \di}_{L^2(0,T;L^2(\Omega))} + \norm{\di}^2_{L^4(0,T;L^4(\Omega))} \leq 3 \norm{\di_0}_{L^2(\Omega)} + 3\sqrt{\alpha |\Omega| T}
    =:\mathcal{D}_2.
\end{equation}
Next, we prove that $\vect{d} \in L^\infty(\Omega \times (0,T))$. 
Let us recall that the vector $\vect{d}$ satisfies the following equation
\begin{equation}
\label{eq:d_soddisfa:i}
    \partial_t d_i + \left(\textcolor{black}{\vect{v}} \cdot \nabla\right) d_i - \kappa \Delta d_i 
    +\alpha \abs{\vect{d}}^2 d_i 
    - \alpha (\phi-\phi_{\rm cr}) d_i 
    + \beta \left(\vect{d}\cdot \nabla \phi\right) \partial_i \phi=0, \quad \forall \, i=1,\dots, n.
\end{equation}
In light of the regularity \eqref{Ex:d:reg1},
multiplying \eqref{eq:d_soddisfa:i} by $\abs{\vect{d}}^{p-2} d_i$ with $p>2$, summing the resulting equations and integrating over the domain $\Omega$, we find
\begin{equation}
\label{d:Linf:1}
    \begin{split}
        &\sum_{i=1}^n \int_\Omega \partial_t d_i |\vect{d}|^{p-2} d_i \, \d x +
        \sum_{i=1}^n \int_\Omega \left(\textcolor{black}{\vect{v}} \cdot \nabla \right) d_i \left| \vect{d} \right|^{p-2} d_i \, \d x + \kappa
        \sum_{i=1}^n \int_\Omega \nabla d_i \cdot \nabla \left( \left| \vect{d} \right|^{p-2} d_i \right) \, \d x\\
        &\quad 
        + \alpha \sum_{i=1}^n \int_\Omega | \vect{d}|^2 d_i |\vect{d}|^{p-2} d_i \, \d x
        -\alpha \sum_{i=1}^n \int_\Omega \left( \phi- \phi_{\rm cr} \right) d_i|\vect{d}|^{p-2} d_i \, \d x 
        + \beta \sum_{i=1}^n \int_\Omega \left( \vect{d} \cdot \nabla \phi\right) \partial_i \phi |\vect{d}|^{p-2} d_i \, \d x=0.
    \end{split}
\end{equation}
Observing that
\begin{align}
\label{d:Linf:2}
    \sum_{i=1}^n \int_\Omega \partial_t d_i |\vect{d}|^{p-2} d_i \, \d x 
    &= \frac{1}{p} \ddt \int_\Omega |\vect{d}|^p \, \d x,
    \\
    \label{d:Linf:3}
    \sum_{i=1}^n \int_\Omega \left( \textcolor{black}{\vect{v}} \cdot \nabla \right) d_i \left| \vect{d} \right|^{p-2} d_i \, \d x 
    &= \int_\Omega \textcolor{black}{\vect{v}} \cdot \nabla \left( \frac{1}{p} |\vect{d}|^p \right)\, \d x,
    \\
    \kappa \sum_{i=1}^n \int_\Omega \nabla d_i \cdot \nabla \left( \left| \vect{d} \right|^{p-2} d_i \right) \, \d x
    &= \kappa \sum_{i,k=1}^n \int_\Omega \partial_k d_i \partial_k d_i |\vect{d}|^{p-2} \, \d x
    + \kappa \sum_{i,k,l=1}^n \int_\Omega \partial_k d_i (p-2) |\vect{d}|^{2\left(\frac{p-2}{2}-1\right)} d_l \partial_k d_l d_i \, \d x \notag\\
    \label{d:Linf:4}
    &= \kappa \int_\Omega |\vect{d}|^{p-2} |\nabla \vect{d}|^2 \, \d x + \kappa
    \frac{4(p-2)}{p^2} \int_\Omega 
    \left| \nabla |\vect{d}|^{\frac{p}{2}}\right|^2
    \, \d x,
    \\
    \label{d:Linf:5}
    \alpha \sum_{i=1}^n \int_\Omega | \vect{d}|^2 d_i |\vect{d}|^{p-2} d_i \, \d x
    &=\alpha \int_\Omega |\vect{d}|^{p+2} \, \d x,
    \\
    \label{d:Linf:6}
    -\alpha \sum_{i=1}^n \int_\Omega \left( \phi- \phi_{\rm cr} \right) d_i|\vect{d}|^{p-2} d_i \, \d x 
    &= - \alpha \int_\Omega \left( \phi-\phi_{\rm cr} \right) |\vect{d}|^p \, \d x,
    \\
    \label{d:Linf:7}
    \beta \sum_{i=1}^n \int_\Omega \left( \vect{d} \cdot \nabla \phi\right) \partial_i \phi |\vect{d}|^{p-2} d_i \, \d x
    &= \beta \int_\Omega |\vect{d}|^{p-2} \left| \left( \vect{d}\cdot \nabla \phi \right) \right|^2 \, \d x,
\end{align}
we derive that
\begin{align}\nonumber
    \dt \left[\frac{1}{p}\int_\Omega \abs{\vect{d}}^p\, \d x\right] &+ \underbrace{\int_\Omega \textcolor{black}{\vect{v}}\cdot \nabla\left(\frac{1}{p} \abs{\vect{d}}^{p}\right)\, \d x}_{=0} + \underbrace{\kappa \int_\Omega |\vect{d}|^{p-2} |\nabla \vect{d}|^2 \, \d x +
    \kappa\frac{4(p-2)}{p^2} \int_\Omega 
    \left| \nabla |\vect{d}|^{\frac{p}{2}}\right|^2
    \, \d x}_{\geq 0} \\\nonumber
    &+\underbrace{\alpha \int_\Omega \abs{\vect{d}}^{p+2}\, \d x}_{\geq 0} -\alpha \int_\Omega (\phi-\phi_{\rm cr}) \abs{\vect{d}}^{p}\, \d x + \underbrace{\beta \int_\Omega (\vect{d}\cdot \nabla \phi)^2 \abs{\vect{d}}^{p-2}\, \d x}_{\geq 0} = 0.
\end{align}
Using the fact that $\norm{\phi}_{L^\infty(\Omega)}\leq 1$, we observe that $\alpha \int_\Omega (\phi-\phi_{\rm cr}) \abs{\vect{d}}^{p}\, \d x 
\leq \alpha (1-\phi_{\rm cr})\int_\Omega |\di|^p \, \d x$ (indeed, the subdomain where $\phi<\phi_{\rm cr}$ can be neglected). 
We recall the Young inequality 
\begin{equation}
\label{Young}
ab\leq \frac{\varepsilon a^p}{p}
+ \frac{b^q}{q \varepsilon^\frac{q}{p}}, \quad \text{for any } \, a, b, \varepsilon >0,\quad \frac{1}{p}+\frac{1}{q}=1.
\end{equation}
By using the above inequality, we then have
\begin{align}
\label{d:Linf:8}
   &\dt\left[\frac{1}{p}\int_\Omega \abs{\vect{d}}^p\, \d x\right] 
   + \alpha(1-\phi_{\rm cr}) \,\frac{p+2}{p} \int_\Omega \abs{\vect{d}}^p \, \d x \leq \alpha (1-\phi_{\rm cr}) \int_\Omega\abs{\vect{d}}^{p}\, \d x 
   + \alpha \frac{2}{p}(1-\phi_{\rm cr})^{\frac{p}{2}+1} \abs{\Omega},
  \end{align}
which immediately gives  
\begin{align}
\label{d:Linf:9}
 &\ddt \norm{\vect{d}}^p_{L^p(\Omega)} 
 + 2\alpha (1-\phi_{\rm cr})\norm{\vect{d}}^p_{L^p(\Omega)} 
 \leq 2\alpha(1-\phi_{\rm cr})^{\frac{p}{2}+1}\abs{\Omega},
\end{align}
where $\abs{\Omega}=(2\pi)^n$. By the Gronwall lemma, we find
\begin{equation}
\label{d:Linf:10}
    \norm{\vect{d}(t)}_{L^p(\Omega)}^p \leq \norm{\vect{d}(0)}^p_{L^p(\Omega)} {\rm e}^{-2 \alpha (1-\phi_{\rm cr}) t} + (1-\phi_{\rm cr})^{\frac{p}{2}} \abs{\Omega} \left(1 - {\rm e}^{-2 \alpha (1-\phi_{\rm cr}) t}\right), \quad \forall \, t \geq 0,
\end{equation}
which, in turn, entails that
$$
\norm{\vect{d}(t)}_{L^p(\Omega)} 
\leq \max\left\lbrace\norm{\vect{d}(0)}_{L^p(\Omega)},  (1-\phi_{\rm cr})^{\frac12} \abs{\Omega}^{\frac{1}{p}} \right\rbrace, \quad \forall \, t \geq 0.
$$
Taking the limit as $p \to \infty$, we obtain
\begin{equation}
    \label{Ex:d:Linf}
   \sup_{t \geq 0} \norm{\vect{d}(t)}_{L^\infty(\Omega)}\leq \max\left\lbrace\norm{\vect{d}_0}_{L^\infty(\Omega)}, \sqrt{1-\phi_{\rm cr}} \right\rbrace=: \mathcal{D}_\infty.
\end{equation}

We now derive a higher order Sobolev estimate of $\di$. Multiplying \eqref{eq:LC_1} by $-\Delta \di$, integrating over $\Omega$ and using \eqref{eq:LC_2}, we have
$$
\begin{aligned}
&\frac{1}{2} \dt \norm{\nabla \di}_{L^2(\Omega)}^2 + \kappa \norm{\Delta \di}_{L^2(\Omega)}^2 \underbrace{- \alpha \int_\Omega \abs{\di}^2 \di \cdot \Delta \di\, \d x}_{\geq 0} \\
&\quad = \int_\Omega \left(\vv\cdot \nabla\right) \di \cdot \Delta \di\, \d x+ \alpha \int_\Omega (\phi-\phi_{\rm cr}) \di \cdot (-\Delta \di)\, \d x + \beta \int_\Omega (\nabla \phi \cdot \di)\nabla \phi \cdot \Delta \di\, \d x.
\end{aligned}
$$
Then, by Gagliardo-Nirenberg interpolation inequalities and \eqref{eq:thrid_est_esistenza}, we observe that the right hand side can be estimated as follows
$$
\begin{aligned}
\left| \int_\Omega \left(\vv\cdot \nabla\right) \di \cdot \Delta \di\, \d x \right| 
&\leq \norm{\vv}_{L^6(\Omega)}\norm{\nabla \di}_{L^3(\Omega)} \norm{\Delta \di}_{L^2(\Omega)}
\leq C \norm{\vv}_{L^6(\Omega)}\norm{\nabla \di}_{L^2(\Omega)}^{\frac{1}{2}} \norm{\Delta \di}_{L^2(\Omega)}^{\frac{3}{2}}\\
&\leq \frac{\kappa}{12}\norm{\Delta \di}_{L^2(\Omega)}^2 + C \norm{\vv}_{L^6(\Omega)}^4 \norm{ \di}_{L^2(\Omega)} \norm{\Delta \di}_{L^2(\Omega)} \leq \frac{\kappa}{6} \norm{\Delta \di}_{L^2(\Omega)}^2 + C_m \mathcal{D}_2^2 \|\vv \|_{L^2(\Omega)}^8,\\
\left| \alpha \int_\Omega (\phi-\phi_{\rm cr}) \di \cdot (-\Delta \di)\, \d x \right| &\leq \alpha\norm{\phi-\phi_{\rm cr}}_{L^\infty(\Omega)}\norm{\di}_{L^2(\Omega)}\norm{\Delta \di}_{L^2(\Omega)} \leq 2 \alpha \norm{\di}_{L^2(\Omega)}\norm{\Delta \di}_{L^2(\Omega)} \\
&\leq \frac{\kappa}{6} \norm{\Delta \di}_{L^2(\Omega)}^2 + C \mathcal{D}_2^2,\\
\left| \beta \int_\Omega (\nabla \phi \cdot \di)\nabla \phi \cdot \Delta \di\, \d x \right|
&\leq \beta C\norm{\nabla \phi}_{L^4(\Omega)}^2 \norm{\di}_{L^\infty(\Omega)} \norm{\Delta \di}_{L^2(\Omega)}
\leq \beta C \norm{\phi}_{L^\infty(\Omega)} \norm{\phi}_{H^2(\Omega)} \mathcal{D}_\infty \norm{\Delta \di}_{L^2(\Omega)} \\
&\leq \beta C \mathcal{D}_\infty \norm{\phi}_{H^2(\Omega)} \norm{\Delta \di}_{L^2(\Omega)}\leq \frac{\kappa}{6} \norm{\Delta \di}_{L^2(\Omega)}^2 + \frac{\beta^2 C\mathcal{D}_\infty^2}{\kappa} \norm{\phi}_{H^2(\Omega)}^2,
\end{aligned}
$$
where $C$ is a generic constant depending only on $\Omega$ and the parameters of the system. Thus, we arrive at
$$
\dt \norm{\nabla \di}_{L^2(\Omega)}^2 + \kappa \norm{\Delta \di}_{L^2(\Omega)}^2 \leq 2 C_m \left( 1+ \|\vv \|_{L^2(\Omega)}^8 \right) \mathcal{D}_2^2 + \frac{2\beta^2  C \mathcal{D}_2^2}{\kappa} \norm{\phi}_{H^2(\Omega)}^2.
$$
Now, by using \eqref{H2:equiv}, \eqref{GN}, \eqref{A:WS_satisfies_th} and recalling that $\overline{\phi(t)}=\overline{\phi_0}$ for any $t\in [0,T]$, an integration in time gives
\begin{equation}
    \label{eq:stima_nablad_esistenza}
    \begin{aligned}
    & \max_{t\in [0,T]}
    \norm{\nabla \di (t)}_{L^2(\Omega)}^2 + \kappa \int_0^T \norm{\Delta \di(\tau)}_{L^2(\Omega)}^2 \, \d \tau 
    \\
    & \quad 
    \leq \norm{\nabla \di_0}_{L^2(\Omega)}^2 
    + 2C_m (1+ \| \vv\|_{L^\infty(0,T;L^2(\Omega))}^8) \mathcal{D}_2^2 T
    + \frac{4\beta C \mathcal{D}_2^2}{\kappa}\int_0^T \left(
    (2 \pi)^n \abs{\overline{\phi}(s)}^2 + 2 \norm{\Delta \phi(s)}_{L^2(\Omega)}^2\right) \, \d s
    \\
    & \quad 
    \leq \norm{\nabla \di_0}_{L^2(\Omega)}^2 
    +2C_m (1+ \| \vv\|_{L^\infty(0,T;L^2(\Omega))}^8) \mathcal{D}_2^2 T
    + \frac{4\beta C \mathcal{D}_2^2}{\kappa}(2 \pi)^n \abs{\overline{\phi_0}}^2 T 
    \\
    & \qquad     
    +\frac{8\beta C \mathcal{D}_2^2}{\kappa} \overline{C} \left( \int_\Omega \abs{\nabla \phi_0}^4 +  \abs{\nabla \phi_0}^2
    + \Psi(\phi_0)\, \d x+
  \mathcal{D}_\infty T^\frac12 \left(\int_0^T \| \dtilde(s) \|_{H^2(\Omega)}^2 \, \d s\right)^\frac12 + T\right) \mathrm{e}^{G(T)}
    \end{aligned}
\end{equation}
where 
$$
G(T):= \overline{C}T + \overline{C}\mathcal{D}_\infty T^\frac12 \left(\int_0^T \| \dtilde(s)\|_{H^2(\Omega)}^2 \, \d s\right)^\frac12 + T \|\vv\|_{L^\infty(0,T; L^2(\Omega))}^2.
$$
Let $\widetilde{C_i}$ be a generic constant depending on the parameters $\alpha, \beta, \gamma, \delta, \varepsilon, \kappa, \Theta_0, n, m, \overline{\phi_0}$, and on the norms $ 
\norm{\di_0}_{H^1(\Omega)}$, $\norm{\phi_0}_{W^{1,4}(\Omega)}$, $\norm{\Psi(\phi_0)}_{L^1(\Omega)}$, $\norm{\vv}_{L^\infty(0,T;L^2(\Omega))}$ and $
\mathcal{D}_\infty$, but is independent of $\dtilde$ and $T$. 
Combining \eqref{eq:thrid_est_esistenza} and \eqref{eq:stima_nablad_esistenza}, we find 
\begin{equation}
\label{dH2-fp}
\begin{split}
    \int_0^T \| \di(s)\|_{H^2(\Omega)}^2 \,\d s
    &\leq \widetilde{C_1} 
    \mathrm{exp}\left( \widetilde{C_2} T + \widetilde{C_2} T^\frac12 \| \dtilde\|_{L^2(0,T;H^2(\Omega))}\right)
    + \widetilde{C_3} T \mathrm{exp}\left( \widetilde{C_2} T + \widetilde{C_2} T^\frac12 \| \dtilde\|_{L^2(0,T;H^2(\Omega))}\right)
   \\
  & \quad + \widetilde{C_4} T^\frac12 \| \dtilde\|_{L^2(0,T;H^2(\Omega))} \mathrm{exp}
  \left( \widetilde{C_2} T + \widetilde{C_2} T^\frac12 \|\dtilde\|_{L^2(0,T;H^2(\Omega))} \right).
    \end{split}
\end{equation}
Let us set
$
\sigma= \widetilde{C_1} \mathrm{e}
$
and assume that
$
\| \dtilde\|_{L^2(0,T_\star;H^2(\Omega))} \leq \sqrt{2 \sigma},
$
where $T_\star$ is sufficiently small such that
$$
\widetilde{C_2} T_\star + \widetilde{C_2} T_\star^\frac12 \sqrt{2 \sigma} \leq 1
\quad \text{and} \quad 
\widetilde{C_3} T_\star \mathrm{e}+ \widetilde{C_4} T_\star^\frac12 \sqrt{2 \sigma} \mathrm{e} \leq \sigma.
$$
Thus, we deduce from \eqref{dH2-fp} that 
$$
\int_0^{T_\star} \| \di(s)\|_{H^2(\Omega)}^2 \leq 2\sigma.
$$
Furthermore, we also infer from \eqref{eq:stima_nablad_esistenza} that
$$
\| \di\|_{C([0,T_\star];H^1(\Omega))}\leq \widetilde{\sigma},
$$
where $\widetilde{\sigma}$ depends on $\sigma$, the parameters of the system, the norm of the initial data, $\norm{\vv}_{L^\infty(0,T_\star;L^2(\Omega))}$ and $\mathcal{D}_\infty$, but is independent of $\dtilde$.
Therefore, we define the map
\begin{equation}
    \label{eq:def_S}
    S: \mathcal{B}_{T_\star} \to \mathcal{B}_{T_\star}, \quad
    \dtilde \mapsto S(\dtilde):= \di,
\end{equation}
where
\begin{equation}
    \label{eq:def_BT}
    \begin{aligned}
    \mathcal{B}_{T_\star} = &\Big\{ \di \in C([0,T_\star];H^1(\Omega))\cap L^2(0,T_\star;H^2(\Omega))\cap L^\infty(0,T_\star;L^\infty(\Omega)):\\
    &\norm{\di}_{L^\infty(0,T_\star;L^\infty(\Omega))}\leq \mathcal{D}_{\infty}, 
    \, \norm{\di}_{C([0,T_\star];H^1(\Omega))}\leq \widetilde{\sigma}, 
    \, \norm{\di}_{L^2(0,T_\star;H^2(\Omega))}\leq \textcolor{black}{\sqrt{2 \sigma}}\Big\},
    \end{aligned}
\end{equation}
and $\di$ solves the $\hbox{EL}$ system \eqref{eq:LC_1}-\eqref{eq:LC_2} in $(0,T_\star)$, with $\phi$ solution to the system \eqref{eq:VCH_1}-\eqref{eq:VCH_2} restricted on $(0,T_\star)$. We notice that $\mathcal{B}_{T_\star}$ is non-empty, convex and closed. 
In addition, it follows from \eqref{eq:LC_1} that
$$
\begin{aligned}
    \norm{\partial_t \di}_{L^2(\Omega)}&\leq \norm{(\vv \cdot \nabla) \di}_{L^2(\Omega)} + \kappa \norm{\Delta \di}_{L^2(\Omega)} + \alpha \norm{\abs{\di}^2 \di}_{L^2(\Omega)} + \alpha \norm{(\phi-\phi_{\rm cr})\di}_{L^2(\Omega)} + \beta \norm{(\nabla \phi \cdot \di) \nabla \phi}_{L^2(\Omega)}\\
    &\leq \norm{\vv}_{L^6(\Omega)}\norm{\nabla \di}_{L^3(\Omega)} + \kappa \norm{\Delta \di}_{L^2(\Omega)} + \alpha \norm{\di}_{L^\infty(\Omega)}^3 \sqrt{\abs{\Omega}} + 2 \alpha \norm{\di}_{L^\infty(\Omega)}\sqrt{\abs{\Omega}} + \beta \norm{\di}_{L^\infty(\Omega)}\norm{\nabla \phi}_{L^4(\Omega)}^2\\
    &\leq C_m \norm{\di}_{H^2(\Omega)} + \alpha \mathcal{D}_\infty^3 \sqrt{\abs{\Omega}} + 2 \alpha \mathcal{D}_\infty \sqrt{\abs{\Omega}} + \beta \mathcal{D}_\infty \norm{\phi}_{H^2(\Omega)}.
\end{aligned}
$$
Exploiting \eqref{A:WS_satisfies_th} and \eqref{eq:def_BT}, 
an integration in time entails 
\begin{equation}
    \label{eq:stima_esistenza_dertempo}
    \int_0^{T_\star} \norm{\partial_t \di}_{L^2(\Omega)}^2 \, \d \tau\leq \overline{\sigma}^2,
\end{equation}
where $\overline{\sigma}:= C(\sigma,\widetilde{\sigma},T_\star)$.
Hence, 
$$
S: \mathcal{B}_{T_\star} \to \mathcal{B}_{T_\star} \cap \left\{\di \in W^{1,2}(0,T_\star;L^2(\Omega)):
\norm{\partial_t\di}_{L^2(0,T_\star;L^2(\Omega))}\leq \overline{\sigma}\right\},
$$
thereby $S$ is a compact application in $L^2(0,T;L^2(\Omega))$. We are left to prove that $S$ is a continuous application. We consider $\{\dtilde_k\}_{k \in \N}\subset \mathcal{B}_T$ such that 
$\dtilde_k \to \dtilde$ in $L^2(0,T;L^2(\Omega))$ and we claim that $\di_k = S(\dtilde_k) \to \di = S(\dtilde)$ in $L^2(0,T;L^2(\Omega))$. To this end, let $(\Phi_k, M_k)$ and $(\Phi,M)$ be the weak solutions to the following problems
$$
\left\{\begin{aligned}
&\partial_t \Phi_k + \vv \cdot \nabla \Phi_k = \Delta M_k,\\
& M_k =\delta \partial_t \Phi_k -\diver\left((\varepsilon + \gamma\abs{\nabla \Phi_k}^2)\nabla \Phi_k\right)+\frac{1}{\varepsilon}\Psi'(\Phi_k)-\frac{\alpha}{2}\abs{\dtilde_k}^2 -\beta\diver\left((\nabla \Phi_k \cdot \dtilde_k)\dtilde_k\right),
\end{aligned}\right.
$$
and
$$
\left\{\begin{aligned}
&\partial_t \Phi + \vv \cdot \nabla \Phi = \Delta M,\\
& M = \delta \partial_t \Phi -\diver\left((\varepsilon + \gamma\abs{\nabla \Phi}^2)\nabla \Phi\right)+\frac{1}{\varepsilon}\Psi'(\Phi)-\frac{\alpha}{2}\abs{\dtilde}^2 -\beta\diver\left((\nabla \Phi \cdot \dtilde)\dtilde\right),
\end{aligned}\right.
$$
in the sense of \eqref{VCH:reg1}-\eqref{VCH:reg3} and \eqref{VCH:wp1}-\eqref{VCH:wp2}.
Defining $\psi:= \Phi-\Phi_k$ and $\chi:= M-M_k$, the pair $(\psi, \chi)$ solves the system
\begin{equation}
\label{eq:VCH_diff}
\left\{\begin{aligned}
&(\partial_t \psi, \xi) + 
(\vv \cdot \nabla \psi, \xi) + (\nabla \chi, \nabla \xi) =0,
\quad \forall \, \xi \in H^1(\Omega), \text{a.e. in } (0,T),\\
&\chi = \delta \partial_t \psi -\varepsilon\Delta \psi -\gamma\diver\left( \abs{\nabla \Phi}^2 \nabla \Phi\right) + \gamma\diver\left( \abs{\nabla \Phi_k}^2 \nabla \Phi_k\right) +\frac{1}{\varepsilon}\Psi'(\Phi)-\frac{1}{\varepsilon}\Psi'(\Phi_k)\\
&\quad \quad-\frac{\alpha}{2}\left(\abs{\dtilde}^2 -\abs{\dtilde_k}^2\right)-\beta\diver\left((\nabla \Phi\cdot \dtilde)\dtilde\right)+\beta\diver\left((\nabla \Phi_k\cdot \dtilde_k)\dtilde_k\right) \quad \text{a.e. in }\Omega \times (0,T).
\end{aligned}\right.
\end{equation}
Let us choose $\xi=(-\Delta)^{-1}\psi$ in \eqref{eq:VCH_diff}. By the conservation of mass, it follows that $\overline{\psi}=0$, thereby $(-\Delta)^{-1}\psi$ is well-defined. We find
$$
\begin{aligned}
&\dt \left(\frac{1}{2}\norm{\psi}^2_{\ast} +\frac{\delta}{2}\norm{\psi}_{L^2(\Omega)}^2\right) + \varepsilon\norm{\nabla \psi}_{L^2(\Omega)}^2 + \underbrace{\gamma \int_\Omega \prt{\abs{\nabla \Phi}^2 \nabla \Phi - \abs{\nabla \Phi_k}^2 \nabla \Phi_k}\cdot \nabla \psi\, \d x}_{\gamma \geq 0, \,\geq 0} +\underbrace{\int_\Omega \frac{1}{\varepsilon} \prt{F'(\Phi) -F'(\Phi_k)}\cdot \psi\, \d x}_{\geq 0} \\
&\quad 
= \frac{\Theta_0}{\varepsilon} \norm{\psi}_{L^2(\Omega)}^2 
- \left(\vv \cdot \nabla \psi,(-\Delta)^{-1} \psi\right)
+\frac{\alpha}{2} \int_\Omega \prt{\abs{\dtilde}^2 -\abs{\dtilde_k}^2}\cdot\psi\, \d x
-\beta\int_\Omega \left[\prt{\nabla \Phi\cdot \dtilde}\dtilde -\prt{\nabla \Phi_k \cdot \dtilde_k}\dtilde_k\right]\cdot \nabla \psi\, \d x,
\end{aligned}
$$
where
$\| f\|_{\ast}=\| \nabla (-\Delta)^{-1}f\|_{L^2(\Omega)}.$
Since $\overline{\psi}=0$,  by \eqref{Rev-SI} and $\dtilde,\dtilde_k \in \mathcal{B}_T$, we have
$$
\begin{aligned}
\left| -\left(\vv \cdot \nabla \psi, (-\Delta)^{-1} \psi\right) \right|
&= \left|\int_\Omega \psi \vv \cdot \nabla\prt{(-\Delta)^{-1} \psi}\, \d x \right|
\leq \norm{\psi}_{L^6(\Omega)} \norm{\vv}_{L^3(\Omega)} \norm{ \psi}_{\ast}
\leq \frac{\varepsilon}{8} \norm{\nabla \psi}_{L^2(\Omega)}^2 + C_m \norm{\psi}_{\ast}^2,
\end{aligned}
$$
and
$$
\begin{aligned}
\left|\frac{\Theta_0}{\varepsilon} \norm{\psi}_{L^2(\Omega)}^2\right| &\leq \frac{\varepsilon}{8}\norm{\nabla \psi}_{L^2(\Omega)}^2 + C \norm{\psi}_{\ast}^2.
\end{aligned}
$$
Similarly, we find
$$
\begin{aligned}
\left|\frac{\alpha}{2} \int_\Omega \prt{\abs{\dtilde}^2 -\abs{\dtilde_k}^2}\cdot\psi\, \d x\right|
&\leq C\prt{ {\color{black} \norm{\dtilde}_{L^\infty(\Omega)} +\norm{\dtilde_k}_{L^\infty(\Omega)}} } \norm{\dtilde-\dtilde_k}_{L^2(\Omega)}\norm{\psi}_{L^2(\Omega)}\\
&\leq \frac{\varepsilon}{8} \norm{\nabla \psi}_{L^2(\Omega)}^2 + C \norm{\dtilde-\dtilde_k}_{L^2(\Omega)}^2,
\end{aligned}
$$
and
$$
\begin{aligned}
&-\beta\int_\Omega \left[\prt{\nabla \Phi\cdot \dtilde}\dtilde -\prt{\nabla \Phi_k \cdot \dtilde_k}\dtilde_k\right]\cdot \nabla \psi\, \d x
\\
&\quad = \underbrace{-\beta \int_\Omega \prt{\nabla \psi \cdot \dtilde}^2\, \d x}_{\leq 0} -\beta \int_\Omega \left[\prt{\nabla \Phi_k \cdot \dtilde}\dtilde- \prt{\nabla \Phi_k \cdot \dtilde_k}\dtilde_k\right]\cdot \nabla \psi\, \d x\\
&\quad \leq  -\beta \int_\Omega \left[\prt{\nabla \Phi_k \cdot (\dtilde-\dtilde_k)}\dtilde_k\right]\cdot \nabla \psi\, \d x -\beta\int_\Omega \left[\prt{\nabla \Phi_k \cdot \dtilde_k} {\color{black} \prt{\dtilde-\dtilde_k}} \right]\cdot \nabla \psi\, \d x\\
&\quad \leq 2 \beta  \norm{\nabla \Phi_k}_{L^4(\Omega)} \norm{\dtilde-\dtilde_k}_{L^4(\Omega)} \norm{\dtilde}_{L^\infty(\Omega)} \norm{\nabla \psi}_{L^2(\Omega)}\\
&\quad \leq \frac{\varepsilon}{4}\norm{\nabla \psi}_{L^2(\Omega)}^2  + C \norm{\nabla \Phi_k}_{L^4(\Omega)}^2 \norm{\dtilde-\dtilde_k}_{L^\infty(\Omega)}\norm{\dtilde-\dtilde_k}_{L^2(\Omega)}\\
&\quad \leq \frac{\varepsilon}{4}\norm{\nabla \psi}_{L^2(\Omega)}^2  + C \mathcal{D}_\infty\norm{\nabla \Phi_k}_{L^4(\Omega)}^2 \norm{\dtilde-\dtilde_k}_{L^2(\Omega)}.
\end{aligned}
$$
Observe that all the above constants are independent of $k$. 
By \eqref{A:WS_satisfies_th}, it follows from $\dtilde_k \in \mathcal{B}_T$ that $\norm{\nabla \Phi_k}_{L^\infty(0,T; L^4(\Omega))}\leq C$, where $C$ is independent of $k$. Hence, we arrive at
$$
\begin{aligned}
\dt \left( \frac{1}{2}\norm{\psi}^2_{\ast} +\frac{\delta}{2}\norm{\psi}_{L^2(\Omega)}^2\right) &+ \frac{\varepsilon}{2}\norm{\nabla \psi}_{L^2(\Omega)}^2 \leq C\norm{\psi}_{\ast}^2 + C \norm{\dtilde-\dtilde_k}_{L^2(\Omega)}^2 + C \norm{\dtilde-\dtilde_k}_{L^2(\Omega)}.
\end{aligned}
$$
The Gronwall lemma implies that 
\begin{equation}
    \label{eq:phitozero}
    \begin{split}
\norm{\psi}_{L^\infty(0,T;L^2(\Omega))} &+ 
    \norm{\psi}_{L^2(0,T;H^1(\Omega))} \\
    &\leq C \mathrm{e}^{CT}\left[\int_0^T \norm{\dtilde(s)-\dtilde_k(s)}_{L^2(\Omega)}^2 \, \d s + \int_0^T \norm{\dtilde(s)-\dtilde_k(s)}_{L^2(\Omega)}\, \d s\right]^\frac12 \to 0.
    \end{split}
\end{equation}
Next, we consider $\di_k = S(\dtilde_k)$ and $\di = S(\dtilde)$, which solve
$$
\left\{\begin{aligned}
&\partial_t \di_k + \prt{\vv \cdot \nabla}\di_k = -\vect{h}_k,\\
&\vect{h}_k = -\kappa \Delta \di_k +\alpha \abs{\di_k}^2 \di_k - \alpha(\Phi_k -\phi_{\rm cr}) \di_k +\beta \prt{\di_k \cdot \nabla \Phi_k}\nabla \Phi_k
\end{aligned}\right.
$$
and
$$
\left\{\begin{aligned}
&\partial_t \di + \prt{\vv \cdot \nabla}\di = -\vect{h},\\
&\vect{h} = -\kappa \Delta \di +\alpha \abs{\di_k}^2 \di - \alpha(\Phi -\phi_{\rm cr}) \di +\beta \prt{\di \cdot \nabla \Phi}\nabla \Phi
\end{aligned}\right.
$$
almost everywhere in $\Omega \times (0,T)$.
We define $\p := \di-\di_k$ and $\vect{\Xi}:= \vect{h} - \vect{h}_k$ such that
\begin{equation}
    \label{eq:LC_diff}
    \left\{
    \begin{aligned}
    &\partial_t \p + (\vv \cdot \nabla)\p = -\vect{\Xi},\\
    &\vect{\Xi}=-\kappa \Delta \p +\alpha \prt{\abs{\di}^2 \di -\abs{\di_k}^2 \di_k} - \alpha \prt{\Phi - \phi_{\rm cr}} \di +\alpha\prt{\Phi_k -\phi_{\rm cr}}\di_k \\
    &\quad +\beta \prt{\di \cdot \nabla \Phi}\nabla \Phi - \beta \prt{\di_k \cdot \nabla \Phi_k}\nabla \Phi_k.
    \end{aligned}\right.
\end{equation}
Multiplying \eqref{eq:LC_diff} by $\p$ and integrating over $\Omega$, we find
$$
\begin{aligned}
&\frac{1}{2}\dt \norm{\p}_{L^2(\Omega)}^2 
+ \kappa \norm{\nabla \p}_{L^2(\Omega)}^2 + \alpha \underbrace{\int_\Omega \prt{\abs{\di}^2 \di -\abs{\di_k}^2 \di_k}\cdot \p\, \d x}_{\geq 0} +\underbrace{\beta\int_\Omega \prt{\p\cdot \nabla \Phi}^2\, \d x }_{\geq 0}
\\
&\quad = \alpha \int_\Omega \psi \di \cdot \p\, \d x + \alpha \int_\Omega \prt{\Phi_k -\phi_{\rm cr}}\p^2\, \d x -\beta \int_\Omega \left[\prt{\di_k \cdot \nabla \Phi}\nabla \Phi -\prt{\di_k \cdot \nabla \Phi_k}\nabla \Phi_k\right] \cdot \p\, \d x.
\end{aligned}
$$
Since $\di, \di_k \in \mathcal{B}_T$ and $\norm{\nabla \Phi_k}_{L^\infty(0,T; L^4(\Omega))}\leq C$, $\norm{\nabla \Phi}_{L^\infty(0,T; L^4(\Omega))}\leq C$, where the positive constant $C$ is independent of $k$ as observed before,
the terms on the right-hand side can be estimated as follows
$$
\begin{aligned}
&\left|\alpha \int_\Omega \psi \di \cdot \p\, \d x\right| \leq \alpha \norm{\psi}_{L^2(\Omega)}\norm{\di}_{L^\infty(\Omega)}\norm{\p}_{L^2(\Omega)}\leq C \norm{\p}_{L^2(\Omega)}^2 + \norm{\psi}_{L^2(\Omega)}^2,\\
&\left|\alpha \int_\Omega \prt{\Phi_k -\phi_{\rm cr}}\p^2\, \d x\right| \leq\alpha \norm{\Phi_k -\phi_{\rm cr}}_{L^\infty(\Omega)}\norm{\p}_{L^2(\Omega)}^2\leq 2 C \norm{\p}_{L^2(\Omega)}^2,
\end{aligned}
$$
and 
$$\begin{aligned}
\left|\beta \int_\Omega \left[\prt{\di_k \cdot \nabla \Phi}\nabla \Phi -\prt{\di_k \cdot \nabla \Phi_k}\nabla \Phi_k\right] \cdot \p\, \d x\right| 
&\leq \beta \norm{\di_k}_{L^\infty(\Omega)}\norm{\nabla \psi}_{L^2(\Omega)}\prt{\norm{\nabla \Phi}_{L^4(\Omega)}+ \norm{\nabla \Phi_k}_{L^4(\Omega)}}\norm{\p}_{L^4(\Omega)}
\\
&\leq  \frac{\kappa}{2}\norm{\nabla \p}_{L^2(\Omega)}^2
+C \norm{\nabla \psi}_{L^2(\Omega)}^2 + C \norm{\p}_{L^2(\Omega)}^2.
\end{aligned}
$$
Hence, we obtain
$$
\frac{1}{2}\dt \norm{\p}_{L^2(\Omega)}^2 +\frac{\kappa}{2}\norm{\nabla \p}_{L^2(\Omega)}^2 \leq C \norm{\p}_{L^2(\Omega)}^2 + C \norm{\psi}_{H^1(\Omega)}^2.
$$
In light of \eqref{eq:phitozero}, we conclude from the Gronwall lemma that
\begin{equation}
\label{eq:d_to0}
\norm{\p}_{L^\infty(0,T;L^2(\Omega))} + \norm{\nabla \p}_{L^2(0,T;L^2(\Omega))} \leq C {\rm e}^{CT} \left[\int_0^T \norm{\psi(s)}_{H^1(\Omega)}^2 \, \d s\right]^\frac12 \to 0,
\end{equation}
namely $\di_k \to \di$ in $L^\infty(0,T;L^2(\Omega))\cap L^2(0,T;L^2(\Omega))$, which is the desired conclusion. 

We are now in the position to conclude the first part of the proof. Owing to the above analysis, the Schauder fixed point theorem entails that $S$ has a fixed point. To summarize, for any $\vv \in C([0,T], \vect{V}_m)$, there exist $0< T_\star \leq T$ and a quadruplet $(\phi, \di, \mu, \vect{h})$ on $[0,T_\star]$ 
such that 
\begin{align}
    \label{VCH-LC:reg1}
    & \phi \in L^\infty(0,T_\star; W^{1,4}(\Omega))\cap L^2(0,T_\star;H^2(\Omega)),\\
    \label{VCH-LC:reg2}
    &\phi \in L^\infty(\Omega \times (0,T_\star)) \, \text{with }
    |\phi(x,t)|<1 \, \text{a.e. in } \Omega \times (0,T_\star),\\
    \label{VCH-LC:reg3}
    &\partial_t \phi \in L^2(0,T_\star; L^2(\Omega)),
    \quad |\nabla \phi|^2 \nabla \phi \in L^2(0,T_\star;H^1(\Omega)), \quad 
    F'(\phi)\in L^2(0,T_\star;L^2(\Omega)),
    \\
    \label{VCH-LC:reg4}
    & \di \in C([0,T_\star];H^1(\Omega))\cap L^2(0,T_\star;H^2(\Omega)) \cap
    H^1(0,T_\star; L^2(\Omega)),
    \\
     \label{VCH-LC:reg5}
    &  \di \in L^\infty(\Omega \times (0,T_\star)) \, \text{with }
    |\di(x,t)|\leq \mathcal{D}_\infty \, \text{a.e. in } \Omega \times (0,T_\star),
    \\
    \label{VCH-LC:reg6}
    &\mu \in L^2(0,T_\star;H^1(\Omega)),
    \quad 
    \vect{h} \in L^2(0,T_\star; L^2(\Omega)),
    \end{align}
satisfying
\begin{subequations}
\begin{alignat}{2}
\label{VCH-LC:wp1}
    &( \partial_t \phi, \xi )
    - (\phi \vv, \nabla \xi)
    + (\nabla \mu, \nabla \xi)
    =0, \quad &&\forall \, \xi \in H^1(\Omega), \text{ a.e. in }(0,T_\star),
    \\
    \label{VCH-LC:wp2}
    &\mu=
    \delta \partial_t \phi
    - \diver \left( (\varepsilon+ \gamma |\nabla \phi|^2) \nabla \phi \right) 
    + \frac{1}{\varepsilon} \Psi'(\phi)
    -\frac{\alpha}{2}  |\di|^2 
    - \beta \diver \left( (\nabla \phi \cdot \di) \di \right), \quad &&\text{a.e. in } \Omega \times (0,T_\star),
    \\
    \label{VCH-LC:wp3}
    &\partial_t \di + \prt{\vv \cdot \nabla}\di = -\vect{h}, \quad 
    &&\text{a.e. in } \Omega \times (0,T_\star),\\
    \label{VCH-LC:wp4}
    &\vect{h} = -\kappa \Delta \di +\alpha \abs{\di}^2 \di -\alpha(\phi -\phi_{\rm cr}) \di +\beta \prt{\di \cdot \nabla \phi}\nabla \phi, \quad &&
    \text{a.e. in } \Omega \times (0,T_\star),
\end{alignat}
\end{subequations}
and  
\begin{equation}
\label{VCH-LC:wp5}
\phi(\cdot,0)=\phi_0, \quad \di(\cdot, 0)=\di_0 \quad \text{in } \Omega.
\end{equation}

\noindent
\textbf{Step 2: Global solutions and vanishing viscosity limit in the Cahn-Hilliard/Ericksen-Leslie system.}
The aim of this part is twofold: showing that $(\phi,\mu,\di,\hh)$ fulfilling \eqref{VCH-LC:reg1}-\eqref{VCH-LC:reg6}, \eqref{VCH-LC:wp1}-\eqref{VCH-LC:wp5} can be extended on $[0,T]$, where $T$ is arbitrary, and letting the parameter $\delta\rightarrow 0$ in  
\eqref{VCH-LC:wp1}-\eqref{VCH-LC:wp4}.
To this end, for any $\delta \in (0,1)$ and $\vv \in C([0,T];\vect{V}_m)$, let $(\phi^\delta,  \di^\delta, \mu^\delta, \hh^\delta)$ be a solution to \eqref{VCH-LC:wp1}-\eqref{VCH-LC:wp4} fulfilling \eqref{VCH-LC:reg1}-\eqref{VCH-LC:reg6} with the initial datum \eqref{VCH-LC:wp5}. It is clear that any $(\phi^\delta, \di^\delta, \mu^\delta, \hh^\delta)$ is defined on $[0,T_\star^\delta]$. We will derive uniform estimates on the sequence of solutions $(\phi^\delta, \di^\delta, \mu^\delta, \hh^\delta)$ which does not blow-up as $t \to T_\star^\delta$ and are independent of the parameter $\delta$. 
First of all, since any $(\phi^\delta, \di^\delta, \mu^\delta, \hh^\delta)$ is obtained from the fixed point argument in Step 1, we deduce from \eqref{eq:thrid_est_esistenza} that 
\begin{equation}
    \label{est:vv:0}
    \begin{split}
\norm{\di^\delta}_{L^\infty(0,T_\star^\delta;L^2(\Omega))} + \sqrt{\kappa} \norm{\nabla \di^\delta}_{L^2(0,T_\star^\delta;L^2(\Omega))} 
\leq 3 \norm{\di_0}_{L^2(\Omega)} + 3\sqrt{\alpha |\Omega| T_\star^\delta} =:\mathcal{D}_2^\star.
\end{split}
\end{equation}
Notice that $\mathcal{D}_2^\star$ is independent of $\delta$ and grows sublinearly in $T_\star^\delta$. We also remind that $|\di^\delta(x,t)|\leq \mathcal{D}_\infty$ almost everywhere in $\Omega \times (0,T_\star^\delta)$, where $\mathcal{D}_\infty$ is independent of $\delta$ and $T_\star^\delta$ (cf. \eqref{Ex:d:Linf}).

Next, we aim to carry out the main energy estimate. Setting 
$
F(s)= \frac{\Theta}{2}\left( (1+s)\ln(1+s)+(1-s)\ln(1-s) \right)$, for $s \in[-1,1]$,
and 
$$
L^2_{(c)}(\Omega)= \left\lbrace f \in L^2(\Omega): \frac{1}{|\Omega|} \int_\Omega f(x) \, \d x=c \right\rbrace,
\quad \text{for } c \in \R.
$$
For any $c \in (-1,1)$, we introduce the functional $E_{\rm{conv}}(\psi): L^2(\Omega)\to \R$ given by
\begin{equation}
\label{E-free0}
E_{\rm{conv}}(\psi)
= 
\begin{cases}
\displaystyle \int_\Omega \frac{\gamma}{4} \abs{\nabla \psi}^4 +\frac{\varepsilon}{2}\abs{\nabla \psi}^2 + \frac{1}{\varepsilon} F(\psi) \, \d x, \quad &\text{if } \, \psi \in W^{1,4}(\Omega): |\psi(x)| \leq 1 \ \text{in} \ \Omega,
\\[5pt]
+ \infty, \quad &\text{otherwise}.
\end{cases}
\end{equation}
The functional $E_{\rm{conv}}(\psi)$ is proper, lower semi-continuous and convex. Its subgradient $\partial E_{\rm{conv}}$ is a maximal monotone operator on $L_{(0)}^2(\Omega)$. 
Furthermore, we have the characterization 
\begin{equation}
\label{D-pE-free0}
D(\partial E_{\rm{conv}})
=
\left\lbrace 
\psi \in H^2(\Omega)\cap L^2_{(c)}(\Omega): |\nabla \psi|^2 \nabla \psi \in H^1(\Omega), F'(\psi)\in L^2(\Omega), 
F''(\psi)|\nabla \psi|^2 \in L^1(\Omega)
\right\rbrace
\end{equation}
and 
\begin{equation}
\label{pE-free0}
\partial E_{\rm{conv}}(\psi)
= - \diver \left( (\varepsilon+ \gamma |\nabla \psi|^2) \nabla \psi \right) 
+ \frac{1}{\varepsilon} P_0 F'(\psi),
\end{equation}
where $P_0 f:= f-\overline{f}$ denotes the orthogonal projection onto $L_{(0)}^2(\Omega)$. This can be proven by following \cite[Example 4, Chapter 2.2, page 59]{barbu2010nonlinear} and \cite[Section 4]{abels2007convergence} together with \cite[Theorem 2.6]{CM2019}. Then, we compute
\begin{align*}
\left(\mu^\delta, \partial_t \phi^\delta\right)
&= \left(\delta \partial_t \phi^\delta
    - \diver \left( (\varepsilon+ \gamma |\nabla \phi^\delta|^2) \nabla \phi^\delta \right) 
    + \frac{1}{\varepsilon} \Psi'(\phi^\delta)
    -\frac{\alpha}{2}  |\di^\delta|^2 
    - \beta \diver \left( (\nabla \phi^\delta \cdot \di^\delta) \di^\delta \right), \partial_t \phi^\delta\right)
\\
&=\delta \| \partial_t \phi^\delta\|_{L^2(\Omega)}^2 +
\left(- \diver \left( (\varepsilon+ \gamma |\nabla \phi^\delta|^2) \nabla \phi^\delta \right) + \frac{1}{\varepsilon} F'(\phi^\delta), \partial_t \phi^\delta\right)
-\frac{\Theta_0}{\varepsilon} (\phi^\delta, \partial_t \phi^\delta)
\\
&\quad - \frac{\alpha}{2}  (|\di^\delta|^2 , \partial_t \phi^\delta)
    - \beta (\diver \left( (\nabla \phi^\delta \cdot \di^\delta) \di^\delta \right), \partial_t \phi^\delta).
\end{align*}
Recalling that $\phi^\delta \in W^{1,2}(0,T_\star^\delta;L^2(\Omega))$ with $\overline{\phi^\delta}(t)=\overline{\phi_0}\in (-1,1)$ for all $t\in [0,T_\star^\delta)$ and $- \diver \left( (\varepsilon+ \gamma |\nabla \phi^\delta|^2) \nabla \phi^\delta \right) + \frac{1}{\varepsilon} F'(\phi^\delta) \in L^2(0,T_\star^\delta;L^2(\Omega))$, and owing to the characterization \eqref{D-pE-free0}-\eqref{pE-free0} of the subgradient of $E_{\rm{conv}}$, an application of \cite[Chapter IV, Lemma 4.3]{S2013} yields
\begin{equation}
\label{mu-ptphi:1}
\begin{split}
\left(\mu^\delta, \partial_t \phi^\delta\right)
&= \delta \| \partial_t \phi^\delta\|_{L^2(\Omega)}^2 
+ \dt
\left[\int_\Omega \frac{\gamma}{4} \abs{\nabla \phi^\delta}^4 +\frac{\varepsilon}{2}\abs{\nabla \phi^\delta}^2 + \frac{1}{\varepsilon} F(\phi^\delta) \, \d x \right]
-\frac{\Theta_0}{\varepsilon} (\phi^\delta, \partial_t \phi^\delta)
\\
&\quad - \frac{\alpha}{2}  (|\di^\delta|^2 , \partial_t \phi^\delta)
    - \beta (\diver \left( (\nabla \phi^\delta \cdot \di^\delta) \di^\delta \right), \partial_t \phi^\delta).
    \end{split}
\end{equation}
Furthermore, by the classical chain rule in $W^{1,2}(0,T_\star^\delta;L^2(\Omega))$, we obtain
\begin{equation}
\label{mu-ptphi:2}
\begin{split}
\left(\mu^\delta, \partial_t \phi^\delta\right)
&= \delta \| \partial_t \phi^\delta\|_{L^2(\Omega)}^2 
+ \dt
\left[\int_\Omega \frac{\gamma}{4} \abs{\nabla \phi^\delta}^4 +\frac{\varepsilon}{2}\abs{\nabla \phi^\delta}^2 + \frac{1}{\varepsilon} \Psi(\phi^\delta) \, \d x \right]
\\
&\quad - \frac{\alpha}{2}  (|\di^\delta|^2 , \partial_t \phi^\delta)
    - \beta (\diver \left( (\nabla \phi^\delta \cdot \di^\delta) \di^\delta \right), \partial_t \phi^\delta).
    \end{split}
\end{equation}
Now, multiplying \eqref{VCH-LC:wp1} by $\mu^\delta$ and \eqref{VCH-LC:wp3} by $\vect{h}^\delta$, using respectively \eqref{eq:VCH_2}, \eqref{eq:LC_2} and \eqref{mu-ptphi:2}, integrating over the domain $\Omega$ and adding the two resulting equations, we obtain
$$
\begin{aligned}
\dt&\left[\int_\Omega \prt{\frac{\gamma}{4} \abs{\nabla \phi^\delta}^4 +\frac{\varepsilon}{2}\abs{\nabla \phi^\delta}^2 + \frac{1}{\varepsilon} \Psi(\phi^\delta) + \frac{\kappa}{2}\abs{\nabla \di^\delta}^2 + \frac{\alpha}{4}\abs{\di^\delta}^4}\, \d x\right]+ \norm{\nabla \mu^\delta}_{L^2(\Omega)}^2 + \delta \norm{\partial_t \phi^\delta}_{L^2(\Omega)}^2+ \norm{\vect{h}^\delta}_{L^2(\Omega)}^2 \\
&-\frac{\alpha}{2}\int_\Omega \abs{\di^\delta}^2 \partial_t \phi^\delta\, \d x -\alpha \int_\Omega \prt{\phi^\delta-\phi_{\rm cr}} \di^\delta \cdot \partial_t \di^\delta\, \d x -\beta \int_\Omega \diver\prt{(\nabla \phi^\delta\cdot \di^\delta)\di^\delta}\partial_t \phi^\delta\, \d x \\
&+\beta \int_\Omega \prt{\di^\delta \cdot \nabla \phi^\delta}\nabla \phi^\delta \cdot \partial_t \di^\delta\, \d x=
\int_\Omega\vv \cdot \nabla \mu^\delta \phi^\delta\, \d x
- \int_\Omega \prt{\vv \cdot \nabla}\di^\delta \cdot \vect{h}^\delta\, \d x.
\end{aligned}
$$
Here, we have also used the chain rule in $L^2(0,T_\star^\delta;H^2(\Omega))\cap W^{1,2}(0,T_\star^\delta;L^2(\Omega))$ for the $\di^\delta$ terms in first integral of the above equation. On the other hand, we can rewrite
$$
\begin{aligned}
-\frac{\alpha}{2}\int_\Omega \abs{\di^\delta}^2 \partial_t \phi^\delta\, \d x -\alpha \int_\Omega \prt{\phi^\delta-\phi_{\rm cr}} \di^\delta \cdot \partial_t \di^\delta\, \d x  &= \dt \left[-\int_\Omega \frac{\alpha}{2} \abs{\di^\delta}^2 (\phi^\delta-\phi_{\rm cr})\, \d x\right],
\end{aligned}
$$
and 
$$
\begin{aligned}
-\beta \int_\Omega \diver\prt{(\nabla \phi^\delta\cdot \di^\delta)\di^\delta}\partial_t \phi^\delta\, \d x +\beta \int_\Omega \prt{\di^\delta \cdot \nabla \phi^\delta}\nabla \phi^\delta \cdot \partial_t \di^\delta\, \d x&
= \dt \left[\int_\Omega \frac{\beta}{2} \abs{\nabla \phi^\delta \cdot \di^\delta}^2\, \d x\right].
\end{aligned}
$$
Thus, we infer that
\begin{equation}
\label{EEE1}
\begin{aligned}
\dt \mathrm{E}_{\rm free}^\gamma(\phi^\delta, \di^\delta) + \norm{\nabla \mu^\delta}_{L^2(\Omega)}^2 + 
\delta \norm{\partial_t \phi^\delta}_{L^2(\Omega)}^2 + \norm{\vect{h}^\delta}_{L^2(\Omega)}^2 = \int_\Omega \phi^\delta \vv \cdot \nabla \mu^\delta\, \d x - \int_\Omega \prt{\vv \cdot \nabla}\di^\delta \cdot \vect{h}^\delta\, \d x.
\end{aligned}
\end{equation}
We now proceed with the control of the terms on the right-hand side. By using \eqref{GN}, \eqref{Rev-SI} and \eqref{VCH-LC:reg2}, we have
$$
\begin{aligned}
\abs{\int_\Omega \phi^\delta \vv \cdot \nabla \mu^\delta\, \d x}&\leq \norm{\phi^\delta}_{L^\infty(\Omega)}\norm{\vv}_{L^2(\Omega)}\norm{\nabla \mu^\delta}_{L^2(\Omega)}\leq \frac{1}{2}\norm{\nabla \mu^\delta}_{L^2(\Omega)}^2 + \frac{1}{2}\norm{\vv}_{L^2(\Omega)}^2,
\end{aligned}
$$
and
$$
\begin{aligned}
&\abs{\int_\Omega \prt{\vv \cdot \nabla}\di^\delta \cdot \vect{h}^\delta\, \d x}
\\
&= \int_\Omega \prt{\vv\cdot \nabla}\di^\delta\cdot \left[\kappa \Delta \di^\delta -\alpha \abs{\di^\delta}^2 \di^\delta + \alpha \prt{\phi^\delta-\phi_{\rm cr}} \di^\delta -\beta \prt{\di^\delta \cdot \nabla \phi^\delta}\nabla \phi^\delta\right]\, \d x\\
&= \int_\Omega \kappa \prt{\vv \cdot \nabla}\di^\delta \cdot \Delta \di^\delta \, \d x + \alpha \int_\Omega \prt{\vv \cdot \nabla}\di^\delta\cdot \prt{\phi^\delta-\phi_{\rm cr}}\di^\delta\, \d x -\beta \int_\Omega \prt{\vv \cdot \nabla}\di^\delta \cdot \prt{\di^\delta \cdot \nabla \phi^\delta}\nabla \phi^\delta\, \d x\\
&\leq \kappa \norm{ \nabla \vv}_{L^\infty(\Omega)}\norm{\nabla \di^\delta}_{L^2(\Omega)}^2 + \alpha \norm{\vv}_{L^2(\Omega)}\norm{\nabla \di^\delta}_{L^2(\Omega)}\norm{\phi^\delta-\phi_{\rm cr}}_{L^\infty(\Omega)} \norm{\di^\delta}_{L^\infty(\Omega)} \\
&\quad+ \beta \norm{\vv}_{L^\infty(\Omega)}\norm{\nabla \di^\delta}_{L^2(\Omega)} \norm{\di^\delta}_{L^\infty(\Omega)}\norm{\nabla \phi^\delta}_{L^4(\Omega)}^2\\
&\leq C_m \norm{\vv}_{L^2(\Omega)}^2 +  C \prt{1+ \norm{\nabla \di^\delta}_{L^2(\Omega)}^2}\norm{\nabla \di^\delta}_{L^2(\Omega)}^2 
+ C_m \mathcal{D}_\infty \norm{\vv}_{L^2(\Omega)}\norm{\nabla \di^\delta}_{L^2(\Omega)} \norm{\nabla \phi^\delta}_{L^4(\Omega)}^2\\
&\leq C_m \norm{\vv}_{L^2(\Omega)}^2 +  C \prt{1+ \norm{\nabla \di^\delta}_{L^2(\Omega)}^2} 
\left( \norm{\nabla \di^\delta}_{L^2(\Omega)}^2 
+ \norm{\nabla \phi^\delta}_{L^4(\Omega)}^4 \right).
\end{aligned}
$$
Hence, we arrive at
\begin{equation}
    \label{eq:der_energy}
    \begin{aligned}
        \dt \left(\mathrm{E}^\gamma_{\rm free}(\phi^\delta, \di^\delta) -(2\pi)^nE_0\right) &+ \delta \norm{\partial_t \phi^\delta}_{L^2(\Omega)}^2 + \frac{1}{2}\norm{\nabla \mu^\delta}_{L^2(\Omega)}^2 + \norm{\vect{h}^\delta}_{L^2(\Omega)}^2\\
    &\leq C \norm{\vv}_{L^2(\Omega)}^2 +  C \prt{1+ \norm{\nabla \di^\delta}_{L^2(\Omega)}^2}
    \left(\mathrm{E}^\gamma_{\rm free}(\phi^\delta, \di^\delta) -(2\pi)^nE_0\right),
    \end{aligned}
\end{equation}
where the positive $C$ depends on $\alpha, \beta, \gamma, \varepsilon, \kappa, n, m, \mathcal{D}_\infty$, but it is independent of $\delta$, and $E_0$ is defined in Section \ref{sub:mr}.
We deduce from the Gronwall lemma together with \eqref{est:vv:0} that
\begin{equation}
    \label{EE:vv}
    \begin{split}
       & 
       \left( E^\gamma_{\rm{free}}(\phi^\delta(t),\di^\delta(t)) - (2\pi)^nE_0\right)
       + \delta \int_0^t  \norm{\partial_t \phi^\delta(s)}_{L^2(\Omega)}^2 \, \d s + \frac{1}{2} \int_0^t \norm{\nabla \mu^\delta(s)}_{L^2(\Omega)}^2 \, \d s
       + \int_0^t \norm{\vect{h}^\delta(s)}_{L^2(\Omega)}^2 \, \d s \\
       &\quad \leq 2 \left( E^\gamma_{\rm{free}}(\phi_0,\di_0) -(2\pi)^nE_0+ C \int_0^t \|\vv(s) \|_{L^2(\Omega)}^2 \, \d s \right)
       \mathrm{exp}\left(C \int_0^t \left( 1+ \norm{\nabla \di^\delta(s)}_{L^2(\Omega)}^2 \right) \, \d s\right)
       \\
       &\quad \leq 2 \left( E^\gamma_{\rm{free}}(\phi_0,\di_0) -(2\pi)^nE_0 + C \int_0^{t} \|\vv(s) \|_{L^2(\Omega)}^2 \, \d s \right)
       \mathrm{exp}\left(CT_\star^\delta 
       + C \frac{\textcolor{black}{\left(\mathcal{D}_2^\star\right)^2}}{\kappa}\right), \quad \forall \, t \in [0,T_\star^\delta].
    \end{split}
\end{equation}
Recalling the uniform bounds
\begin{equation}
    \label{est:vv:2}
    \| \phi^\delta\|_{L^\infty(\Omega \times (0,T_\star^\delta))}\leq 1, \quad 
    \| \di^\delta\|_{L^\infty(\Omega \times (0,T_\star^\delta))}\leq \mathcal{D}_\infty,
\end{equation}
and owing to \eqref{EE:vv}, we obtain
\begin{equation}
\label{est:vv:1}
\begin{aligned}
&\| \phi^\delta\|_{L^\infty(0,T_\star^\delta;W^{1,4}(\Omega))}
+ \| \di^\delta\|_{L^\infty(0,T_\star^\delta; H^1(\Omega))}\leq  C \mathrm{exp}(C T_\star^\delta)
\\
& \| \nabla \mu^\delta\|_{L^2(0,T_\star^\delta;L^2(\Omega))}+ \sqrt{\delta} \| \partial_t \phi^\delta \|_{L^2(0,T_\star^\delta;L^2(\Omega))}
+ \| \vect{h}^\delta\|_{L^2(0,T_\star^\delta;L^2(\Omega))}
\leq C \mathrm{exp}(C T_\star^\delta),
\end{aligned}
\end{equation}
for some positive constant $C$ depending on $\alpha$, $\beta$, $\gamma$, $\varepsilon$, $\kappa$, $n$, $m$, $\mathcal{D}_\infty$, $E_{\rm free}^\gamma(\phi_0,\di_0)$ and $\| \vv\|_{L^2(0,T_\star^\delta; L^2(\Omega))}$, but it is independent of $\delta$. 
Let us now consider 
$\vect{h}^\delta = \kappa \Delta \vect{d}^\delta -\alpha \abs{\vect{d}^\delta}^2 \vect{d}^\delta + \alpha (\phi^\delta-\phi_{\rm cr})\vect{d}^\delta - \beta \left(\vect{d}^\delta\cdot \nabla \phi^\delta\right)\nabla \phi^\delta$.
By \eqref{est:vv:2} and \eqref{est:vv:1}, 
we notice that
$$
\int_0^{T_\star^\delta} \norm{(\di^\delta(s) \cdot \nabla \phi^\delta(s))\nabla \phi^\delta(s)}_{L^2(\Omega)}^2 \, \d s 
\leq 
\int_0^{T_\star^\delta} \norm{\di(s)}_{L^\infty(\Omega)}^2 \norm{\nabla \phi^\delta(s)}_{L^4(\Omega)}^4 \, \d s \leq C T_\star^\delta \mathrm{exp}(C T_\star^\delta).
$$
Then, exploiting once again \eqref{est:vv:2} and \eqref{est:vv:1}, we find 
\begin{equation}
\label{est:vv:3}
\|\di^\delta\|_{L^2(0,T_\star^\delta;H^2(\Omega))}\leq C(1+T_\star^\delta)^\frac12 \mathrm{exp}(C T_\star^\delta).
\end{equation}
In addition, by \eqref{VCH-LC:wp3}
\begin{align*}
\| \partial_t \di^\delta\|_{L^2(0,T_\star^\delta;L^2(\Omega))}
&\leq C_m\| \vv\|_{L^2(0,T_\star^\delta;L^2(\Omega))}
\| \nabla \di^\delta\|_{L^\infty(0,T_\star^\delta;L^2(\Omega))}
+ \| \vect{h}^\delta\|_{L^2(0,T_\star^\delta;L^2(\Omega))},
\end{align*}
which, in turn, gives
\begin{equation}
    \label{est:vv:4}
    \| \partial_t \di^\delta\|_{L^2(0,T_\star^\delta;L^2(\Omega))} 
    \leq C \mathrm{exp}(C T_\star^\delta).
\end{equation}
Next, we observe that 
\begin{equation}
\label{media-mu-d}
{\color{black}
\overline{\mu^\delta}=
\frac{1}{\varepsilon}  \overline{\Psi'(\phi^\delta)} - \frac{\alpha}{2 \abs{\Omega}}  \| \di^\delta \|_{L^2(\Omega)}^2.
}
\end{equation}
Testing \eqref{eq:def_mu} by $\phi^\delta-\overline{\phi^\delta}$, we have
$$
\begin{aligned}
    &\int_{\Omega} \varepsilon |\nabla \phi^\delta|^2 + \gamma | \nabla \phi^\delta|^4 + \beta |\nabla \phi^\delta\cdot \di^\delta|^2 \, \d x 
+ {\color{black} \frac{1}{\varepsilon}} \int_{\Omega} F'(\phi^\delta) (\phi^\delta-\overline{\phi^\delta}) \, \d x  \\
&= - \delta \int_\Omega \partial_t \phi^\delta(\phi^\delta-\overline{\phi^\delta})\, \d x + \int_{\Omega} \mu^\delta (\phi^\delta-\overline{\phi^\delta}) \, \d x 
+ \Theta_0 \int_{\Omega} \phi^\delta (\phi^\delta-\overline{\phi^\delta}) \, \d x
- \frac{\alpha}{2} \int_{\Omega} |\di^\delta|^2 (\phi^\delta-\overline{\phi^\delta})\, \d x.
\end{aligned}
$$
Thanks to \eqref{H1:equiv}, we get
\begin{align*}
{\color{black} \frac{1}{\varepsilon}}\int_{\Omega} F'(\phi^\delta) (\phi^\delta-\overline{\phi^\delta}) \, \d x
&\leq \delta\norm{\partial_t \phi^\delta}_{L^2(\Omega)} \norm{\phi^\delta- \overline{\phi^\delta}}_{L^2(\Omega)} +C\| \nabla \mu^\delta\|_{L^2(\Omega)} \| \phi^\delta - \overline{\phi^\delta}\|_{L^2(\Omega)} \\
&\quad +C \|\phi^\delta - \overline{\phi^\delta} \|_{L^2(\Omega)}^2 + C \|\di^\delta\|_{L^2(\Omega)}^2 \| \phi^\delta\|_{L^\infty(\Omega)}\\
&\leq  \delta C \norm{\partial_t \phi^\delta}_{L^2(\Omega)}+ C\| \nabla \mu^\delta\|_{L^2(\Omega)}+C.
\end{align*}
Recalling the well-known inequality in \cite[Proposition A.1]{miranville2004robust}
$$
\int_{\Omega} |F'(\phi^\delta)|\, \d x
\leq C_1 \int_{\Omega} F'(\phi^\delta) (\phi^\delta-\overline{\phi}) \, \d x+ C_2,
$$
where $C_1$ and $C_2$ are positive constants depending on $\Omega$ and $\overline{\phi_0}$, we infer from \eqref{est:vv:1} that
$$
\int_0^{T_\star^\delta} \| F'(\phi^\delta(s))\|_{L^1(\Omega)}^2 \, \d s
\leq C T_\star^\delta +C \int_0^{T_\star^\delta}
\delta\| \partial_t \phi^\delta\|_{L^2(\Omega)}^2
+\| \nabla \mu^\delta(s)\|_{L^2(\Omega)}^2  \, \d s
\leq C(1+T_\star^\delta)\mathrm{exp}(C T_\star^\delta).
$$
As a consequence, we deduce from \eqref{H1:equiv} and \eqref{media-mu-d} that 
\begin{equation}
\label{est:vv:5}
\| \mu^\delta\|_{L^2(0,T_\star^\delta; H^1(\Omega))}\leq C(1+T_\star^\delta)^\frac12 \mathrm{exp}(C T_\star^\delta).
\end{equation}
For any $k \in \mathbb{N}$, let us introduce
\begin{equation}
h_{k}:\mathbb{R}\rightarrow \mathbb{R},
\quad 
{\color{black}
h_{k}(s)=%
\begin{cases}
-k, & s<-k, \\
s, & s\in \lbrack -k,k], \\
k, & s>k,%
\end{cases}
}
\quad \text{and}
\quad 
F_k': (-1,1) \rightarrow \R, 
   \quad
F_k'(s)= h_k(s) \circ F'(s).
\label{trunc_F}
\end{equation}
We multiply \eqref{VCH-LC:wp2} by $F'_k(\phi^\delta)$ and integrate over $\Omega$. 
Since $F'_k(\phi^\delta) \in W^{1,2}(\Omega)$ almost everywhere in $(0,T_\star^\delta)$ and $\nabla F'_k(\phi^\delta)= F''_k(\phi^\delta) \nabla \phi^\delta$ almost everywhere in $\Omega \times (0,T_\star^\delta)$, we have 
\begin{equation*}
\begin{split}
     &\int_{\Omega} F_k''(\phi^\delta)\left[\varepsilon \abs{\nabla \phi^\delta}^2 + \gamma \abs{\nabla \phi^\delta}^4 + \beta\abs{\nabla \phi^\delta \cdot \di^\delta}^2\right]\, \d x +  {\color{black} \frac{1}{\varepsilon}} \int_\Omega F'(\phi^\delta) F'_k(\phi^\delta) \, \d x \\
     & \quad
     \leq  \int_\Omega \left( \mu^\delta {\color{black} - \delta \partial_t \phi^\delta} +\Theta_0 \phi^\delta + \frac{\alpha}{2} |\di^\delta|^2 \right) F'_k(\phi^\delta) \, \d x.
    \end{split}
\end{equation*}
Observing that $F'(s)s\geq 0$, we infer that
\begin{equation}
    \frac{1}{\varepsilon} \int_0^{T_\star^\delta} \int_\Omega |F'_k(\phi^\delta(s))|^2 \, \d x \, \d s
    \leq 
    {\color{black} 2\varepsilon 
    \int_0^{T_\star^\delta}
    \| \mu^\delta(s)\|_{L^2(\Omega)}^2 
    + \delta^2 \| \partial_t \phi\|_{L^2(\Omega)}^2}
    + \Theta_0^2 \| \phi^\delta(s)\|_{L^2(\Omega)}^2 
    + \alpha^2 \| \di^\delta(s)\|_{L^4(\Omega)}^4 \, \d s.
\end{equation}
In light of $|\phi^\delta(x,t)|<1$ almost everywhere in $\Omega\times (0,T_\star^\delta)$, $F_k'(\phi^\delta)\rightarrow F'(\phi^\delta)$ almost everywhere in $\Omega\times (0,T_\star^\delta)$. Thus, the Fatou lemma {\color{black} and the estimates \eqref{est:vv:1} and \eqref{est:vv:5}, yield } that 
\begin{equation}
\label{est:vv:6}
    \int_0^{T_\star^\delta} \int_\Omega |F'(\phi^\delta(s))|^2 \, \d x \, \d s \leq C(1+T_\star^\delta)\,\mathrm{exp}(C T_\star^\delta).
\end{equation}
Now, multiplying \eqref{VCH-LC:wp2} by $-\Delta \phi^\delta$ and integrating over $\Omega$, we find
\begin{equation}
\label{est:vv:7}
\begin{split}
&
\frac{\delta}{2}\dt \| \nabla \phi^\delta\|_{L^2(\Omega)}^2
+\varepsilon\norm{\Delta \phi^\delta}_{L^2(\Omega)}^2 + \gamma \int_\Omega \diver \left( |\nabla \phi^\delta |^2 \nabla \phi^\delta \right) \Delta \phi^\delta \, \d x - \frac{1}{\varepsilon} \int_\Omega F'(\phi^\delta) \Delta \phi^\delta \, \d x\\
&
= \int_\Omega \nabla \mu^\delta \cdot \nabla \phi^\delta \, \d x
+ \frac{\Theta_0}{\varepsilon} 
\int_{\Omega} |\nabla \phi^\delta|^2 \, \d x
+\frac{\alpha}{2}\int_\Omega \Delta \phi^\delta \abs{\di^\delta}^2\, \d x -
\beta \int_\Omega \diver\prt{(\nabla \phi^\delta \cdot \di^\delta)\di^\delta} \Delta \phi^\delta \, \d x.
\end{split}
\end{equation}
For any $k \in \N$, let $F'_k(s)$ be the function defined in \eqref{trunc_F}. We observe that
\begin{equation}
\label{F2:pos}
\begin{split}
-\frac{1}{\varepsilon}\int_\Omega F'(\phi^\delta) \Delta \phi^\delta \, \d x 
&=-\frac{1}{\varepsilon}\int_\Omega F_k'(\phi^\delta) \Delta \phi^\delta \, \d x 
-\frac{1}{\varepsilon}\int_\Omega \left( F'(\phi^\delta)
-F'_k(\phi^\delta) \right) \Delta \phi^\delta \, \d x 
\\
&=\frac{1}{\varepsilon}\int_\Omega F_k''(\phi^\delta) |\nabla \phi^\delta|^2 \, \d x 
-\frac{1}{\varepsilon}\int_\Omega \left( F'(\phi^\delta)
-F'_k(\phi^\delta) \right) \Delta \phi^\delta \, \d x
\\
&\geq -\frac{1}{\varepsilon}\int_\Omega \left( F'(\phi^\delta)
-F'_k(\phi^\delta) \right) \Delta \phi^\delta \, \d x.
\end{split}
\end{equation}
{\color{black} Since $|F_k'(\phi^\delta)| \nearrow |F'(\phi^\delta)|$ almost everywhere} in $\Omega$ and $F'(\phi^\delta) \in L^2(\Omega)$ for almost any $t\in (0,T)$, it easily follows that $\int_\Omega \left( F'(\phi^\delta)
-F'_k(\phi^\delta) \right) \Delta \phi^\delta \, \d x\rightarrow 0$ as $k \rightarrow \infty$. Thus, we conclude that $-\frac{1}{\varepsilon}\int_\Omega F'(\phi^\delta) \Delta \phi^\delta \, \d x \geq 0$.
On the other hand, integrating by parts, we notice that
\begin{equation}
\label{gamma:term}
\begin{split}
    \int_\Omega \diver\left(\abs{\nabla \phi^\delta}^2 \nabla \phi^\delta\right)\Delta \phi^\delta\, \d x  &= \int_{\Omega} \partial_i\left(\abs{\nabla \phi^\delta}^2 \partial_i\phi^\delta\right)\partial_{jj}\phi^\delta \, \d x = \int_{\Omega} \partial_j(\abs{\nabla \phi^\delta}^2 \partial_i \phi^\delta) \partial_j\partial_i \phi^\delta \, \d x
    \\
   &=\int_{\Omega} \left(\partial_j \partial_i\phi^\delta \partial_j \partial_i \phi^\delta \abs{\nabla \phi^\delta}^2 + 2 \partial_j\partial_i \phi^\delta \partial_i \phi^\delta \partial_j\partial_l\phi^\delta \partial_l\phi^\delta \right)\, \d x 
    \\
   &=  \int_{\Omega} \abs{D^2 \phi^\delta}^2 \abs{\nabla \phi^\delta}^2 + \frac{1}{4} \left| \nabla\left(\abs{\nabla \phi^\delta}^2\right) \right|^2 \, \d x,
   \end{split}
\end{equation}
and
\begin{equation}
\label{beta:term}
\begin{split}
     \int_\Omega \diver\prt{(\nabla \phi^\delta \cdot \di^\delta)\di^\delta} \Delta \phi^\delta \, \d x 
    &=  \int_\Omega \partial_l \left( (\nabla \phi^\delta \cdot \di^\delta)\di^\delta_j \right) \partial_l \partial_j \phi^\delta \, \d x
    \\
    &= \int_\Omega \partial_l (\nabla \phi^\delta \cdot \di^\delta) \di^\delta_j \partial_l \partial_j \phi^\delta \, \d x
    + \int_\Omega (\nabla \phi^\delta \cdot \di^\delta) \partial_l \di^\delta_j \partial_l \partial_j \phi^\delta \, \d x
    \\
    &= -  \int_\Omega \nabla(\nabla \phi^\delta \cdot \di^\delta) \cdot (\nabla (\di^\delta)^T \nabla \phi^\delta) \, \d x
    +  \int_\Omega (\nabla \phi^\delta \cdot \di^\delta) \nabla \di^\delta : D^2 \phi^\delta \, \d x\\
    &\quad +\underbrace{\int_\Omega |\nabla(\nabla \phi^\delta \cdot \di^\delta)|^2 \, \d x}_{\geq 0}.
    \end{split}
\end{equation}
A straightforward computation shows that
\begin{align}
\label{est:vv:8}
\left| \frac{\alpha}{2}\int_\Omega \Delta \phi^\delta \abs{\di^\delta}^2\, \d x \right| 
&
\leq \frac{\varepsilon}{4}
\| \Delta \phi^\delta\|_{L^2(\Omega)}^2 
+ \frac{\alpha^2 \abs{\Omega} C(n)}{4 \varepsilon} \mathcal{D}_\infty^4,\\
\label{est:vv:9}
\left| \beta \int_\Omega (\nabla \phi^\delta \cdot \di^\delta) \nabla \di^\delta : D^2 \phi^\delta \, \d x \right|
       &\leq \frac{\gamma}{4} \int_\Omega \abs{D^2 \phi^\delta}^2 \abs{\nabla \phi^\delta}^2 \, \d x
       + \frac{\beta^2 C(n)}{\gamma} D_\infty^2 \| \nabla \di^\delta\|_{L^2(\Omega)}^2,
\end{align}
and 
\begin{equation}
\label{est:vv:10}
    \left| \left( \frac{\Theta_0}{\varepsilon}+ \frac12 \right) \norm{\nabla \phi^\delta}_{L^2(\Omega)}^2 \right| \leq \frac{\varepsilon}{4}\| \Delta \phi^\delta\|_{L^2(\Omega)}^2 + \frac{1}{\varepsilon}\left( \frac{\Theta_0}{\varepsilon}+ \frac12 \right)^2 |\Omega|.
\end{equation}
By \eqref{H2:equiv} and \eqref{GN}, we have
\begin{equation}
\label{est:vv:11}
    \begin{split}
         \left| \beta  \int_\Omega \nabla(\nabla \phi^\delta \cdot \di^\delta) \cdot (\nabla (\di^\delta)^T \nabla \phi^\delta) \, \d x \right| 
         &=
         \beta \left| \int_\Omega \partial_l \partial_i \phi^\delta \di^\delta_i \partial_l \di^\delta_j \partial_j \phi^\delta 
         + \partial_i \phi^\delta \partial_l \di^\delta_i \partial_l \di^\delta_j \partial_j \phi^\delta \, \d x
         \right|
         \\
         &\leq \frac{\gamma}{4} \int_\Omega \abs{D^2 \phi^\delta}^2 \abs{\nabla \phi^\delta}^2 \, \d x
       + \frac{\beta^2 C(n)}{\gamma} D_\infty^2 \| \nabla \di^\delta\|_{L^2(\Omega)}^2 
       \\
       &\quad + \beta \| \nabla \di^\delta\|_{L^4(\Omega)}^2 
       \| \nabla \phi^\delta\|_{L^4(\Omega)}^2
       \\
       &\leq \frac{\gamma}{4} \int_\Omega \abs{D^2 \phi^\delta}^2 \abs{\nabla \phi^\delta}^2 \, \d x
       + \frac{\beta^2 C(n)}{\gamma} D_\infty^2 \| \nabla \di^\delta\|_{L^2(\Omega)}^2 
       \\
       &\quad + \frac{\varepsilon}{4} \norm{\Delta \phi^\delta}_{L^2(\Omega)}^2 
       + C \| \di^\delta\|_{H^2(\Omega)}^2 
       +C |\overline{\phi_0}|^2. 
    \end{split}
\end{equation}
Thus, we infer that 
\begin{equation}
\label{est:vv:12}
\begin{split}
& 
 \frac{\delta}{2}\dt \| \nabla \phi^\delta\|_{L^2(\Omega)}^2
+\frac{\varepsilon}{4} \norm{\Delta \phi^\delta}_{L^2(\Omega)}^2 + \frac{\gamma}{2} 
\int_{\Omega} \abs{D^2 \phi^\delta}^2 \abs{\nabla \phi^\delta}^2 \, \d x
+\frac{\gamma}{4} \int_\Omega \left| \nabla\left(\abs{\nabla \phi^\delta}^2\right) \right|^2 \, \d x
\\
&\leq 
\frac12 \norm{\nabla \mu^\delta}_{L^2(\Omega)}^2
+ C \| \di^\delta\|_{H^2(\Omega)}^2+C,
\end{split}
\end{equation}
where the positive constants $C$ depends only on $\beta$, $\gamma$, $\varepsilon$, $n$ and $\mathcal{D}_\infty$. Integrating in time and exploiting \eqref{est:vv:1}-\eqref{est:vv:3}, we find
\begin{equation}
\label{est:vv:13}
    \frac{\varepsilon}{4} \int_0^{T_\star^\delta}\norm{\Delta \phi^\delta(s)}_{L^2(\Omega)}^2 \, \d s 
    + \frac{\gamma}{4} \int_0^{T_\star^\delta} 
\int_{\Omega} \left| \nabla\left(\abs{\nabla \phi^\delta}^2\right) \right|^2 \, \d x \, \d s
\leq 
C(1+T_\star^\delta) \mathrm{exp}(C T_\star^\delta),
\end{equation}
for some positive constant $C$ as in \eqref{est:vv:1}.
Owing to \eqref{est:vv:2}-\eqref{est:vv:13}, and arguing by comparison in \eqref{VCH-LC:wp2}, it follows that 
$$
\| -\diver \left( (\varepsilon+ \gamma |\nabla \phi^\delta|^2) \phi^\delta\right)\|_{L^2(0,T^\delta_\star; L^2(\Omega))} \leq  C(1+T_\star^\delta)^\frac12 \mathrm{exp}(C T_\star^\delta).
$$ 
Then, by \cite[Theorem 2.6]{CM2019}, we conclude that 
\begin{equation}
    \label{est:vv:14}
    \norm{|\nabla \phi^\delta|^2 \nabla \phi^\delta}_{L^2(0,T_\star^\delta; H^1(\Omega;\R^n))}
    \leq  C(1+T_\star^\delta)^\frac12 \mathrm{exp}(C T_\star^\delta).
\end{equation}
Finally, it is easily seen from \eqref{VCH-LC:wp1} that 
\begin{equation}
\label{est:vv:15}
    \| \partial_t \phi^\delta\|_{L^2(0,T_\star^\delta;H^1(\Omega)')} 
    \leq 
    \| \vv\|_{L^2(0,T;L^2(\Omega))} 
    + \| \nabla \mu^\delta\|_{L^2(0,T_\star^\delta;L^2(\Omega))}
    \leq C(1+T_\star^\delta)^\frac12 \mathrm{exp}(C T_\star^\delta).
\end{equation}
In light of the above estimates \eqref{est:vv:2}, \eqref{est:vv:3}, \eqref{est:vv:4}, \eqref{est:vv:5}, \eqref{est:vv:13}, \eqref{est:vv:14} and \eqref{est:vv:15}, we have shown that all the norms corresponding to the function spaces in \eqref{VCH-LC:reg1}-\eqref{VCH-LC:reg6} remain bounded for any $T_\star^\delta$ finite, and thereby the quadruplet $(\phi^\delta, \di^\delta, \mu^\delta, \hh^\delta)$ can be extended on $[0,T]$. In particular, the estimates \eqref{est:vv:2}, \eqref{est:vv:3}, \eqref{est:vv:4}, \eqref{est:vv:5}, \eqref{est:vv:13}, \eqref{est:vv:14} and \eqref{est:vv:15} hold by replacing $T_\star^\delta$ with $T$. Thus, by classical compactness results we find (up to subsequences) that
\begin{equation}
\label{vv:conv1}
\begin{split}
\begin{aligned}
&\phi^\delta \rightharpoonup \phi \  &&\text{weak-star in } L^\infty(0,T;W^{1,4}(\Omega)),
\quad 
&&&\phi^\delta \rightharpoonup \phi \  &&&&\text{weak-star in } L^\infty(\Omega \times (0,T)),
\\
&\phi^\delta \rightharpoonup \phi \  &&\text{weakly in } 
L^2(0,T;H^2(\Omega)),\quad
&&&\phi^\delta \rightarrow \phi \
&&&&\text{strongly in }
L^2(0,T; W^{1,p}(\Omega)),
\\
&\partial_t \phi^\delta \rightharpoonup \partial_t \phi \ &&\text{weakly in } 
L^2(0,T; H^1(\Omega)'),
\quad 
&&&|\nabla \phi^\delta|^2 \rightharpoonup |\nabla \phi|^2 \
&&&&\text{weakly in } 
L^2(0,T;H^1(\Omega)),\\
&\mu^\delta \rightharpoonup \mu \ &&\text{weakly in } L^2(0,T;H^1(\Omega)), 
\quad 
&&&\di^\delta \rightharpoonup \di \ &&&&\text{weakly in } L^2(0,T;H^2(\Omega)),
\\
&\di^\delta \rightarrow \di \  &&\text{strongly in } L^2(0,T;H^1(\Omega)),
\quad 
&&&\partial_t \di^\delta \rightharpoonup \partial_t \di \  &&&&\text{weakly in } L^2(0,T;L^2(\Omega)),\\
&\hh^\delta \rightharpoonup \hh  \quad &&\text{weakly in } L^2(0,T;L^2(\Omega)),
\end{aligned}
\end{split}
\end{equation}
for $p<6$ if $d=3$ and any $p \in [2,\infty)$ if $d=2$, where 
\begin{equation}
    \label{est:vv:l1}
    \| \phi\|_{L^\infty(\Omega \times (0,T))}\leq 1, \quad 
    \| \di\|_{L^\infty(\Omega \times (0,T))}\leq \mathcal{D}_\infty.
\end{equation}
Thanks to the above convergences, we are in the position to pass to limit as $\delta\rightarrow 0$ in \eqref{VCH-LC:wp1}-\eqref{VCH-LC:wp4}. The case of \eqref{VCH-LC:wp3}-\eqref{VCH-LC:wp4} is standard. Concerning \eqref{VCH-LC:wp1}-\eqref{VCH-LC:wp2},
we infer from \eqref{vv:conv1} that $\phi^\delta  \rightarrow \phi$ almost everywhere in $\Omega \times (0,T)$, which entails that $F'(\phi^\delta) \rightarrow \widetilde{F'}(\phi)$ almost everywhere in $\Omega \times (0,T)$, where $\widetilde{F'}(s)=F'(s)$ if $s \in (-1,1)$ and $\widetilde{F'}(\pm 1)=\pm \infty$. Nevertheless, by the Fatou lemma and \eqref{est:vv:6}, $\int_{\Omega \times (0,T)} |\widetilde{F'}(\phi)|^2 \, \d x \d s\leq C$, which implies that
$\phi \in L^\infty(\Omega\times (0,T))$ such that $|\phi(x,t)| < 1$ for almost every $(x,t) \in \Omega\times(0,T)$.
Hence, $\widetilde{F'}(\phi)= F'(\phi)$ almost everywhere in $\Omega \times (0,T)$ and
$F'(\phi^\delta) \rightharpoonup F'(\phi)$  weakly in $L^2(0,T;L^2(\Omega))$.
Letting $\delta \rightarrow 0$ in \eqref{VCH-LC:wp1}-\eqref{VCH-LC:wp2}, we obtain 
\begin{subequations}
\begin{alignat}{2}
\label{CH-LC:wp1-pr}
    &\langle \partial_t \phi, \xi \rangle_{H^1(\Omega)}
    - (\phi \vv, \nabla \xi)
    + (\nabla \mu, \nabla \xi)
    =0, \\
\label{CH-LC:wp2-pr}
    &(\mu,\zeta)=
     \left( \left( (\varepsilon+ \gamma |\nabla \phi|^2) \nabla \phi \right), \nabla \zeta\right) 
    + \frac{1}{\varepsilon} \left( \Psi'(\phi), \zeta \right)
    -\frac{\alpha}{2} \left( |\di|^2,\zeta \right) 
    - \beta (\diver \left( (\nabla \phi \cdot \di) \di \right), \zeta),
\end{alignat}
\end{subequations}
for all $\xi, \zeta \in H^1(\Omega)$, almost every $t \in (0,T)$. 
On other hand, since 
$\mu - \frac{1}{\varepsilon} \Psi'(\phi)+ \frac{\alpha}{2}|\di|^2 +\beta \diver \left( (\nabla \phi \cdot \di) \di \right) \in L^2(0,T;L^2(\Omega)),
$
taking advantage of \cite[Theorem 2.6]{CM2019}, we have 
$|\nabla \phi|^2 \nabla \phi \in L^2(0,T; H^1(\Omega))$, 
and thereby
$$
\mu=
    - \diver \left( (\varepsilon+ \gamma |\nabla \phi|^2) \nabla \phi \right) 
    + \frac{1}{\varepsilon} \Psi'(\phi)
    -\frac{\alpha}{2}  |\di|^2 
    - \beta \diver \left( (\nabla \phi \cdot \di) \di \right), \quad \text{a.e. in } \Omega \times (0,T).
$$
To sum up, for any $\vv \in C([0,T], \vect{V}_m)$, there exists a quadruplet $(\phi, \di, \mu, \vect{h})$ on $[0,T]$ 
such that 
\begin{align}
    \label{CH-LC:reg1}
    & \phi \in L^\infty(0,T; W^{1,4}(\Omega))\cap L^2(0,T;H^2(\Omega)),\\
    \label{CH-LC:reg2}
    &\phi \in L^\infty(\Omega \times (0,T)) \, \text{with }
    |\phi(x,t)|<1 \, \text{a.e. in } \Omega \times (0,T),\\
    \label{CH-LC:reg3}
    &\partial_t \phi \in L^2(0,T; H^1(\Omega)'),
    \quad |\nabla \phi|^2 \nabla \phi \in L^2(0,T;H^1(\Omega)), \quad 
    F'(\phi)\in L^2(0,T;L^2(\Omega)),
    \\
    \label{CH-LC:reg4}
    & \di \in C([0,T];H^1(\Omega))\cap L^2([0,T];H^2(\Omega)) \cap
    H^1(0,T; L^2(\Omega)),
    \\
     \label{CH-LC:reg5}
    &  \di \in L^\infty(\Omega \times (0,T)) \, \text{with }
    |\di(x,t)|\leq \mathcal{D}_\infty \, \text{a.e. in } \Omega \times (0,T),
    \\
    \label{CH-LC:reg6}
    &\mu \in L^2(0,T;H^1(\Omega)),
    \quad 
    \vect{h} \in L^2(0,T; L^2(\Omega)),
    \end{align}
satisfying
\begin{subequations}
\begin{alignat}{2}
\label{CH-LC:wp1}
    & \langle \partial_t \phi, \xi \rangle_{H^1(\Omega)}
    - (\phi \vv, \nabla \xi)
    + (\nabla \mu, \nabla \xi)
    =0, \quad &&\forall \, \xi \in H^1(\Omega), \text{ a.e. in }(0,T),
    \\
    \label{CH-LC:wp2}
    &\mu=
    - \diver \left( (\varepsilon+ \gamma |\nabla \phi|^2) \nabla \phi \right) 
    + \frac{1}{\varepsilon} \Psi'(\phi)
    -\frac{\alpha}{2}  |\di|^2 
    - \beta \diver \left( (\nabla \phi \cdot \di) \di \right), \quad &&\text{a.e. in } \Omega \times (0,T),
    \\
    \label{CH-LC:wp3}
    &\partial_t \di + \prt{\vv \cdot \nabla}\di = -\vect{h}, \quad 
    &&\text{a.e. in } \Omega \times (0,T),\\
    \label{CH-LC:wp4}
    &\vect{h} = -\kappa \Delta \di +\alpha \abs{\di}^2 \di - \alpha(\phi -\phi_{\rm cr}) \di +\beta \prt{\di \cdot \nabla \phi}\nabla \phi, \quad &&
    \text{a.e. in } \Omega \times (0,T),
\end{alignat}
\end{subequations}
and  
\begin{equation}
\label{CH-LC:wp5}
\phi(\cdot,0)=\phi_0, \quad \di(\cdot, 0)=\di_0 \quad \text{in } \Omega.
\end{equation}
In addition, passing to the limit in \eqref{EEE1}, the free energy inequality
\begin{equation}
\label{CH-LC:EE}
\begin{split}
&E_{\rm free}^\gamma(\phi(t), \di(t)) + 
\int_0^t \norm{\nabla \mu(s)}_{L^2(\Omega)}^2 \, \d s+ 
\int_0^t \norm{\vect{h}(s)}_{L^2(\Omega)}^2 \, \d s \\
&\quad \leq 
E_{\rm free}^\gamma(\phi_0, \di_0)
+\int_0^t \int_\Omega \phi \vv \cdot \nabla \mu \, \d x \, \d s-
\int_0^t \int_\Omega \prt{\vv \cdot \nabla}\di \cdot \vect{h}\, \d x \, \d s
\end{split}
\end{equation}
holds for all $t\in [0,T]$. Moreover, we have from \eqref{EE:vv} that
\begin{equation}
\label{CH-LC:EE2}
\begin{split}
     \int_0^t \norm{\nabla \mu(s)}_{L^2(\Omega)}^2 \, \d s
       &+ \int_0^t \norm{\vect{h}(s)}_{L^2(\Omega)}^2 \, \d s \\
        &\leq 4 \left( E^\gamma_{\rm{free}}(\phi_0,\di_0) -(2\pi)^n E_0+ C \int_0^t \|\vv(s) \|_{L^2(\Omega)}^2 \, \d s \right)
       \mathrm{exp}\left(CT + C \| \di_0\|_{L^2(\Omega)}^2\right),
       \end{split}
\end{equation}
for all $t\in [0,T]$, where $C$ depends on $\alpha, \beta, \gamma, \varepsilon, \kappa, n, m, \mathcal{D}_\infty$. Also, there exist two positive constants $C_1^\star$, $C_2^\star$ such that
\begin{align}
\label{est:lim:0}
& \| \di\|_{L^\infty(0,T; L^2(\Omega))}
+ \| \di\|_{L^2(0,T; H^1(\Omega))}
\leq C_1^\star,
\\
\label{est:lim:1}
 &\| \phi\|_{L^\infty(0,T;W^{1,4}(\Omega))}
+\| \phi\|_{L^2(0,T;H^2(\Omega))}
+ \| |\nabla \phi|^2 \nabla \phi \|_{L^2(0,T;H^1(\Omega))}
+ \| \partial_t \phi\|_{L^2(0,T;H^1(\Omega)')} \leq C_2^\star,
\\
\label{est:lim:2}
& \| \di\|_{L^\infty(0,T; H^1(\Omega))}
+ \| \di\|_{L^2(0,T; H^2(\Omega))}
+ \| \partial_t \di\|_{L^2(0,T; L^2(\Omega))}\leq C_2^\star,
\\
\label{est:lim:3}
&\| \mu\|_{L^2(0,T;H^1(\Omega))}
+ \| \vect{h}\|_{L^2(0,T;L^2(\Omega))}
+ \|F'(\phi)\|_{L^2(0,T;L^2(\Omega))}
\leq C_2^\star,
\end{align}
where $C_1^\star$ depends on $\alpha$, $\beta$, $\gamma$, $\varepsilon$, $\kappa$, $\mathcal{D}_\infty$, $T$, $n$, whereas $C_2^\star$ depends on $\alpha$, $\beta$, $\gamma$, $\varepsilon$, $\kappa$, $m$, $\mathcal{D}_\infty$, $E^\gamma_{\rm free}(\phi_0,\di_0)$, $T$, $n$, and $\| \vv\|_{L^2(0,T; L^2(\Omega))}$. More precisely, the dependence on $T$ is only through its measure.
\medskip

\noindent
{\bf Step 3: Existence of global weak solutions for the full NS-CH-EL system.}
We introduce the semi Galerkin approximation of the momentum equations \eqref{eq:identity_RHS_NSE}. 
Let $\lbrace \zz_k \rbrace$ be the sequence of eigenfunctions of the Stokes operator. For any $N \in \N$, we define the finite dimensional space $\vect{V}_N= \text{Span} \lbrace \zz_1, \dots, \zz_N\rbrace$. We fix $\vv_N \in C([0,T]; \vect{V}_N)$ such that $\widehat{(\vv_N)}_0=\mathbf{0}$. Thanks to the previous Step 2, for any $N \in \N$, there exists a quadruplet $(\phi^N,\mu^N, \di^N, \vect{h}^N)$ satisfying \eqref{CH-LC:reg1}-\eqref{CH-LC:reg6}, the system \eqref{CH-LC:wp1}-\eqref{CH-LC:wp4} with $\vv$ replaced by $\vv_N$, and the initial conditions in \eqref{CH-LC:wp5}.
Furthermore, the energy inequality \eqref{CH-LC:EE}, the integral estimate \eqref{CH-LC:EE2} and the estimates \eqref{est:lim:0}-\eqref{est:lim:3} holds by replacing $(\phi, \di, \mu, \vect{h})$, $\vv$ and $m$ with $(\phi^N, \di^N, \mu^N, \vect{h}^N)$,  $\vv_N$ and $N$, respectively.

For any $N\in \N$, we look for a function
$\uu^N(t,x) = \sum_{k=1}^N a_k^N(t)\zz_k$, 
that solves
\begin{equation}
    \label{eq:NSE_G}
    \prt{\partial_t \uu^N, \zz_i} +\prt{(\vv_N \cdot \nabla) \uu^N, \zz_i} + \prt{\nu(\phi^N) D\uu^N, \nabla \zz_i} = \prt{\mu^N\nabla \phi^N+(\nabla \di^N)^T \vect{h}^N, \zz_i}\qquad \forall\, i = 1,\dots,n,
\end{equation}
together with the initial condition $\uu^N(0,\vect{x}) = \tens{P}_N\uu_0$, where $\tens{P}_N$ is the orthogonal projection from $L^2_\sigma(\Omega)$ onto $\vect{V}_N$.
The system \eqref{eq:NSE_G} is equivalent to a system of linear ODEs with non-constant coefficients, i.e.
\begin{equation}
\label{LODE}
\dt \texttt{A}^N(t) + \texttt{L}^N(t)\, \texttt{A}^N(t) = \texttt{G}^N(t)
\end{equation}
where $\texttt{A}^N(t) = \left[a_1^N(t),\dots,a^N_N(t)\right]^{T}$ and
$$
\begin{aligned}
&\prt{\texttt{L}^N(t)}_{i,j} = \int_\Omega \prt{\vv_N\cdot \nabla}\zz_j\cdot \zz_i + \nu(\phi^N)D\zz_j:\nabla \zz_i\, \d x,&&\prt{\texttt{G}^N(t)}_i = \int_\Omega \prt{\mu^N \nabla \phi^N +(\nabla \di^N)^T \vect{h}^N}\cdot \zz_i\, \d x.
\end{aligned}
$$
Since $\vv \in C([0,T];\vect{V}_N)$, $\phi \in C([0,T];L^2(\Omega))$ and $\di \in C([0,T];H^1(\Omega))$, we have that $\texttt{L}^N(t) \in C([0,T])$. Concerning $\texttt{G}^N(t)$, it follows from \eqref{CH-LC:reg1}, \eqref{CH-LC:reg4} and \eqref{CH-LC:reg5} that 
$\texttt{G}^N(t) \in L^2([0,T])$. 
Then, making use of the theory of systems of linear ODEs with summable coefficients (see, e.g., \cite[Chapter 3, Problem 1]{coddington1955theory}),
there exists a unique vector-valued function $\texttt{A}^N(\cdot) \in C([0,T])\cap W^{1,2}(0,T)$ fulfilling \eqref{LODE} almost everywhere in $(0,T)$ such that 
$\texttt{A}^N(0) = \left[(\uu_0,\zz_1),\dots,(\uu_0,\zz_N)\right]^{T}$. 

We multiply \eqref{eq:NSE_G} by $a^N_i(t)$ and sum over the index $i$. As a result, we find
\begin{equation}
\label{EE:NS}
\frac{1}{2}\dt \norm{\uu^N}_{L^2(\Omega)}^2 + \int_\Omega \nu(\phi^N)\abs{D\uu^N}^2\, \d x 
= -\int_\Omega \prt{\nabla \mu^N  \phi^N-\left(\nabla \di^N\right)^T\vect{h}^N} \cdot \uu^N\, \d x.
\end{equation}
By using \eqref{Rev-SI} and \eqref{CH-LC:reg2}, we arrive at
$$
\begin{aligned}
\left| \int_\Omega \prt{-\nabla \mu^N  \phi^N+(\nabla \di^N)^T\vect{h}^N} \cdot \uu^N\, \d x \right| 
&\leq \norm{\phi^N}_{L^\infty(\Omega)}
\norm{\nabla \mu^N}_{L^2(\Omega)}
\norm{\uu^N}_{L^2(\Omega)} 
+ \norm{\nabla \di^N}_{L^2(\Omega)} \norm{\vect{h}^N}_{L^2(\Omega)} \norm{\uu^N}_{L^\infty(\Omega)}\\
&\leq \norm{\nabla \mu^N}_{L^2(\Omega)} \norm{\uu^N}_{L^2(\Omega)} + C_N \norm{\nabla \di^N}_{L^2(\Omega)} 
\norm{\vect{h}^N}_{L^2(\Omega)} \norm{\uu^N}_{L^2(\Omega)}\\
&\leq C\prt{1 + \norm{\nabla \di^N}_{L^2(\Omega)}^2}\norm{\uu^N}_{L^2(\Omega)}^2 + C \norm{\nabla \mu^N}_{L^2(\Omega)}^2 + C_N\norm{\vect{h}^N}_{L^2(\Omega)}^2.
\end{aligned}
$$
Hence, we have
$$
\dt \norm{\uu^N}_{L^2(\Omega)}^2+ 
\nu_{*} \norm{D\uu^N}_{L^2(\Omega)}^2 
\leq C\prt{1 + \norm{\nabla \di^N}_{L^2(\Omega)}^2}\norm{\uu^N}_{L^2(\Omega)}^2 + C \norm{\nabla \mu^N}_{L^2(\Omega)}^2 
+ C_N\norm{\vect{h}^N}_{L^2(\Omega)}^2.
$$
By applying the Gronwall lemma, the latter entails
\begin{equation}
    \label{est:1N}
    \begin{split}
    &\norm{\uu^N(t)}_{L^2(\Omega)}^2 + \nu_{*}\int_0^t \norm{D\uu^N(s)}_{L^2(\Omega)}^2\, \d s
    \\
    &\quad \leq  
    \left( \norm{\uu^N(0)}_{L^2(\Omega)}^2
    + C_N \int_0^t \prt{\norm{\nabla \mu^N(s)}_{L^2(\Omega)}^2 + \norm{\vect{h}^N(s)}_{L^2(\Omega)}^2}\, \d s\right)\mathrm{exp}\left( C \int_0^t 1 + \norm{\nabla \di^N(s)}_{L^2(\Omega)}^2 \, \d s\right)
    \\
     &\quad \leq  
    \left( \norm{\uu_0}_{L^2(\Omega)}^2
    + C_N \int_0^t \prt{\norm{\nabla \mu^N(s)}_{L^2(\Omega)}^2 + \norm{\vect{h}^N(s)}_{L^2(\Omega)}^2}\, \d s\right)\mathrm{exp} \left( C \int_0^t 1 + \norm{\nabla \di^N(s)}_{L^2(\Omega)}^2 \, \d s \right).
    \end{split}
\end{equation}
In light of \eqref{CH-LC:EE2} and \eqref{est:lim:0}, we observe that 
\begin{equation}
    \label{eq:nablaN}
    \int_0^T \prt{1 + \norm{\nabla \di^N(s)}_{L^2(\Omega)}^2}\, \d s
    \leq C(T),
\end{equation}
and for any $t \in[0,T]$
\begin{equation}
\label{eq:muhN}
\int_0^t \norm{\nabla \mu^N(s)}_{L^2(\Omega)}^2 + \norm{\vect{h}^N(s)}_{L^2(\Omega)}^2\, \d s \leq C_N(T,E^\gamma_{\rm free}(\phi_0,\di_0)) 
\int_0^t \norm{\vv_N(s)}_{L^2(\Omega)}^2 \, \d s + C \prt{T, E^\gamma_{\rm free}(\phi_0,\di_0)}.
\end{equation}
Here, the constants $C_N$ and $C$ also depends on the parameters of the system. Substituting the above estimates into \eqref{est:1N}, we deduce that
\begin{equation}
\label{est:2N}
\norm{\uu^N(t)}_{L^2(\Omega)}^2 + \nu_{*}\int_0^t \norm{D\uu^N(s)}_{L^2(\Omega)}^2 \, \d s \leq \widetilde{K}_1 + \widetilde{K}_2 \int_0^t \norm{\vv_N(s)}_{L^2(\Omega)}^2 \, \d s, 
\quad \forall \, t \in [0,T],
\end{equation}
where $\widetilde{K}_1, \widetilde{K}_2 := C(\|\uu_0\|_{L^2(\Omega)},E^\gamma_{\rm free}(\phi_0,\di_0),T,N)$. 
Assuming that
\begin{equation}
\label{ASS-vvN}
\int_0^t \norm{\vv_N (s)}_{L^2(\Omega)}^2 \, \d s\leq \widetilde{K}_1 T\, \mathrm{e}^{\widetilde{K}_2 t}  \quad \hbox{ for } t \in [0,T], 
\end{equation}
we obtain from \eqref{est:2N} that
\begin{equation}
\label{est:3N}
\int_0^t \norm{\uu^N(s)}_{L^2(\Omega)}^2 \, \d s 
\leq \widetilde{K}_1 T + \widetilde{K}_2 \left[\int_0^t \int_0^s \norm{\vv_N(\tau)}_{L^2(\Omega)}^2\, \d \tau\, \d s\right]
= \widetilde{K}_1 T \mathrm{e}^{\widetilde{K}_2 T}, \quad \forall \, t \in [0,T].
\end{equation}
As a direct consequence of \eqref{est:2N} and \eqref{ASS-vvN}, we also infer
\begin{equation}
    \label{est:4N}
    \| \uu^N\|_{C([0,T]; L^2(\Omega))}
    \leq \left( \widetilde{K}_1+ \widetilde{K}_1 \widetilde{K}_2 T \mathrm{e}^{\widetilde{K}_2 T} \right)^\frac12=: \widetilde{K}_3.
\end{equation}
Furthermore, multiplying \eqref{eq:NSE_G} by $\partial_t a_i^N$, summing over $i$,  integrating on $(0,T)$ and exploiting \eqref{Rev-SI}, we obtain
\begin{equation}
    \begin{split}
   \int_0^T \| \partial_t \uu^N(s)\|_{L^2(\Omega)}^2 \, \d s
    &\leq C_N \int_0^T 
    \| \vv_N(s)\|_{L^2(\Omega)}^2 
    \| \uu^N(s)\|_{L^2(\Omega)}^2  
    \, \d s
    + C_N \nu^\ast \int_0^T 
    \| \uu^N(s)\|_{L^2(\Omega)}^2  
    \, \d s
    \\
    &\quad + C_N \int_0^T \| \phi^N (s)\|_{L^2(\Omega)}^2 \| \nabla \mu^N(s)\|_{L^2(\Omega)}^2
   + \| \nabla \di^N(s)\|_{L^2(\Omega)}^2
    \| \vect{h}^N(s)\|_{L^2(\Omega)}^2 \, \d s.
\end{split}
\end{equation}
Then, observing from \eqref{ASS-vvN} and \eqref{est:lim:2} that $\| \nabla \di^N\|_{L^\infty(0,T; L^2(\Omega))}
\leq C(\widetilde{K}_1,\widetilde{K}_2)$,  and using \eqref{eq:muhN}, and \eqref{est:4N}, it follows that
\begin{equation}
\label{est:5N}
    \int_0^T \| \partial_t \uu^N(s)\|_{L^2(\Omega)}^2 \, \d s
    \leq C(\widetilde{K}_1, \widetilde{K}_2, \widetilde{K}_3)=:\widetilde{K}_4^2.
\end{equation}
Hence, for any $N \in \N$, we can define the map
\begin{equation}
    \label{eq:def_Z}
    Z_N: \, 
    \mathcal{X}_T \to \mathcal{Y}_T \subset \mathcal{X}_T, \quad 
    \vv_N \mapsto Z_N(\vv_N):= \uu^N,
\end{equation}
where
\begin{equation}
    \label{eq:def_XT}
    \begin{aligned}
    \mathcal{X}_T = \Big\{ \vv \in C([0,T];\vect{V}_N): \int_0^t \norm{\vv(s)}_{L^2(\Omega)}^2 \, \d s 
    \leq \widetilde{K}_1 T\, \mathrm{e}^{\widetilde{K}_2 t} \, \text{  for } t\in [0,T], \text{ and } \| \vv \|_{C([0,T]; L^2(\Omega))}\leq \widetilde{K}_3 \Big\},
    \end{aligned}
\end{equation}
and 
\begin{equation}
    \label{eq:def_YT}
    \begin{aligned}
    \mathcal{Y}_T = \Big\{ \vv \in \mathcal{X}_T\cap W^{1,2}(0,T;\vect{V}_N) : \| \partial_t \vv\|_{L^2(0,T;L^2(\Omega))}\leq \widetilde{K}_4 \Big\}.
    \end{aligned}
\end{equation}
We immediately notice that $\mathcal{X}_T$ is convex and closed in $L^2(0,T; \vect{V}_N)$. In addition, $\mathcal{Y}_T$ is compact in $L^2(0,T; \vect{V}_N)$. We claim that $Z_N$ is continuous in $L^2(0,T; \vect{V}_N)$. For the sake of brevity, we leave the proof to the interested reader. Nonetheless, it can be easily established by using the properties of $\mathcal{X}_T$ and performing a similar argument to the one in Section \ref{sec:uniqueness_2D} (cf. also \cite{giorgini2022existence, giorgini2021well}). 
Therefore, the Schauder fixed point theorem yields that, for any $N \in \N$, there exists $(\uu^N,\phi^N,\di^N, \mu^N,\vect{h}^N)$ on $[0,T]$ such that 
\begin{align}
    \label{NS-CH-LC:reg0}
    &\uu^N \in W^{1,2}(0,T;\vect{V}_N),\\
     \label{NS-CH-LC:reg1}
    & \phi^N \in L^\infty(0,T; W^{1,4}(\Omega))\cap L^2(0,T;H^2(\Omega)),\\
    \label{NS-CH-LC:reg2}
    &\phi^N \in L^\infty(\Omega \times (0,T)) \, \text{with }
    |\phi^N(x,t)|<1 \, \text{a.e. in } \Omega \times (0,T),\\
    \label{NS-CH-LC:reg3}
    &\partial_t \phi^N \in L^2(0,T; H^1(\Omega)'),
    \quad |\nabla \phi^N|^2 \nabla \phi^N \in L^2(0,T;H^1(\Omega)), \quad 
    F'(\phi^N)\in L^2(0,T;L^2(\Omega)),
    \\
    \label{NS-CH-LC:reg4}
    & \di^N \in C([0,T];H^1(\Omega))\cap L^2(0,T;H^2(\Omega)) \cap
    H^1(0,T; L^2(\Omega)),
    \\
     \label{NS-CH-LC:reg5}
    &  \di^N \in L^\infty(\Omega \times (0,T)) \, \text{with }
    |\di^N(x,t)|\leq \mathcal{D}_\infty \, \text{a.e. in } \Omega \times (0,T),
    \\
    \label{NS-CH-LC:reg6}
    &\mu^N \in L^2(0,T;H^1(\Omega)),
    \quad 
    \vect{h}^N \in L^2(0,T; L^2(\Omega)),
    \end{align}
satisfying
\begin{subequations}
\begin{alignat}{2}
\label{NS-CH-LC:wp0}
&\prt{\partial_t \uu^N, \zz} +\prt{(\uu^N \cdot \nabla) \uu^N, \zz} + \prt{\nu(\phi^N) D\uu^N, \nabla \zz} = (\mu^N\nabla \phi^N+(\nabla \di^N)^T \vect{h}^N, \zz),\ &&\forall \, \zz \in \vect{V}_N, \text{ a.e. in }(0,T),\\
\label{NS-CH-LC:wp1}
    & \langle \partial_t \phi^N, \xi \rangle_{H^1(\Omega)}
    - (\phi^N \uu^N, \nabla \xi)
    + (\nabla \mu^N, \nabla \xi)
    =0, \ &&\forall \, \xi \in H^1(\Omega), \text{ a.e. in }(0,T),
    \\
    \label{NS-CH-LC:wp2}
    &\mu^N=
    - \diver \left( (\varepsilon+ \gamma |\nabla \phi^N|^2) \nabla \phi^N \right) 
    + \frac{1}{\varepsilon} \Psi'(\phi^N)
    -\frac{\alpha}{2}  |\di^N|^2 
    - \beta \diver ( (\nabla \phi^N \cdot \di^N) \di^N), \ &&\text{a.e. in } \Omega \times (0,T),
    \\
    \label{NS-CH-LC:wp3}
    &\partial_t \di^N + \prt{\uu^N \cdot \nabla}\di^N = \vect{h}^N, \
    &&\text{a.e. in } \Omega \times (0,T),\\
    \label{NS-CH-LC:wp4}
    &\vect{h}^N = \kappa \Delta \di^N -\alpha \abs{\di^N}^2 \di^N + \alpha(\phi^N -\phi_{\rm cr}) \di^N -\beta \prt{\di^N \cdot \nabla \phi^N}\nabla \phi^N, \ &&
    \text{a.e. in } \Omega \times (0,T),
\end{alignat}
\end{subequations}
and  
\begin{equation}
\label{NS-CH-LC:wp5}
\uu^N(\cdot, 0)=\tens{P}_N\uu_0, \quad
\phi^N(\cdot,0)=\phi_0, \quad \di^N(\cdot, 0)=\di_0 \quad \text{in } \Omega.
\end{equation}
Furthermore, the free energy inequality \eqref{CH-LC:EE} and \eqref{est:lim:0}-\eqref{est:lim:3} holds replacing $\vv_N$ with $\uu^N$. Then, integrating \eqref{EE:NS} in time on $(0,t)$ with $t\leq T$, adding the resulting equation to the free energy inequality \eqref{CH-LC:EE} and recalling that $E_{\rm kin}(\tens{P}_N\uu_0)\leq E_{\rm kin}(\uu_0)$, we obtain
\begin{equation}
\label{NS-CH-LC:EE}
\begin{split}
&E_{\rm kin}(\uu^N(t))+ E^\gamma_{\rm free}(\phi^N(t), \di^N(t)) + 
\int_0^t \norm{ \sqrt{\nu(\phi^N(s))} D \uu^N(s)}_{L^2(\Omega)}^2 \, \d s
\\
&\quad +\int_0^t \norm{\nabla \mu^N(s)}_{L^2(\Omega)}^2 \, \d s+ 
\int_0^t \norm{\vect{h}^N(s)}_{L^2(\Omega)}^2 \, \d s  \leq 
E_{\rm kin}(\uu_0)+
E^\gamma_{\rm free}(\phi_0, \di_0), \quad \forall \, t \in [0,T].
\end{split}
\end{equation}
As  direct consequence, we infer that
\begin{equation}
    \label{NSCHLC:est1}
    \begin{split}
    \norm{\uu^N}_{L^\infty(0,T;\dot{L}_\sigma^2(\Omega))}+
    \norm{\phi^N}_{L^\infty(0,T;W^{1,4}(\Omega))}+
    \norm{\di^N}_{L^\infty(0,T;H^1(\Omega))}
    \leq C_E,\\
    \norm{\uu^N}_{L^2(0,T;\dot{H}_\sigma^1(\Omega))}+ 
    \norm{\nabla \mu^N}_{L^2(0,T;L^2(\Omega))}
    +\norm{ \vect{h}^N}_{L^2(0,T;L^2(\Omega))}
    \leq C_E,
    \end{split}
\end{equation}
where $C_E$ depends on the parameters of the system and $E^\gamma_{\rm tot}(\uu_0,\phi_0,\di_0)$, but it is independent of $N$. Besides, in light of \eqref{NSCHLC:est1}, repeating line by line the computations to get \eqref{est:vv:3}-\eqref{est:vv:13}, we directly deduce that 
\begin{align}
\label{NSCHLC:est2}
&\| \phi^N\|_{L^2(0,T;H^2(\Omega))}
+ \| |\nabla \phi^N|^2 \nabla \phi^N\|_{L^2(0,T;H^1(\Omega))}
+\| \partial_t \phi^N\|_{L^2(0,T;H^1(\Omega)')} \leq C_E^\star,
\\
\label{NSCHLC:est3}
&\| \di^N\|_{L^2(0,T; H^2(\Omega))}
+ \| \partial_t \di^N\|_{L^2(0,T; L^2(\Omega))}\leq C_E^\star,
\\
\label{NSCHLC:est4}
&\| \mu^N\|_{L^2(0,T;H^1(\Omega))}
+\|F'(\phi^N)\|_{L^2(0,T;L^2(\Omega))}
\leq C_E^\star,
\end{align}
where $C_E^\star$ depends on the parameters of the system, the measure of $T$ and $E^\gamma_{\rm tot}(\uu_0,\phi_0,\di_0)$, but is independent of $N$\footnote{In fact, we highlight that the constant $C_2^\star$ in \eqref{est:lim:1}-\eqref{est:lim:3} depends on $m$ (and so $N$ in this step) only through \eqref{EE:vv}. But, the new estimate \eqref{NSCHLC:est1}, which is the analogue of \eqref{EE:vv} in the last step, is now independent of $N$.}.

We now proceed to the estimate of the partial derivative of $\uu^N$. We consider the {\color{black} case $n = 3$}. Let us fix $\ww \in \dot{H}^1_\sigma(\Omega)$ with $\norm{\ww}_{\dot{H}_\sigma^1(\Omega)}\leq 1$. It follows from \eqref{NS-CH-LC:wp0} that 
\begin{equation}
    \begin{split}
    |\langle \partial_t \uu^N, \ww \rangle_{\dot{H}^1_\sigma(\Omega)}|
    &=  \left| \prt{(\uu^N \cdot \nabla) \uu^N, \tens{P}_N \ww}\right| 
    + \left| \prt{\nu(\phi^N) D\uu^N, \nabla \tens{P}_N \ww}\right| + 
    \left| (\mu^N\nabla \phi^N-\nabla \di^N\, \vect{h}^N, \tens{P}_N \ww)\right|
    \\
    &\leq \norm{\uu^N}_{L^3(\Omega)} 
    \norm{\nabla \uu^N}_{L^2(\Omega)}
    \norm{P_N \ww}_{L^6(\Omega)}
    + \nu^\ast  \norm{\nabla \uu^N}_{L^2(\Omega)}
    \norm{ \nabla P_N \ww}_{L^2(\Omega)}
    \\
    & \quad + \norm{\phi^N}_{L^\infty(\Omega)}
    \norm{\nabla \mu^N}_{L^2(\Omega)}
    \norm{\tens{P}_N \ww}_{L^2(\Omega)}
    + \norm{\nabla \di^N}_{L^3(\Omega)}
    \norm{ \hh^N}_{L^2(\Omega)}
    \norm{\tens{P}_N \ww}_{L^6(\Omega)}
    \\
    &\leq C_E \| \nabla \uu^N\|_{L^2(\Omega)}^\frac32 + C \| \nabla \uu^N\|_{L^2(\Omega)} +
    C\| \nabla \mu^N\|_{L^2(\Omega)} 
    + C_E \| \di^N\|_{H^2(\Omega)}^\frac12 
    \| \hh^N\|_{L^2(\Omega)},
    \end{split}
\end{equation}
where the constants are independent of $N$.
Thus, \eqref{NSCHLC:est1}, \eqref{NSCHLC:est3} and \eqref{NSCHLC:est4} entail that 
\begin{equation}
    \label{NSCHLC:est5}
    \norm{\partial_t \uu^N}_{L^\frac43(0,T; \dot{H}^1_\sigma(\Omega)')}\leq C_E^\star.
\end{equation}
The estimate of $\partial_t \uu^N$ is slightly different if $n=2$. In fact, 
the right-hand side of \eqref{NS-CH-LC:wp0} can be rewritten by exploiting \eqref{RHS-NS}. Then, integrating by parts in the terms on the right-hand side, and using \eqref{LADY} and \eqref{NSCHLC:est1}-\eqref{NSCHLC:est4}, it is possible to show that $\norm{\partial_t \uu^N}_{L^2(0,T; \dot{H}^1_\sigma(\Omega)')}\leq C_E^\star$.

Thanks to the above uniform estimates with respect to the parameter $N$, we are in the position to apply compactness arguments to deduce the existence of a limit quintet $(\uu,\phi,\di,\mu, \hh)$ which satisfies \eqref{W:reg}, \eqref{W:weak:u}-\eqref{W:weak:d} and \eqref{MR:en:in} on $[0,T]$ (see e.g. \cite{giorgini2021well, giorgini2022existence}). Finally, by the arbitrariness of $T$, a continuation argument (see, e.g., \cite[Chapter V, Sec. 1.3.6]{boyer2012mathematical}) allows to extend the solution on $[0,\infty)$ fulfilling \eqref{W:weak:u}-\eqref{W:weak:d} and \eqref{MR:en:in}. The regularity in \eqref{W:reg}, except for the properties $\phi \in L^4_{\rm uloc}([0,\infty); H^2(\Omega))$ and $\Psi'(\phi)\in L_{\rm uloc}^2([0,\infty); L^p(\Omega))$, follows from the combination of \eqref{MR:en:in} and the estimates 
\eqref{est:vv:3}-\eqref{est:vv:13} replacing $(\vv,\phi^\delta,\di^\delta,\mu^\delta,\hh^\delta)$ with $(\uu,\phi,\di,\mu,\hh)$.

\medskip 

\noindent
\textbf{Step 4. Further regularity properties of the concentration.}
We improve the regularity of $\phi$ in the spirit of \cite{giorgini2018cahn}. Let us multiply \eqref{eq:def_mu} by $-\Delta \phi$ and integrate over $\Omega$. By using \eqref{gamma:term}, \eqref{beta:term}, \eqref{est:vv:8}, and \eqref{est:vv:10},
we obtain
\begin{equation}
\label{FR:1}
\begin{split}
& 
\frac{\varepsilon}{2} \norm{\Delta \phi}_{L^2(\Omega)}^2 + \gamma
\int_{\Omega} \abs{D^2 \phi}^2 \abs{\nabla \phi}^2 \, \d x
+\frac{\gamma}{4} \int_\Omega \left| \nabla\left(\abs{\nabla \phi}^2\right) \right|^2 \, \d x
+ \beta \int_\Omega |\nabla(\nabla \phi \cdot \di)|^2 \, \d x
\\
&\leq 
\norm{\nabla \mu}_{L^2(\Omega)}
\norm{\nabla \phi}_{L^2(\Omega)}
+ \frac{\alpha^2 C(n)}{4\varepsilon}\mathcal{D}_\infty^4+\frac{1}{\varepsilon} \left( \frac{\Theta_0}{\varepsilon}+\frac12\right)^2 C(n)\\
&\quad +
\beta \left| \int_\Omega \nabla(\nabla \phi \cdot \di) \cdot (\nabla (\di)^T \nabla \phi) \, \d x    +  \int_\Omega (\nabla \phi \cdot \di) \nabla \di : D^2 \phi \, \d x\right|.
\end{split}
\end{equation}
Now, by \eqref{GN} and \eqref{MR:en:in}, we have
\begin{equation}
\label{FR:2}
\begin{split}
&\beta \left| \int_\Omega \nabla(\nabla \phi \cdot \di) \cdot (\nabla (\di)^T \nabla \phi) \, \d x +  \int_\Omega (\nabla \phi \cdot \di) \nabla \di : D^2 \phi \, \d x\right|
 \\   
&= \beta \left| \sum_{i,j=1}^n \int_\Omega 
-\partial_l \partial_i \phi \di_i \partial_l \di_j \partial_j \phi
-\partial_i \phi \partial_l \di_i \partial_l \di_j \partial_j \phi 
+ \partial_i \phi \di_i \partial_l \di_j \partial_l \partial_j \phi 
\, \d x 
\right|
\\
&\leq 
\frac{\gamma}{2} \int_\Omega |D^2\phi|^2 |\nabla \phi|^2 \, \d x
+ \frac{2\beta^2 C(n)}{\gamma}\mathcal{D}_\infty^2 \|\nabla \di \|_{L^2(\Omega)}^2
+\beta
\| \nabla \di\|_{L^4(\Omega)}^2 \| \nabla \phi\|_{L^4(\Omega)}^2
\\
&\leq 
\frac{\gamma}{2} \int_\Omega |D^2\phi|^2 |\nabla \phi|^2 \, \d x
+ C \| \di \|_{H^2(\Omega)},
\end{split}
\end{equation}
where $C$ is a positive constant depending on $\beta, \gamma, n, \mathcal{D}_\infty$ and $E^\gamma_{\rm tot}(\uu_0,\phi_0,\di_0)$.
Combining \eqref{FR:1} and \eqref{FR:2}, we arrive at
\begin{equation}
\label{FR:3}
\begin{split}
& 
\frac{\varepsilon}{2} \norm{\Delta \phi}_{L^2(\Omega)}^2 + \frac{\gamma}{2} 
\int_{\Omega} \abs{D^2 \phi}^2 \abs{\nabla \phi}^2 \, \d x
+\frac{\gamma}{4} \int_\Omega \left| \nabla\left(\abs{\nabla \phi}^2\right) \right|^2 \, \d x
\leq 
C\norm{\nabla \mu}_{L^2(\Omega)}
+ C \| \di\|_{H^2(\Omega)}+C,
\end{split}
\end{equation}
for some positive constant $C$ depending on $\alpha, \beta, \gamma, \varepsilon, n, \mathcal{D}_\infty$ and $E^\gamma_{\rm tot}(\uu_0,\phi_0,\di_0)$.
Thanks to $\mu \in L^2_{\rm uloc}([0,\infty);H^1(\Omega))$ and $\di \in L^2_{\rm uloc}([0,\infty);H^2(\Omega))$, we deduce that
$$
\phi \in L_{\rm uloc}^4([0,\infty);H^2(\Omega)),
\quad
|\nabla \phi|^2 \in L_{\rm uloc}^4([0,\infty);H^1(\Omega)).
$$

Next, testing \eqref{eq:def_mu} by $|F_k'(\phi)|^{p-2}F_k'(\phi)$ (cf. \eqref{trunc_F}) and integrate over $\Omega$, we find
\begin{equation}
\label{F_k}
\begin{split}
\int_{\Omega} (p-1) &|F_k'(\phi)|^{p-2} F_k''(\phi) \left(\varepsilon |\nabla\phi|^2 + \gamma |\nabla \phi|^4) + \beta |\nabla \phi \cdot \di|^2 \right) \, \d x+
\int_{\Omega} F'(\phi)|F_k'(\phi)|^{p-2}F_k'(\phi) \, \d x\\
&=
\int_{\Omega} \mu \, |F_k'(\phi)|^{p-2}F_k'(\phi)\, \d x
+ \Theta_0 \int_{\Omega} \phi \, |F_k'(\phi)|^{p-2}F_k'(\phi) \, \d x + \frac{\alpha}{2} \int_{\Omega} |\di|^2 \, |F_k'(\phi)|^{p-2}F_k'(\phi)\, \d x.
\end{split}
\end{equation}
Then, exploiting the sign of $F'$ and the global boundedness in $L^\infty$ of $\phi$ and $\di$, we are led to
$$
\sup_{t\geq 0}\int_t^{t+1} \| F_k'(\phi(s))\|_{L^p(\Omega)}^2 \, \d s
\leq C \sup_{t\geq 0} \int_t^{t+1}\| \mu(s)\|_{L^p(\Omega)}^2 \, \d s+ C
\leq C \sup_{t\geq 0} \int_t^{t+1}\| \mu(s)\|_{H^1(\Omega)}^2 \, \d s+ C,
$$
for $p=6$ if $n=3$ and any $p\in [2,\infty)$ if $n=2$. Since $\mu \in L^2_{\rm uloc}([0,\infty);H^1(\Omega))$, by passing to the limit as $k\to \infty$, we conclude that $F'(\phi)\in L^2_{\rm uloc}([0,\infty); L^p(\Omega))$ for $p$ as above.
\end{proof}

\subsection{Global existence of weak solutions: the limit \texorpdfstring{$\gamma \to 0$}{γ->0}}
\label{sec:gamma0}

We now conclude the proof of the first part of Theorem \ref{MR} by studying the limit as $\gamma \to 0$ for weak solutions to \eqref{eq:NSE}-\eqref{eq:def_h}. 

\begin{proof}[Proof of Theorem \ref{MR} - Existence of weak solutions.]
Let us consider $\phi_0 \in H^1(\Omega)$ such that $\| \phi_0\|_{L^\infty(\Omega)}\leq 1$ and $|\overline{\phi_0}|<1$. 
For any $\gamma \in (0,1]$, let us define $\phi_0^\gamma$ as the solution to the elliptic problem $- \sqrt{\gamma} \Delta \phi_0^\gamma+
\phi_0^\gamma=\phi_0$ with periodic boundary conditions. This means that the Fourier coefficients are 
$(\widehat{\phi_0^\gamma})_k= (1+\sqrt{\gamma}|k|^2)^{-1}(\widehat{\phi_0})_k$ for any $k \in \mathbb{Z}^n$. We end up with $\phi_0^\gamma \in H^2(\Omega)$ such that
\begin{equation}
\label{phi0gamma}
\overline{\phi_0^\gamma}=\overline{\phi_0}, \quad
\| \phi_0^\gamma\|_{L^p(\Omega)}\leq 
\| \phi_0\|_{L^p(\Omega)} \ \forall \, p \in [2,\infty], \quad
\| \nabla \phi_0^\gamma\|_{L^2(\Omega)}
\leq \| \nabla \phi_0\|_{L^2(\Omega)}, \quad
\gamma^\frac14 \| \Delta \phi_0^\gamma\|_{L^2(\Omega)}
\leq \| \nabla \phi_0\|_{L^2(\Omega)}
\end{equation}
In addition, $\phi_0^\gamma \rightarrow \phi_0$ in $H^1(\Omega)$. By \eqref{GN} and \eqref{phi0gamma}, we have
\begin{equation}
\label{phi0gamma-2} 
\frac{\gamma}{4}\| \nabla \phi_0^\gamma\|_{L^4(\Omega)}^4
=
\frac{\gamma}{4} \left\| \nabla \left(\phi_0^\gamma-\overline{\phi_0^\gamma}\right) \right\|_{L^4(\Omega)}^4
\leq \frac{\gamma}{4} C_\Omega \|\phi_0^\gamma-\overline{\phi_0^\gamma} \|_{L^\infty(\Omega)}^2
\|\Delta\phi_0^\gamma\|_{L^2(\Omega)}^2
\leq \sqrt{\gamma} C(\| \phi_0\|_{H^1(\Omega)}, \Omega).
\end{equation}
Now, for any $\gamma \in (0,1]$, Theorem \ref{MR:gamma-p} guarantees the existence of a sequence $(\uu^\gamma,\phi^\gamma, \mu^\gamma, \di^\gamma, \hh^\gamma)$ of global weak solutions to \eqref{eq:NSE}-\eqref{eq:def_h} 
satisfying the regularity properties in \eqref{W:reg}, the weak formulation \eqref{W:weak:u}-\eqref{W:weak:d} and the initial conditions 
$\uu(\cdot, 0)=\uu_0(\cdot)$, $\phi(\cdot, 0)=\phi_0^\gamma(\cdot)$, $\di(\cdot,0)=\di_0(\cdot)$ in $\Omega$, as well as the the energy inequality \eqref{MR:en:in}. 
We now carry out uniform estimates with respect to the parameter $\gamma$.
First, we have
\begin{equation*}    
    \begin{split}
    &\sup_{t\geq 0} \int_\Omega \left(
    \frac{1}{2} \abs{\vect{u}^\gamma(t)}^2 +
    \frac{\gamma}{4}\abs{\nabla \phi^\gamma(t)}^4 +\frac{\varepsilon}{2} \abs{\nabla \phi^\gamma(t)}^2 + \frac{1}{\varepsilon}\Psi(\phi^\gamma(t)) -\frac{\alpha}{2}(\phi^\gamma(t)-\phi_{\rm cr}) \abs{\vect{d}^\gamma(t)}^2 
    + \frac{\alpha}{4} \abs{\vect{d}^\gamma (t)}^4 + \frac{\kappa}{2}\abs{\nabla \vect{d}^\gamma(t)}^2 \right. 
    \\
    &\qquad \quad \left.   + \frac{\beta}{2} \abs{\nabla \phi^\gamma(t) \cdot \vect{d}^\gamma(t)}^2 -E_0 \right)
    \, \d x
    +
    \int_0^\infty \left\| \sqrt{\nu(\phi^\gamma(s))} D\uu^\gamma(s)\right\|_{L^2(\Omega)}^2 
        +\| \nabla \mu^\gamma(s)\|_{L^2(\Omega)}^2
        +\| \vect{h}^\gamma (s) \|_{L^2(\Omega)}^2 
        \, \d s
    \\
    & \leq 
    \int_\Omega \left( \frac{1}{2}  \abs{\vect{u}_0}^2 
    +\frac{\gamma}{4}\abs{\nabla \phi_0^\gamma}^4 +\frac{\varepsilon}{2} \abs{\nabla \phi_0^\gamma}^2 + \frac{1}{\varepsilon}\Psi(\phi_0^\gamma) -\frac{\alpha}{2}(\phi_0^\gamma-\phi_{\rm cr}) \abs{\vect{d}_0}^2 + \frac{\alpha}{4} \abs{\vect{d}_0}^4 + \frac{\kappa}{2}\abs{\nabla \vect{d}_0}^2 + \frac{\beta}{2} \abs{\nabla \phi_0^\gamma \cdot \vect{d}_0}^2 -E_0 \right)  \, \d x.
    \end{split}
\end{equation*}
Since $E_{\rm tot}^\gamma(\uu_0,\phi_0,\di_0)=
E_{\rm tot}(\uu_0,\phi_0,\di_0) + \frac{\gamma}{4}\| \nabla \phi_0\|_{L^4(\Omega)}^4$, we infer from the construction of $\phi_0^\gamma$ and \eqref{phi0gamma-2} that, for any $\varrho>0$, there exists $\gamma_0>0$ (possibly small) such that (up to a subsequence in $\phi_0^\gamma$)
\begin{equation}
\label{diff-ener}
\left| E_{\rm tot}^\gamma(\uu_0,\phi_0^\gamma,\di_0) 
-E_{\rm tot}(\uu_0,\phi_0,\di_0) \right| \leq \varrho, \quad \forall \, \gamma \in (0,\gamma_0).
\end{equation}
Let $\varrho=1$. Recalling that 
\begin{equation}
    \label{E:g_z:1}
    \| \phi^\gamma\|_{L^\infty(\Omega \times (0,\infty))}\leq 1 \qquad \text{and} \qquad
    \| \di^\gamma\|_{L^\infty(\Omega \times (0,\infty))}\leq \mathcal{D}_\infty,
\end{equation}
we deduce that
\begin{equation}
\label{E:g_z:2}
\begin{aligned}
&
\| \uu^\gamma\|_{L^\infty(0,\infty; \dot{L}^2(\Omega))}
+
\|\phi^\gamma\|_{L^\infty(0,\infty;H^1(\Omega))}
+ \gamma^\frac14 \|\phi^\gamma\|_{L^\infty(0,\infty;W^{1,4}(\Omega))}
+ \| \di^\gamma\|_{L^\infty(0,\infty; H^1(\Omega))}\leq  K_0,
\\
& 
\| \uu^\gamma\|_{L^2(0,\infty; \dot{H}^1(\Omega))}
+ \| \nabla\mu^\gamma\|_{L^2(0,\infty;L^2(\Omega))}
+ \|\vect{h}^\gamma\|_{L^2(0,\infty;L^2(\Omega))}
\leq K_0,
\end{aligned}
\end{equation}
where $K_0$ is independent of $\gamma$.
Next, we perform the key elliptic type estimates for the CH-EL subsystem, which will provide the necessary compactness to pass to the limit. Multiplying 
\eqref{W:weak:mu} and \eqref{W:weak:d} by $-\Delta \phi^\gamma$ and $\Delta \di^\gamma$, respectively, integrating over $\Omega$ and summing the resulting equations, we find
\begin{equation}
\label{E:g_z:3}
\begin{split}
&
\varepsilon\norm{\Delta \phi^\gamma}_{L^2(\Omega)}^2 
+ \kappa \norm{\Delta \di^\gamma}_{L^2(\Omega)}^2
+ \gamma \int_\Omega \diver \left( |\nabla \phi^\gamma |^2 \nabla \phi^\gamma \right) \Delta \phi^\gamma \, \d x - \frac{1}{\varepsilon} \int_\Omega F'(\phi^\gamma) \Delta \phi^\gamma \, \d x
- \alpha \int_\Omega |\di^\gamma|^2 \di^\gamma \cdot \Delta \di^\gamma \, \d x
\\
&
+
\beta \int_\Omega \diver\prt{(\nabla \phi^\gamma \cdot \di^\gamma)\di^\gamma} \Delta \phi^\gamma \, \d x
- \beta \int_\Omega (\di^\gamma \cdot \nabla \phi^\gamma) \nabla \phi^\gamma \cdot \Delta \di^\gamma \, \d x= \int_\Omega \nabla \mu^\gamma \cdot \nabla \phi^\gamma \, \d x 
+ \frac{\Theta_0}{\varepsilon} 
\int_{\Omega} |\nabla \phi^\gamma|^2 \, \d x\\
&\quad+\frac{\alpha}{2}\int_\Omega \Delta \phi^\gamma \abs{\di^\gamma}^2\, \d x 
+\int_\Omega \hh^\gamma \cdot \Delta \di^\gamma \, \d x
-\alpha \int_\Omega (\phi^\gamma-\phi_{\rm cr}) \di^\gamma \cdot \Delta \di^\gamma \, \d x
.
\end{split}
\end{equation}
It is clear that the third to the fifth terms on the left-hand side in \eqref{E:g_z:3} are non-negative (cf. \eqref{F2:pos} and \eqref{gamma:term}).
Concerning the right-hand side, it is easily seen from \eqref{E:g_z:1} and \eqref{E:g_z:2} that
\begin{equation}
\begin{split}
&\int_\Omega \nabla \mu^\gamma \cdot \nabla \phi^\gamma \, \d x 
+ \frac{\Theta_0}{\varepsilon} 
\int_{\Omega} |\nabla \phi^\gamma|^2 \, \d x
+\frac{\alpha}{2}\int_\Omega \Delta \phi^\gamma \abs{\di^\gamma}^2\, \d x 
+\int_\Omega \hh^\gamma \cdot \Delta \di^\gamma \, \d x
-\alpha \int_\Omega (\phi^\gamma-\phi_{\rm cr}) \di^\gamma \cdot \Delta \di^\gamma \, \d x
\\
&\quad 
\leq 
\frac{\varepsilon}{2} \| \Delta \phi^\gamma\|_{L^2(\Omega)}^2 
+ 
\frac{\kappa}{2} \| \Delta \di^\gamma\|_{L^2(\Omega)}^2 
+
K_0 \| \nabla \mu^\gamma\|_{L^2(\Omega)}
+ \frac{1}{2\kappa} \| \hh^\gamma\|_{L^2(\Omega)}^2
+C(\alpha, \varepsilon, \kappa, \Theta_0, \mathcal{D}_\infty, K_0, |\Omega|).
\end{split}
\end{equation}
Besides, we observe that
\begin{equation}
\label{crucial}
    \begin{split}
        &\beta \int_\Omega \diver\prt{(\nabla \phi^\gamma \cdot \di^\gamma)\di^\gamma} \Delta \phi^\gamma \, \d x
        - \beta \int_\Omega (\di^\gamma \cdot \nabla \phi^\gamma) \nabla \phi^\gamma \cdot \Delta \di^\gamma \, \d x
        \\
        &= \beta \int_\Omega \partial_k \left( (\nabla \phi^\gamma\cdot \di^\gamma) \di_j^\gamma \right) \partial_k \partial_j \phi^\gamma \, \d x
        + \beta \int_\Omega \partial_k\left( (\nabla \phi^\gamma\cdot \di^\gamma) \partial_j \phi^\gamma \right) \partial_k \di_j^\gamma \, \d x
        \\
        &= \beta \int_\Omega \partial_k (\nabla \phi^\gamma\cdot \di^\gamma) \di^\gamma_j \partial_k \partial_j \phi^\gamma \, \d x+ \beta \int_\Omega
        (\nabla \phi^\gamma\cdot \di^\gamma) \partial_k \di^\gamma_j \partial_k \partial_j \phi^\gamma \, \d x 
        + \beta \int_\Omega \partial_k(\nabla \phi^\gamma \cdot \di^\gamma) \partial_j \phi^\gamma \partial_k \di^\gamma_j \, \d x
        \\
        &\quad +
        \beta \int_\Omega (\nabla \phi^\gamma \cdot \di^\gamma) \partial_k \partial_j \phi^\gamma \partial_k \di^\gamma_j \, \d x 
        \\
        &=\beta \int_\Omega |\nabla (\nabla \phi^\gamma\cdot \di^\gamma)|^2 \, \d x + 2 \beta \int_\Omega (\nabla \phi^\gamma \cdot \di^\gamma) \nabla \di^\gamma: D^2 \phi^\gamma \, \d x. 
    \end{split}
\end{equation}
It is important to notice that the first term, which is troublesome due to the product of the two terms with double derivatives on $\phi^\gamma$, is rewritten as a positive term and a lower order term in terms of derivative acting on $\phi^\gamma$. {\color{black} In this step, the periodic boundary conditions plays a crucial role to avoid boundary integral terms.}
Let $n=3$. By \eqref{GN} and \eqref{E:g_z:1}, we find
\begin{equation}
\label{E:g_z:4}
    \begin{split}
        2 \beta \left| \int_\Omega (\nabla \phi^\gamma \cdot \di^\gamma) \nabla \di^\gamma: D^2 \phi^\gamma \, \d x \right| 
        &\leq 2\beta \| \nabla \phi^\gamma\|_{L^4(\Omega)} \| \di^\gamma\|_{L^\infty(\Omega)} \| \nabla \di^\gamma\|_{L^4(\Omega)} \| D^2 \phi^\gamma\|_{L^2(\Omega)}
        \\
        & \leq 2 \beta C_\Omega^2 \mathcal{D}_\infty^\frac32 \| \phi^\gamma\|_{H^2(\Omega)}^\frac32 \| \di^\gamma \|_{H^2(\Omega)}^\frac12 
        \\
        &\leq 2 \beta C_\Omega^2 \mathcal{D}_\infty^\frac32 
        \left( \frac{3^\frac34}{4} \| \phi^\gamma\|_{H^2(\Omega)}^2 +
        \frac{3^\frac34}{4} \| \di^\gamma\|_{H^2(\Omega)}^2 \right)
        \\
        &\leq \frac{3^\frac34}{2}\beta C_\Omega^2 \mathcal{D}_\infty^\frac32 
        \left( \| \Delta \phi^\gamma\|_{L^2(\Omega)}^2 +
         \| \Delta \di^\gamma\|_{L^2(\Omega)}^2 \right)
        +C(\beta,\Omega, \mathcal{D}_\infty).
    \end{split}
\end{equation}
Thus, we end up with 
\begin{equation}
\label{E:g_z:5}
\begin{split}
&
\frac12 \left( \min \lbrace \varepsilon, \kappa \rbrace
-3^\frac34\beta C_\Omega^2 \mathcal{D}_\infty^\frac32 \right)
\left( \norm{\Delta \phi^\gamma}_{L^2(\Omega)}^2 
+ \norm{\Delta \di^\gamma}_{L^2(\Omega)}^2
\right)
+ \gamma \int_{\Omega} \abs{D^2 \phi^\gamma}^2 \abs{\nabla \phi^\gamma}^2 + \frac{1}{4} \left| \nabla\left(\abs{\nabla \phi^\gamma}^2\right) \right|^2 \, \d x
\\
&\quad \leq K_0 \| \nabla \mu^\gamma\|_{L^2(\Omega)}
+ \frac{1}{2\kappa} \| \hh^\gamma\|_{L^2(\Omega)}^2
+C(\alpha, \varepsilon, \kappa, \Theta_0, \mathcal{D}_\infty, K_0, |\Omega|).
\end{split}
\end{equation}
Let $n=2$. It is clear that \eqref{E:g_z:4} holds in this case as well. On the other hand, by exploiting \eqref{LADY} and \eqref{E:g_z:1}, we also have
\begin{equation}
\label{E:g_z:6}
    \begin{split}
        2 \beta \left| \int_\Omega (\nabla \phi^\gamma \cdot \di^\gamma) \nabla \di^\gamma: D^2 \phi^\gamma \, \d x \right| 
        &\leq 2\beta \| \nabla \phi^\gamma\|_{L^4(\Omega)} \| \di^\gamma\|_{L^\infty(\Omega)} \| \nabla \di^\gamma\|_{L^4(\Omega)} \| D^2 \phi^\gamma\|_{L^2(\Omega)}
        \\
        & \leq 2 \beta \widetilde{C_\Omega}^2 \mathcal{D}_\infty 
        \| \nabla \phi^\gamma\|_{L^2(\Omega)}^\frac12 
        \| \phi^\gamma\|_{H^2(\Omega)}^\frac32 
        \| \nabla \di^\gamma\|_{L^2(\Omega)}^\frac12
        \| \di^\gamma \|_{H^2(\Omega)}^\frac12
        \\
        & \leq \frac{3^\frac34}{2} \beta \widetilde{C_\Omega}^2 \mathcal{D}_\infty 
        \frac{\left(E_{\rm tot}^\gamma(\uu, \phi, \di) -(2\pi)^n E_0\right)^\frac12}{\varepsilon^\frac14\kappa^\frac14}
        \left( 
        \|\phi^\gamma \|_{H^2(\Omega)}^2
        +\|\di^\gamma \|_{H^2(\Omega)}^2
        \right)
        \\
        & \leq \frac{3^\frac34}{2} \beta \widetilde{C_\Omega}^2 \mathcal{D}_\infty 
        \frac{\left(E_{\rm tot}^\gamma(\uu_0, \phi_0^\gamma, \di_0) -(2\pi)^n E_0 \right)^\frac12}{\varepsilon^\frac14\kappa^\frac14}
        \left( 
        \|\phi^\gamma \|_{H^2(\Omega)}^2
        +\|\di^\gamma \|_{H^2(\Omega)}^2
        \right)
        \\
        & \leq \frac{3^\frac34}{2} \beta \widetilde{C_\Omega}^2 \mathcal{D}_\infty 
        \frac{\left( E_{\rm tot}^\gamma(\uu_0, \phi_0^\gamma, \di_0)-(2\pi)^nE_0 \right)^\frac12}{\varepsilon^\frac14\kappa^\frac14}
        \left( 
        \| \Delta \phi^\gamma \|_{L^2(\Omega)}^2
        +\|\Delta \di^\gamma \|_{L^2(\Omega)}^2
        \right)
        \\
        &\quad 
        + C(\beta, \varepsilon, \kappa, \mathcal{D}_\infty, K_0, \Omega).
    \end{split}
\end{equation}
As a consequence, for a general $\varrho<1$ and $\gamma \in (0,\gamma_0)$, we deduce that
\begin{equation}
\label{E:g_z:7}
\begin{split}
&
\frac12 \left( \min \lbrace \varepsilon, \kappa \rbrace
-3^\frac34\beta \widetilde{C_\Omega}^2 \mathcal{D}_\infty 
        \frac{\left( E_{\rm tot}(\uu_0, \phi_0, \di_0) -(2\pi)^n E_0 +\varrho\right)^\frac12}{\varepsilon^\frac14\kappa^\frac14} \right)
\left( \norm{\Delta \phi^\gamma}_{L^2(\Omega)}^2 
+ \norm{\Delta \di^\gamma}_{L^2(\Omega)}^2
\right)
\\
&\quad 
+ \gamma \int_{\Omega} \abs{D^2 \phi^\gamma}^2 \abs{\nabla \phi^\gamma}^2 \, \d x
+ \frac{\gamma}{4} \int_{\Omega}\left| \nabla\left(\abs{\nabla \phi^\gamma}^2\right) \right|^2 \, \d x
\\
& \qquad 
\leq K_0 \| \nabla \mu^\gamma\|_{L^2(\Omega)} 
+ \frac{1}{2\kappa} \| \hh^\gamma\|_{L^2(\Omega)}^2
+C(\alpha, \varepsilon, \kappa, \Theta_0, \mathcal{D}_\infty, K_0, |\Omega|).
\end{split}
\end{equation}
Owing to the above estimates \eqref{E:g_z:5} and \eqref{E:g_z:7}, the main assumptions
in \eqref{Key-ass} entail that
\begin{equation}
\label{E:g_z:8}
    \| \phi^\gamma\|_{L_{\rm uloc}^2([0,\infty); H^2(\Omega))}
    + \| \di^\gamma\|_{L_{\rm uloc}^2([0,\infty); H^2(\Omega))}
    +\sqrt{\gamma} 
    \| |\nabla\phi^\gamma|^2 \|_{L_{\rm uloc}^2([0,\infty);H^1(\Omega))}
    \leq K_1   
\end{equation}
for any $\gamma \in (0,\gamma_0)$ ($\gamma_0$ chosen sufficiently small such that \eqref{diff-ener} holds),
where $K_1=C(\alpha, \varepsilon, \kappa, \Theta_0, \mathcal{D}_\infty, K_0, |\Omega|)$ and it is independent on $\gamma$.
Arguing similarly as in the proof of Theorem \ref{MR:gamma-p}, it follows that
\begin{align}
\label{E:g_z:9}
&\| \textcolor{black}{\mu^\gamma} \|_{L_{\rm uloc}^2([0,\infty); H^1(\Omega))}
    + 
    \| F'(\textcolor{black}{\phi^\gamma}) \|_{L_{\rm uloc}^2([0,\infty); L^2(\Omega))}
    \leq K_2,
\\
\label{E:g_z:10}
    &\| \partial_t \uu^\gamma \|_{L_{\rm uloc}^\frac{4}{n}([0,\infty); \dot{H}^1_\sigma(\Omega)')}
    + 
    \| \partial_t \phi^\gamma \|_{L_{\rm uloc}^2([0,\infty); H^1(\Omega)')}
    +
    \| \partial_t \di^\gamma \|_{L_{\rm uloc}^2([0,\infty); L^2(\Omega))}
    \leq K_2,
\end{align}
where $K_2=C(\alpha, \varepsilon, \kappa, \Theta_0, \mathcal{D}_\infty, K_0, |\Omega|)$. If $n=2$, one can show as in the previous proof that
$\norm{\partial_t \uu^N}_{L^2(0,T; \dot{H}^1_\sigma(\Omega)')}\leq K_2$.
Thanks to the estimates \eqref{E:g_z:1}, \eqref{E:g_z:2}, \eqref{E:g_z:8}, \eqref{E:g_z:9} and \eqref{E:g_z:10}, classical compactness and diagonalization arguments yield that there exists a limit quintet $(\uu,\phi,\mu,\di,\hh)$ satisfying \eqref{W:regularity}, \eqref{W:weak-s:u}-\eqref{W:weak-s:d} and \eqref{MR:energy:in}.
We leave this part to the reader since it follows from the lines of the proof of Theorem \ref{MR:gamma-p} (see also, e.g., \cite[Proof of Theorem 2.1]{giorgini2018cahn} or \cite[Proof of Theorem 1.1]{giorgini2022existence}).
\end{proof}

\section{Uniqueness of weak solutions in two dimensions}
\label{sec:uniqueness_2D}

\begin{proof}[Proof of Theorem \ref{MR} - Uniqueness of weak solutions if $n=2$.]
Let $(\uu_1, \phi_1,\di_1, \mu_1, \hh_1)$ and $(\uu_2, \phi_2,\di_2, \mu_2, \hh_2)$ be two global weak solutions as in Theorem \ref{MR} originating from the initial data 
$(\uu_0^1,\phi_0^1, \di_0^1)$ and $(\uu_0^2,\phi_0^2,\di_0^2)$, respectively. Without loss of generality, we assume that $\overline{\phi_0^1}=\overline{\phi_0^2}$.  We define 
$\UU= \uu_1-\uu_2$, $\Phi= \phi_1-\phi_2$ and  
$\DD= \di_1-\di_2$. Notice that, since the total mass of the concentration is preserved by the flow, we have that $\overline{\Phi}(t)=0$ for all $t \in [0,\infty)$.
By exploiting \eqref{USEFUL}, they solve
\begin{align}\label{eq:NSE:uni}
    &\langle\partial_t \UU, \vv \rangle_{H^1_\sigma(\Omega)}  + 
     \left(\uu_1 \cdot \nabla \UU, \vv \right) 
     + \left( \UU \cdot \nabla \uu_2, \vv\right) +\left(\nu(\phi_1) D \UU, \nabla \vv\right) +\left( (\nu(\phi_1)-\nu(\phi_2)) D \uu_2, \nabla \vv \right)\\
     \nonumber
    &=\varepsilon\left(\nabla \phi_1 \otimes\nabla \Phi + \nabla \Phi \otimes\nabla \phi_2, \nabla \vv \right)
     +\kappa \left(\nabla \di_1 \odot \nabla \DD + \nabla \DD \odot \nabla \di_2, \nabla \vv \right) \\
     \nonumber
     &\quad +\beta 
     \left((\di_1\cdot \nabla \phi_1) \nabla \phi_1 \otimes\di_1 -(\di_2\cdot \nabla \phi_2) \nabla \phi_2 \otimes\di_2, \nabla \vv \right),\\
      \label{eq:CH:uni}
     &\langle \partial_t \Phi, v \rangle_{H^1(\Omega)} + 
    \left(\uu_1 \cdot \nabla \Phi, v\right) + 
    \left( \UU \cdot \nabla \phi_2,v \right) \\
    \nonumber
    &= -\left( \nabla \left( -\varepsilon \Delta \Phi+  
     \frac{1}{\varepsilon}\left(\Psi'(\phi_1)-\Psi'(\phi_2)\right) 
     -\frac{\alpha}{2} \left(\abs{\vect{d}_1}^2 -\abs{\di_2}^2\right)-\beta \diver\left((\nabla \phi_1\cdot \di_1)\di_1 -(\nabla \phi_2 \cdot \di_2)\di_2\right)
     \right), \nabla v \right),
\end{align}    
for any $\vv \in H^1_\sigma(\Omega)$, $v \in H^1(\Omega)$, respectively, and for almost every $t \in (0,\infty)$, as well as
 \begin{equation}
 \label{eq:d:uni}
     \begin{aligned}
&\partial_t \DD + 
\uu_1 \cdot \nabla \DD + 
\UU \cdot \nabla \di_2 - 
\kappa \Delta \DD 
+ \alpha \left(\abs{\di_1}^2 \di_1- \abs{\di_2}^2 \di_2\right) \\
&\quad -\alpha \left((\phi_1-\phi_{\rm cr}) \di_1 -(\phi_2-\phi_{\rm cr})\di_2\right)+ \beta (\di_1\cdot \nabla \phi_1)\nabla \phi_1 - \beta(\di_2\cdot \nabla \phi_2)\nabla \phi_2=0
\end{aligned}
 \end{equation}
 almost everywhere in $\Omega \times (0,\infty)$.
Testing \eqref{eq:NSE:uni} by $\mathbf{A}^{-1}\UU$, where $\mathbf{A}$ is the Stokes operator (see, e.g., \cite[Section 2.2]{temambook}, we obtain
$$
\begin{aligned}
\frac{1}{2}\dt\norm{\UU}_{\sharp}^2 
+ \nu_{*} \norm{\UU}_{L^2(\Omega)}^2
& \leq -\left( \left(\nu(\phi_1) -\nu(\phi_2)\right) D\uu_2, \nabla \mathbf{A}^{-1}\UU\right)
+\left(\uu_1\otimes \UU + \UU\otimes \uu_2, \nabla \mathbf{A}^{-1}\UU\right)
\\
&\quad 
+\varepsilon\left(\nabla \phi_1 \otimes\nabla \phi + \nabla \phi \otimes\nabla \phi_2, \nabla \mathbf{A}^{-1}\UU\right)+ \kappa \left(\nabla \di_1 \odot \nabla \DD + \nabla \DD \odot \nabla \di_2, \nabla \mathbf{A}^{-1}\UU\right)\\
&\quad + \beta \left( \left(\di_1 \cdot \nabla \phi_1\right)\nabla \phi_1 \otimes \di_1 - \left( \di_2 \cdot \nabla \phi_2\right) \nabla \phi_2 \otimes \di_2,\nabla \mathbf{A}^{-1}\UU\right) + \left(\UU, \nu'(\phi_1) D\mathbf{A}^{-1}\UU\nabla \phi_1\right),
\end{aligned}
$$
where $\| \UU \|_{\sharp}=
\| \nabla \mathbf{A}^{-1} \UU\|_{L^2(\Omega)}$.
Here, we have used that
\begin{align*}
\label{eq:term_V}
\left( \nu(\phi_1)D\UU, \nabla \mathbf{A}^{-1}\UU \right) &= -\left(\UU, \nu'(\phi_1) D\mathbf{A}^{-1}\UU\nabla \phi_1\right)- \frac{1}{2} \left(\UU, \nu(\phi_1)\Delta \mathbf{A}^{-1} \UU\right)\\
&\geq {\color{black} -\left(\UU, \nu'(\phi_1) D\mathbf{A}^{-1}\UU \nabla \phi_1\right)} + \nu_{*} \norm{\UU}_{L^2(\Omega)}^2,
\end{align*}
which follows from $\mathbf{A}\ff=
-\Delta \ff$ for any $\ff \in D(\mathbf{A})$, where $D(\mathbf{A})=\dot{H}^2_\sigma(\Omega)$.
Testing \eqref{eq:CH:uni} by $(-\Delta)^{-1}\Phi$, we find
\begin{equation*}
\label{eq:CH_mupltiplied}
\begin{split}
&\frac{1}{2}\dt\norm{\Phi}_{\ast}^2 
+\varepsilon\norm{\nabla \Phi}^2_{L^2(\Omega)} 
+ \frac{1}{\varepsilon}
\underbrace{\left(F'(\phi_1) -F'(\phi_2), \Phi\right)}_{\geq 0}
+ \beta \| \nabla \Phi \cdot \di_1\|_{L^2(\Omega)}^2
\\
&= \frac{\Theta_0}{\varepsilon} \| \Phi\|_{L^2(\Omega)}^2
+\frac{\alpha}{2} \left(\abs{\di_1}^2 -\abs{\di_2}^2, \Phi\right)
+ \int_\Omega \uu_1 \phi \cdot \nabla\left(-\Delta\right)^{-1} \Phi\, \d x + {\color{black} \int_\Omega \UU \phi_2 \cdot \nabla \left(-\Delta\right)^{-1}\Phi\, \d x}
\\
&\quad - \beta\int_\Omega \left(\nabla \phi_2 \cdot \DD \right)\left(\di_1\cdot \nabla \Phi\right)\, \d x - \beta \int_\Omega \left(\nabla \phi_2 \cdot \di_2\right)\left(\DD \cdot \nabla \Phi\right)\, \d x, 
\end{split}
\end{equation*}
where $\| f\|_{\ast}=\| \nabla (-\Delta)^{-1}f\|_{L^2(\Omega)}.$
Next, testing \eqref{eq:d:uni} by $\DD$, we have 
\begin{equation*}
\begin{split}
    &\frac{1}{2}\dt\norm{\DD}_{L^2(\Omega)}^2 + \kappa \norm{\nabla \DD}_{L^2(\Omega)}^2 + \underbrace{\alpha \int_\Omega {\color{black} \left(\abs{\di_1}^2\di_1  -\abs{\di_2}^2\di_2\right)} \cdot \DD\, \d x}_{\geq 0}  
    + \int_\Omega \UU\cdot \nabla \di_2 \DD \, \d x -\alpha \int_\Omega\Phi(\di_1\cdot \DD)\, \d x \\
    &-\alpha \int_\Omega (\phi_2-\phi_{\rm cr})\abs{\DD}^2\, \d x + \underbrace{\beta \int_\Omega \abs{\DD\cdot \nabla \phi_1}^2\, \d x}_{\geq 0} + \beta \int_\Omega (\di_2\cdot \nabla \Phi)(\nabla \phi_1\cdot \DD)\, \d x 
    + \beta \int_\Omega (\di_2 \cdot \nabla \phi_2)(\nabla \Phi\cdot \DD)\, \d x=0.
    \end{split}
\end{equation*}
Therefore, we arrive at
\begin{equation}
    \begin{split}
&\frac{1}{2}\dt \left( \norm{\UU}_{\sharp}^2 + \norm{\Phi}_{\ast}^2 + \norm{\DD}_{L^2(\Omega)}^2 \right) 
+ \nu_{*} \norm{\UU}_{L^2(\Omega)}^2
+ \varepsilon \| \nabla \Phi\|_{L^2(\Omega)}^2
+\kappa \| \nabla \DD \|_{L^2(\Omega)}^2
\\
& \leq- \int_\Omega \left(\nu(\phi_1) -\nu(\phi_2)\right) D\uu_2 :  \nabla \mathbf{A}^{-1}\UU\, \d x
+\int_\Omega \left(\uu_1\otimes \UU + \UU\otimes \uu_2\right) : \nabla \mathbf{A}^{-1}\UU\, \d x
\\
&\quad 
+\varepsilon \int_\Omega\left(\nabla \phi_1 \otimes\nabla \phi + \nabla \phi \otimes\nabla \phi_2\right) : \nabla \mathbf{A}^{-1}\UU\, \d x
+ \kappa \int_\Omega\left(\nabla \di_1 \odot \nabla \DD + \nabla \DD \odot \nabla \di_2 \right) : \nabla \mathbf{A}^{-1}\UU\, \d x\\
&\quad + \int_\Omega \nu'(\phi_1) D\mathbf{A}^{-1}\UU\nabla \phi_1 \cdot \UU \, \d x + \beta \int_\Omega\left(\left(\di_1 \cdot \nabla \phi_1\right)\nabla \phi_1 \otimes \di_1 - \left( \di_2 \cdot \nabla \phi_2\right) \nabla \phi_2 \otimes \di_2 \right): \nabla \mathbf{A}^{-1}\UU\, \d x
\\
&\quad 
+\frac{\Theta_0}{\varepsilon} \| \Phi\|_{L^2(\Omega)}^2
+\frac{\alpha}{2} \int_\Omega \left(\abs{\di_1}^2 -\abs{\di_2}^2\right)  \Phi \, \d x
+ \int_\Omega \uu_1 \Phi \cdot \nabla\left(-\Delta\right)^{-1} \Phi\, \d x + \int_\Omega \UU \phi_2 \cdot \nabla \left(-\Delta\right)^{-1}\Phi\, \d x
\\
&\quad - \beta\int_\Omega \left(\nabla \phi_2 \cdot \DD \right)\left(\di_1\cdot \nabla \Phi\right)\, \d x - \beta \int_\Omega \left(\nabla \phi_2 \cdot \di_2\right)\left(\DD \cdot \nabla \Phi\right)\, \d x - \int_\Omega \UU\cdot \nabla \di_2 \, \DD \, \d x
\\
& \quad  +\alpha \int_\Omega\Phi(\di_1\cdot \DD)\, \d x  +\alpha \int_\Omega (\phi_2-\phi_{\rm cr})\abs{\DD}^2\, \d x  - \beta \int_\Omega (\di_2\cdot \nabla \Phi)(\nabla \phi_1\cdot \DD)\, \d x 
    - \beta \int_\Omega (\di_2 \cdot \nabla \phi_2)(\nabla \Phi\cdot \DD)\, \d x.
    \end{split}
\end{equation}
By using \cite[Proposition C.1]{giorgini2019uniqueness} and $\nu(\cdot) \in W^{1,\infty}(\mathbb{R})$, we infer that
$$
\begin{aligned}
\left| - \int_\Omega \left(\nu(\phi_1) -\nu(\phi_2)\right) D\uu_2 :  \nabla \mathbf{A}^{-1}\UU\, \d x
\right| 
&\leq C\norm{D \uu_2}_{L^2(\Omega)} \norm{\nabla \Phi}_{L^2(\Omega)} \norm{\nabla \mathbf{A}^{-1} \UU}_{L^2(\Omega)} \ln^{\frac{1}{2}}\left(C\, \frac{\norm{\UU}_{L^2(\Omega)}}{\norm{\nabla \mathbf{A}^{-1}\UU}_{L^2(\Omega)}}\right)\\
&\leq \frac{\varepsilon}{7} \norm{\nabla \Phi}_{L^2(\Omega)}^2 + C \norm{D \uu_2}_{L^2(\Omega)}^2 \norm{\UU}_{\sharp}^2 \ln\left( \frac{C}{\norm{\UU}_{\sharp}^2}\right).
\end{aligned}
$$
By \eqref{LADY}, we have 
$$
\begin{aligned}
\left| \int_\Omega {\color{black} \left(\uu_1\otimes \UU + \UU \otimes \uu_2\right)}:\nabla \mathbf{A}^{-1}\UU\, \d x \right| 
&\leq C\left(\norm{\uu_1}_{L^2(\Omega)}^{\frac{1}{2}}\norm{\uu_1}_{H^1(\Omega)}^{\frac{1}{2}} + \norm{\uu_2}_{L^2(\Omega)}^{\frac{1}{2}}\norm{\uu_2}_{H^1(\Omega)}^{\frac{1}{2}}\right)\norm{\UU}_{\sharp}^{\frac{1}{2}}\norm{\UU}_{L^2(\Omega)}^{\frac{3}{2}}\\
&\leq \frac{\nu_{*}}{6}\norm{\UU}_{L^2(\Omega)}^2
+
C \left(\norm{\uu_1}_{H^1(\Omega)}^2 +\norm{\uu_2}_{H^1(\Omega)}^2 \right)\norm{\UU}_{\sharp}^2
\end{aligned}
$$
and
$$
\begin{aligned}
&\left|\varepsilon \int_\Omega\left(\nabla \phi_1 \otimes\nabla \phi + \nabla \phi \otimes\nabla \phi_2\right) : \nabla \mathbf{A}^{-1}\UU\, \d x
+ \kappa \int_\Omega\left(\nabla \di_1 \odot \nabla \DD + \nabla \DD \odot \nabla \di_2 \right) : \nabla \mathbf{A}^{-1}\UU\, \d x \right| 
\\
&\leq 
\frac{\nu_\ast}{8} \| \UU\|_{L^2(\Omega)}^2
+\frac{\varepsilon}{7}\norm{\nabla \phi}_{L^2(\Omega)}^2 
+ \frac{\kappa}{5} \| \nabla \DD\|_{L^2(\Omega)}^2 
+ C \left( \| \phi_1\|_{H^2(\Omega)}^2 +
\| \phi_2\|_{H^2(\Omega)}^2 +
\| \di_1\|_{H^2(\Omega)}^2 +
\| \di_2\|_{H^2(\Omega)}^2 \right)\norm{\uu}_{\sharp}^2.
\end{aligned}
$$
Recalling that $\nu(\cdot) \in W^{1,\infty}(\mathbb{R})$, and {\color{black} exploiting} again \eqref{LADY}, we also find 
\begin{align*}
\left| \int_\Omega \nu'(\phi_1) D\mathbf{A}^{-1}\UU\nabla \phi_1 \cdot \UU \, \d x \right| 
\leq \frac{\nu_{*}}{6} \norm{\UU}_{L^2(\Omega)}^2 + C \| \phi_1\|_{H^2(\Omega)}^2 \norm{\UU}_{\sharp}^2
\end{align*}
and
$$
\begin{aligned}
&\left| \beta \int_\Omega\left(\left(\di_1 \cdot \nabla \phi_1\right)\nabla \phi_1 \otimes \di_1 - \left( \di_2 \cdot \nabla \phi_2\right) \nabla \phi_2 \otimes \di_2 \right): \nabla \mathbf{A}^{-1}\UU\, \d x
\right| \\
&\leq \left| \beta \int_\Omega \left(\di_1 \cdot \nabla \phi_1\right) 
\nabla \phi_1 \otimes \DD : \nabla \mathbf{A}^{-1}\UU \, \d x \right|
+ \left| \beta \int_\Omega \left(\di_1\cdot \nabla \phi_1 \right) \nabla \Phi \otimes \di_2 : \nabla \mathbf{A}^{-1}\UU \, \d x \right|
\\
&\quad +\left| \beta \int_\Omega \left( \di_1 \cdot \nabla \Phi -\DD \cdot \nabla \phi_2\right) \nabla \phi_2 \otimes \di_2 :\nabla \mathbf{A}^{-1}\UU \, \d x
\right|
\\
&\leq \beta \left(\norm{\di_1}_{L^\infty(\Omega)} \norm{\nabla \phi_1}_{L^4(\Omega)}^2  + \norm{\nabla \phi_2}_{L^4(\Omega)}^2 \norm{\di_2}_{L^\infty(\Omega)}\right)
\norm{\DD}_{L^4(\Omega)}
\norm{\nabla \mathbf{A}^{-1}\UU}_{L^4(\Omega)} \\
&\quad 
+ \beta \left(
\norm{\nabla \phi_1}_{L^4(\Omega)} +\norm{\nabla \phi_2}_{L^4(\Omega)}\right)  \norm{\di_1}_{L^\infty(\Omega)}\norm{\di_2}_{L^\infty(\Omega)}\norm{\nabla \Phi}_{L^2(\Omega)}  \norm{\nabla \mathbf{A}^{-1}\UU}_{L^4(\Omega)}\\
&\leq C \left(\| \phi_1\|_{H^2(\Omega)}+
\| \phi_2\|_{H^2(\Omega)} \right)
\norm{\DD}_{L^2(\Omega)}^{\frac{1}{2}}
\norm{\nabla \DD}_{L^2(\Omega)}^{\frac{1}{2}} \norm{\UU}_{\sharp}^{\frac{1}{2}} \norm{\UU}_{L^2(\Omega)}^{\frac{1}{2}}
\\
&\quad +C\left(\| \phi_1\|_{H^2(\Omega)}^\frac12+
\phi_2\|_{H^2(\Omega)}^\frac12\right)
\| \nabla \Phi\|_{L^2(\Omega)}
\norm{\UU}_{\sharp}^{\frac{1}{2}} \norm{\UU}_{L^2(\Omega)}^{\frac{1}{2}}
\\
&\leq \frac{\nu_{*}}{6}\norm{\UU}_{L^2(\Omega)}^2 + \frac{\varepsilon}{7}\norm{\nabla \Phi}_{L^2(\Omega)}^2 
+ \frac{\kappa}{5} \norm{\nabla \DD}_{L^2(\Omega)}^2 + C
\left( \| \phi_1\|_{H^2(\Omega)}^2
+\| \phi_2\|_{H^2(\Omega)}^2 \right)
\left(\norm{\DD}_{L^2(\Omega)}^2+ \norm{\UU}_{\sharp}^2\right).
\end{aligned}
$$
It is easily seen that
\begin{align*}
&\left| \frac{\Theta_0}{\varepsilon} \| \Phi\|_{L^2(\Omega)}^2
+\frac{\alpha}{2} \int_\Omega \left(\abs{\di_1}^2 -\abs{\di_2}^2\right)  \Phi \, \d x
 +\alpha \int_\Omega\Phi(\di_1\cdot \DD)\, \d x  +\alpha \int_\Omega (\phi_2-\phi_{\rm cr})\abs{\DD}^2\, \d x
 \right|
 \\
 &\leq 
 \frac{\varepsilon}{7} \| \nabla \Phi\|_{L^2(\Omega)}^2 + C \| \Phi\|_{\ast}^2 
 +C \| \DD\|_{L^2(\Omega)}^2
\end{align*}
and
\begin{align*}
    \left| \int_\Omega \uu_1 \Phi \cdot \nabla\left(-\Delta\right)^{-1} \Phi\, \d x + \int_\Omega \UU \phi_2 \cdot \nabla \left(-\Delta\right)^{-1}\Phi\, \d x
     \right|
      &\leq C \| \uu_1\|_{L^6(\Omega)} \| \Phi\|_{L^3(\Omega)} \| \Phi\|_{\ast}^2+
     \|\UU \|_{L^2(\Omega)} \| \Phi\|_{\ast}^2
     \\
     &\leq 
     \frac{\nu_{*}}{6}\norm{\UU}_{L^2(\Omega)}^2 + \frac{\varepsilon}{7}\norm{\nabla \Phi}_{L^2(\Omega)}^2 
    + C \norm{\Phi}_\ast^2.
\end{align*}
By exploiting \eqref{LADY} once again, we deduce that
$$
\begin{aligned}
\left| -\int_\Omega \UU \cdot \nabla \di_2 \DD\, \d x \right|
&\leq C \norm{\UU}_{L^2(\Omega)} \norm{\nabla \di_2}_{L^2(\Omega)}^{\frac{1}{2}}\norm{\di_2}_{H^2(\Omega)}^{\frac{1}{2}} \norm{\DD}_{L^2(\Omega)}^{\frac{1}{2}} \norm{\nabla \DD}_{L^2(\Omega)}^{\frac{1}{2}}\\
&\leq \frac{\nu_{*}}{6}\norm{\UU}_{L^2(\Omega)}^2 + \frac{\kappa}{5}\norm{\nabla \DD}_{L^2(\Omega)}^2 + C \norm{\di_2}_{H^2(\Omega)}^2 \norm{\DD}_{L^2(\Omega)}^2.
\end{aligned}
$$
In a similar way, we also find
$$
\begin{aligned}
&\left| - \beta\int_\Omega \left(\nabla \phi_2 \cdot \DD\right)\left(\di_1\cdot \nabla \Phi\right)\, \d x - \beta \int_\Omega \left(\nabla \phi_2 \cdot \di_2\right)\left(\DD \cdot \nabla \Phi\right)\, \d x 
\right|
 \\
 &\leq C\norm{\nabla \phi_2 }_{L^4(\Omega)}
 \norm{\DD}_{L^4(\Omega)} \left( \norm{\di_1}_{L^\infty(\Omega)}
 +\norm{\di_2}_{L^\infty(\Omega)}\right)
 \norm{\nabla \Phi}_{L^2(\Omega)}
 \\
&\leq C\norm{\phi_2}_{H^2(\Omega)}^\frac12
 \norm{\DD}_{L^2(\Omega)}^\frac12
 \norm{\nabla \DD}_{L^2(\Omega)}^\frac12
 \norm{\nabla \Phi}_{L^2(\Omega)}
 \\
&\leq\frac{\varepsilon}{7}\norm{\nabla \Phi}_{L^2(\Omega)}^2 + \frac{\kappa}{5}\norm{\nabla \DD}_{L^2(\Omega)}^2 + C \norm{\phi_2}_{H^2(\Omega)}^2 \norm{\DD}_{L^2(\Omega)}^2
\end{aligned}
$$
and 
\begin{align*}
&\left| - \beta \int_\Omega (\di_2\cdot \nabla \Phi)(\nabla \phi_1\cdot \DD)\, \d x 
    - \beta \int_\Omega (\di_2 \cdot \nabla \phi_2)(\nabla \Phi\cdot \DD)\, \d x
    \right| 
    \\
    &\leq \frac{\varepsilon}{7}\norm{\nabla \Phi}_{L^2(\Omega)}^2 + \frac{\kappa}{5}\norm{\nabla \DD}_{L^2(\Omega)}^2 + C \left( \norm{\phi_1}_{H^2(\Omega)}^2+ \norm{\phi_2}_{H^2(\Omega)}^2 \right) \norm{\DD}_{L^2(\Omega)}^2.
\end{align*}
Setting
$\mathcal{H}:= \norm{\UU}_{\sharp}^2 + \norm{\Phi}_{\ast}^2+
\norm{\DD}_{L^2(\Omega)}^2$,
we end up with the following differential inequality
\begin{equation}
    \begin{aligned}
    \frac{1}{2}\dt\mathcal{H} \leq \mathcal{F}_1 \mathcal{H}
    + \mathcal{F}_2 \norm{\UU}_{\sharp}^2 \ln\left( \frac{C}{\norm{\UU}_{\sharp}^2}\right)
    \end{aligned}
\end{equation}
almost everywhere in $(0,\infty)$, where 
$$
\begin{aligned}
&\mathcal{F}_1:= C\left(\norm{\uu_1}_{H^1(\Omega)}^2+\norm{\uu_2}_{H^1(\Omega)}^2 + \norm{\di_1}_{H^2(\Omega)}^2+ \norm{\di_2}_{H^2(\Omega)}^2 + \norm{\phi_1}_{H^2(\Omega)}^2 + \norm{\phi_2}^2_{H^2(\Omega)}\right),
\quad \mathcal{F}_2 := C \norm{D \uu_2}_{L^2(\Omega)}^2.
\end{aligned}
$$
Observing that the real function $s \ln \left( \frac{C}{s}\right)$ is increasing for $s\in (0,C)$, we rewrite the above as
\begin{equation}\label{eq:inequality_unique}
\dt\mathcal{H} \leq \mathcal{F} \mathcal{H}\ln\left(\frac{C}{\mathcal{H}}\right),
\end{equation}
where $\mathcal{F}=\mathcal{F}_1+\mathcal{F}_2 \in L^1(0,T)$ for any $T>0$. Since $g(s):= s \ln\left[\frac{C}{s}\right]$ is an {\em Osgood's modulus of continuity}, an application of the {\em Osgood lemma} to \eqref{eq:inequality_unique}
implies the uniqueness of weak solutions. This shows the second part of Theorem \ref{MR}.
\end{proof}

\section{Global strong solutions and regularity in two dimensions}
\label{sec:strong_solution&regularity}


In this section, we prove the last part of Theorem \ref{MR}, i.e. the existence of global strong solutions in two dimensions. We provide {\it a priori} higher-order energy estimates in Sobolev spaces whose rigorous justification can be carry out in an approximation scheme as for the weak solutions in Sections \ref{subsec:gamma} and \ref{sec:gamma0}.

\begin{proof}[Proof of Theorem \ref{MR} - Existence of strong solutions if $n=2$.]
Testing \eqref{NSE} by $- \Delta \uu$, we have
\begin{equation}
    \label{eq:NSE_per_partialt_u}
    \begin{aligned}
     \frac12 \dt \norm{\nabla \uu}_{L^2(\Omega)}^2 + \frac{\nu_{*}}{2}\norm{\Delta \uu}_{L^2(\Omega)}^2 \leq &
     - \int_\Omega \frac{\nu'(\phi)}{2} \left( D\uu \nabla \phi\right) \cdot \Delta \uu\, \d x 
     + \int_\Omega \left( \uu \cdot \nabla\right) \uu \cdot \Delta \uu\, \d x\\
     &+ \int_\Omega \mu \nabla \phi \cdot \prt{- \Delta \uu}\, \d x + 
     \int_\Omega (\nabla \di)^T \vect{h} \cdot \prt{- \Delta \uu}\, \d x.
    \end{aligned}
\end{equation}
Next, we multiply \eqref{var_phi} by $\partial_t \mu$ and integrate over $\Omega$.
Observing that
$$
\begin{aligned}
\left(\partial_t \phi, \partial_t \mu\right)  
&= -\varepsilon \left(\partial_t \phi, \Delta \partial_t \phi\right)
+ \frac{1}{\varepsilon} \left(\partial_t \phi,  F''(\phi)\partial_t \phi\right) - \frac{\Theta_0}{\varepsilon}\| \partial_t \phi\|_{L^2(\Omega)}^2
- \alpha  \left(\partial_t \phi, \di\cdot \partial_t \di\right) -\beta \left(\partial_t \phi,\partial_t\diver\left(\left(\nabla \phi\cdot \di\right)\di\right) \right)\\
&= \varepsilon\norm{\nabla \partial_t \phi}_{L^2(\Omega)}^2 
+ \frac{1}{\varepsilon} \int_\Omega F''(\phi) \abs{\partial_t \phi}^2 \, \d x 
 - \frac{\Theta_0}{\varepsilon}\| \partial_t \phi\|_{L^2(\Omega)}^2
 - \alpha \int_\Omega \partial_t \phi \left(\di\cdot \partial_t \di\right)\, \d x
\\
&\quad + 
\beta \int_\Omega \left| \nabla \partial_t \phi \cdot \di\right|^2 \, \d x
+\beta \int_\Omega \left(\nabla \partial_t \phi \cdot \di \right) \left( \nabla \phi\cdot \partial_t \di \right) \, \d x
+\beta \int_\Omega \left(\nabla \partial_t \phi \cdot \partial_t \di \right) \left( \nabla \phi\cdot \di \right) \, \d x
\end{aligned}
$$
and
$$
\begin{aligned}
\left(\uu\cdot \nabla \phi, \partial_t \mu\right) = \dt\Big( \prt{\uu\cdot \nabla \phi, \mu}\Big) -\prt{\partial_t \uu\cdot \nabla \phi, \mu}-\prt{\uu\cdot \nabla \partial_t \phi, \mu},
\end{aligned}
$$
we find
\begin{equation}
\label{UU:testCH}
\begin{split}
    &\dt\left( \frac{1}{2}\norm{\nabla \mu}_{L^2(\Omega)}^2 +\prt{\uu\cdot \nabla \phi, \mu}\right) 
    + \varepsilon\norm{\nabla \partial_t \phi}_{L^2(\Omega)}^2
    + \frac{1}{\varepsilon} \int_\Omega F''(\phi) \abs{\partial \phi}^2 \, \d x
   + \beta \int_\Omega \left| \nabla \partial_t \phi \cdot \di\right|^2 \, \d x
    \\
    &=
      \frac{\Theta_0}{\varepsilon}\| \partial_t \phi\|_{L^2(\Omega)}^2
    + \alpha \int_\Omega \partial_t \phi \left(\di\cdot \partial_t \di\right)\, \d x -\beta \int_\Omega \left(\nabla \partial_t \phi \cdot \di \right) \left( \nabla \phi\cdot \partial_t \di \right) \, \d x
\\
&\quad -\beta \int_\Omega \left(\nabla \partial_t \phi \cdot \partial_t \di \right) \left( \nabla \phi\cdot \di \right) \, \d x
+\int_\Omega \partial_t \uu\cdot \nabla \phi\,  \mu \,\d x 
+\int_{\Omega} \uu\cdot \nabla \partial_t \phi \, \mu \, \d x.
    \end{split}
\end{equation}
Now, differentiating \eqref{var_d} in time,  multiplying by $\partial_t \di$ and integrating over $\Omega$, we obtain
\begin{equation}
\label{UU:testLC}
\begin{split}
    &\frac12 \frac{{\rm d}}{{\rm d}t}\| \partial_t \di\|_{L^2(\Omega)}^2 
    +\kappa \| \nabla \partial_t \di \|_{L^2(\Omega)}^2
    + 2\alpha \int_\Omega \left| \di \cdot \partial_t \di \right|^2 \, \d x 
    +\alpha \int_\Omega |\di |^2 |\partial_t \di|^2 \, \d x
    + \beta \int_{\Omega} |\partial_t \di \cdot \nabla \phi|^2 \, \d x
    \\
    &=-\int_{\Omega} \left( \partial_t \uu \cdot \nabla \right) \di \cdot \partial_t \di \, \d x  
    +\alpha \int_\Omega \partial_t \phi \, \di \cdot \partial_t \di \, \d x
    +\alpha \int_\Omega \left( \phi-\phi_{\rm cr}\right) |\partial_t \di |^2 \, \d x
    \\
    & \quad - \beta \int_{\Omega} \left( \di \cdot \nabla \partial_t \phi\right) \left( \nabla \phi \cdot \partial_t \di \right) \, \d x
    - \beta \int_{\Omega} \left( \di \cdot \nabla \phi\right) \left(\nabla \partial_t \phi \cdot \partial_t \di \right) \, \d x.
    \end{split}
\end{equation}
Thus, adding up \eqref{eq:NSE_per_partialt_u}, \eqref{UU:testCH} and \eqref{UU:testLC}, we arrive at 
\begin{equation*}
\begin{split}
    &\dt\left( \frac12 \| \nabla \uu\|_{L^2(\Omega)}^2 + \frac{1}{2}\norm{\nabla \mu}_{L^2(\Omega)}^2 
    +\frac12 \| \partial_t \di\|_{L^2(\Omega)}^2
    +\prt{\uu\cdot \nabla \phi, \mu}\right) 
    +\nu_\ast \| \Delta \uu\|_{L^2(\Omega)}^2
    + \varepsilon\norm{\nabla \partial_t \phi}_{L^2(\Omega)}^2 + \kappa \|\nabla \partial_t \di \|_{L^2(\Omega)}^2
    \\
    &\quad + \frac{1}{\varepsilon} \int_\Omega F''(\phi) \abs{\partial_t \phi}^2 \, \d x
    + 2\alpha \int_\Omega \left| \di \cdot \partial_t \di \right|^2 \, \d x 
    +\alpha \int_\Omega |\di |^2 |\partial_t \di|^2 \, \d x
   + \beta \int_\Omega \left| \nabla \partial_t \phi \cdot \di\right|^2 \, \d x
   + \beta \int_{\Omega} |\partial_t \di \cdot \nabla \phi|^2 \, \d x
   \\
   &
   \leq - \int_\Omega \frac{\nu'(\phi)}{2} \left( D\uu \nabla \phi\right) \cdot \Delta \uu\, \d x 
     + \int_\Omega \left( \uu \cdot \nabla\right) \uu \cdot \Delta \uu\, \d x
    + \int_\Omega \mu \nabla \phi \cdot \prt{- \Delta \uu}\, \d x +\int_\Omega (\nabla \di)^T \vect{h}\cdot \prt{- \Delta \uu}\, \d x
    \\
    &\quad 
    + \frac{\Theta_0}{\varepsilon}\| \partial_t \phi\|_{L^2(\Omega)}^2
    + 2 \alpha \int_\Omega \partial_t \phi \left(\di\cdot \partial_t \di\right)\, \d x
    +\alpha \int_\Omega \left( \phi-\phi_{\rm cr}\right) |\partial_t \di |^2 \, \d x
    +\int_\Omega \partial_t \uu\cdot \nabla \phi\,  \mu \,\d x 
    +\int_{\Omega} \uu\cdot \nabla \partial_t \phi \, \mu \, \d x
\\
&\quad -\int_{\Omega} \left( \partial_t \uu \cdot \nabla \right) \di \cdot \partial_t \di \, \d x   
 -2\beta \int_\Omega \left(\nabla \partial_t \phi \cdot \di \right) \left( \nabla \phi\cdot \partial_t \di \right) \, \d x
-2\beta \int_\Omega \left(\nabla \partial_t \phi \cdot \partial_t \di \right) \left( \nabla \phi\cdot \di \right) \, \d x.
\end{split}
\end{equation*}
We proceed in estimating the terms on the right-hand side of the above differential inequality. In the sequel, $C$ stands for a generic positive constant depending on the parameters of the system and the norms of the inidial data. By using \eqref{GN} and \eqref{LADY} together with \eqref{W:regularity}, we find
$$
\begin{aligned}
\left| \int_\Omega \frac{\nu'(\phi)}{2}
\left( D\uu \nabla \phi\right)\cdot  \Delta \uu\, \d x \right| &\leq \norm{\frac{\nu'(\phi)}{2}}_{L^\infty(\Omega)} 
\norm{\nabla \phi}_{L^4(\Omega)}
\norm{D\uu}_{L^4(\Omega)} 
\norm{\Delta \uu}_{L^2(\Omega)} 
\\
&\leq C\norm{\phi}_{L^\infty(\Omega)}^\frac12 \|\phi \|_{H^2(\Omega)}^\frac12 
\norm{\nabla \uu}_{L^2(\Omega)}^\frac12 
\norm{\Delta \uu}_{L^2(\Omega)}^\frac32\\
&\leq \frac{\nu_{*}}{8} \norm{\Delta \uu}_{L^2(\Omega)}^2 
+ C \|\phi \|_{H^2(\Omega)}^2\norm{\nabla \uu}_{L^2(\Omega)}^2.
\end{aligned}
$$
It is well-known that
$$
\begin{aligned}
\left| \int_\Omega \left(\uu \cdot \nabla \right) \uu \cdot \Delta \uu\, \d x\right|  
&\leq \frac{\nu_{*}}{8} \norm{\Delta \uu}_{L^2(\Omega)}^2 + C \norm{\nabla \uu}_{L^2(\Omega)}^4
\end{aligned}
$$
and
$$
\begin{aligned}
\left| \int_\Omega \mu \nabla \phi \cdot \prt{- \Delta \uu}\, \d x \right| 
&=\left| \int_\Omega \phi\nabla \mu \cdot \prt{- \Delta \uu}\, \d x\right|
\leq \frac{\nu_{*}}{8} \norm{\Delta \uu}_{L^2(\Omega)}^2 
+ C \norm{\nabla \mu}_{L^2(\Omega)}^2.
\end{aligned}
$$
Since $\| \partial_t \phi\|_{H^1(\Omega)'}
\leq C \left( \| \uu\|_{L^2(\Omega)}+ \| \nabla \mu\|_{L^2(\Omega)}\right)$, we infer by Hilbert interpolation that 
\begin{equation}
\begin{split}
    \left| \frac{\Theta_0}{\varepsilon}\| \partial_t \phi\|_{L^2(\Omega)}^2
    + 2 \alpha \int_\Omega \partial_t \phi \left(\di\cdot \partial_t \di\right)\, \d x\right| 
    &\leq C \| \partial_t \phi\|_{H^1(\Omega)'} \| \nabla \partial_t \phi\|_{L^2(\Omega)} +
    C \| \partial_t \phi\|_{L^2(\Omega)} 
    \| \di\|_{L^\infty(\Omega)}
    \| \partial_t \di\|_{L^2(\Omega)}
    \\
    &\leq \frac{\varepsilon}{6} \| \nabla \partial_t \phi\|_{L^2(\Omega)}^2
    + C \left( 1+ \| \nabla \mu\|_{L^2(\Omega)}^2 + \| \partial_t \di\|_{L^2(\Omega)}^2\right).
    \end{split}
\end{equation}
Also, it is immediately seen that
\begin{equation}
    \left| \alpha \int_\Omega \left( \phi-\phi_{\rm cr}\right) |\partial_t \di |^2 \, \d x \right|
    \leq C \| \partial_t \di\|_{L^2(\Omega)}^2.
\end{equation}
By exploiting \eqref{var_d}, \eqref{GN} and \eqref{LADY}, we have
\begin{equation}
\label{est-dhu}
\begin{split}
&\left| \int_\Omega (\nabla \di)^T \vect{h}\cdot \prt{- \Delta \uu}\, \d x\right|
\\
&\leq \norm{\nabla \di}_{L^4(\Omega)} \norm{\vect{h}}_{L^4(\Omega)} \norm{\Delta \uu}_{L^2(\Omega)}\\
&\leq \frac{\nu_{*}}{16} \norm{\Delta \uu}_{L^2(\Omega)}^2 + C \norm{ \di}_{L^\infty(\Omega)}
\norm{\Delta \di}_{L^2(\Omega)} \norm{\vect{h}}_{L^2(\Omega)}
\norm{\vect{h}}_{H^1(\Omega)}\\
&\leq \frac{\nu_{*}}{16} \norm{\Delta \uu}_{L^2(\Omega)}^2 + C \norm{\Delta \di}_{L^2(\Omega)} \norm{\vect{h}}_{L^2(\Omega)}^2 +C \norm{\Delta \di}_{L^2(\Omega)} \norm{\vect{h}}_{L^2(\Omega)} \norm{\nabla \prt{\partial_t \di + \uu\cdot \nabla \di}}_{L^2(\Omega)}\\
&\leq \frac{\nu_{*}}{16} \norm{\Delta \uu}_{L^2(\Omega)}^2 + \frac{\kappa}{6} \norm{\nabla\partial_t \di}_{L^2(\Omega)}^2 + C \left(1+ \norm{\Delta \di}_{L^2(\Omega)}^2\right)\norm{\vect{h}}^2_{L^2(\Omega)} 
\\
&\quad 
+ C \norm{\Delta \di}_{L^2(\Omega)}
\norm{\vect{h}}_{L^2(\Omega)}  
\norm{\nabla \uu}_{L^4(\Omega)}
\norm{\nabla \di}_{L^4(\Omega )} + C \norm{\Delta \di}_{L^2(\Omega)}^2
\norm{\vect{h}}_{L^2(\Omega)}  
\norm{\uu}_{L^\infty(\Omega)}
\\
&\leq \frac{\nu_{*}}{16} \norm{\Delta \uu}_{L^2(\Omega)}^2 + \frac{\kappa}{6} \norm{\nabla\partial_t \di}_{L^2(\Omega)}^2 + C \left(1+ \norm{\Delta \di}_{L^2(\Omega)}^2\right)\norm{\vect{h}}^2_{L^2(\Omega)}  
\\
&\quad 
+ C \norm{\Delta \di}_{L^2(\Omega)}^\frac32
\norm{\vect{h}}_{L^2(\Omega)}  
\norm{\nabla \uu}_{L^2(\Omega)}^\frac12
\norm{\Delta \uu}_{L^2(\Omega)}^\frac12
+ C \norm{\Delta \di}_{L^2(\Omega)}^2
\norm{\vect{h}}_{L^2(\Omega)}  
\norm{\Delta \uu}_{L^2(\Omega)}^\frac12
\\
&\leq \frac{\nu_{*}}{8} \norm{\Delta \uu}_{L^2(\Omega)}^2 + \frac{\kappa}{6} \norm{\nabla\partial_t \di}_{L^2(\Omega)}^2 + C \left(1+ \norm{\Delta \di}_{L^2(\Omega)}^2\right)\norm{\vect{h}}^2_{L^2(\Omega)}  
\\
&\quad 
+ C \norm{\Delta \di}_{L^2(\Omega)}^2
\norm{\vect{h}}_{L^2(\Omega)}^\frac43  
\norm{\nabla \uu}_{L^2(\Omega)}^\frac23
+C \norm{\Delta \di}_{L^2(\Omega)}^\frac83
\norm{\vect{h}}_{L^2(\Omega)}^\frac43
\\
&\leq \frac{\nu_{*}}{8} \norm{\Delta \uu}_{L^2(\Omega)}^2 + \frac{\kappa}{6} \norm{\nabla\partial_t \di}_{L^2(\Omega)}^2 + C \left(1+ \norm{\Delta \di}_{L^2(\Omega)}^2\right)
\left(\norm{\vect{h}}_{L^2(\Omega)}^2+
\norm{\nabla \uu}_{L^2(\Omega)}^2 +
\norm{\Delta \di}_{L^2(\Omega)}^2 \right).
\end{split}
\end{equation}
Similarly, we obtain
\begin{equation}
    \begin{split}
       \left| \int_\Omega \partial_t \uu\cdot \nabla \phi\,  \mu \,\d x 
    +\int_{\Omega} \uu\cdot \nabla \partial_t \phi \, \mu \, \d x \right|
    &= \left| -\int_\Omega \partial_t \uu\cdot \nabla \mu\,  \phi \,\d x 
    +\int_{\Omega} \uu\cdot \nabla \partial_t \phi \, \mu \, \d x \right|
    \\
    &\leq \frac{\varpi}{2} \|\partial_t \uu \|_{L^2(\Omega)}^2
    +\frac{\varepsilon}{6}\| \nabla \partial_t \phi\|_{L^2(\Omega)}^2
    + C\left( 1+\| \uu\|_{H^1(\Omega)}^2\right) \left(1+\| \nabla \mu\|_{L^2(\Omega)}^2\right)
    \end{split}
\end{equation}
and
\begin{equation}
\begin{split}
    \left|-\int_{\Omega} \left( \partial_t \uu \cdot \nabla \right) \di \cdot \partial_t \di \, \d x \right|
    &\leq \| \partial_t \uu\|_{L^2(\Omega)}
    \| \nabla \di\|_{L^4(\Omega)}
    \| \partial_t \di\|_{L^4(\Omega)}
    \\
    &\leq C\| \partial_t \uu\|_{L^2(\Omega)}
    \| \Delta \di\|_{L^2(\Omega)}^\frac12
    \| \partial_t \di\|_{L^2(\Omega)}^\frac12
    \| \partial_t \di\|_{H^1(\Omega)}^\frac12\\
    &\leq \frac{\varpi}{2} \|\partial_t \uu \|_{L^2(\Omega)}^2 + \frac{\kappa}{6}\| \nabla \partial_t \di\|_{L^2(\Omega)}^2
    +C \left( 1+ \| \Delta \di\|_{L^2(\Omega)}^2\right) \| \partial_t \di\|_{L^2(\Omega)}^2,
\end{split}
\end{equation}
where $\varpi$ is a sufficiently small parameter whose value will be determined later on.
Lastly, we find
\begin{equation}
\begin{split}
&\left| -2\beta \int_\Omega \left(\nabla \partial_t \phi \cdot \di \right) \left( \nabla \phi\cdot \partial_t \di \right) \, \d x
-2\beta \int_\Omega \left(\nabla \partial_t \phi \cdot \partial_t \di \right) \left( \nabla \phi\cdot \di \right) \, \d x \right|
\\
&\leq C \|\nabla \partial_t \phi \|_{L^2(\Omega)} \| \di\|_{L^\infty(\Omega)}
\| \partial_t \di\|_{L^4(\Omega)} 
\| \nabla \phi\|_{L^4(\Omega)}
\\
&\leq C\|\nabla \partial_t \phi \|_{L^2(\Omega)} 
\| \partial_t \di\|_{L^2(\Omega)}^\frac12
    \| \partial_t \di\|_{H^1(\Omega)}^\frac12
\| \phi\|_{H^2(\Omega)}^\frac12
\\
&\leq \frac{\varepsilon}{6}\| \nabla \partial_t \phi\|_{L^2(\Omega)}^2+
\frac{\kappa}{6}\| \nabla \partial_t \di\|_{L^2(\Omega)}^2
    +C \left( 1+ \| \phi\|_{H^2(\Omega)}^2\right) \| \partial_t \di\|_{L^2(\Omega)}^2.
\end{split}
\end{equation}
Therefore, combining the above estimates, we end up with
\begin{equation}
\label{DI:2}
\begin{split}
    &\dt\left( \frac12 \| \nabla \uu\|_{L^2(\Omega)}^2 + \frac{1}{2}\norm{\nabla \mu}_{L^2(\Omega)}^2 
    +\frac12 \| \partial_t \di\|_{L^2(\Omega)}^2
    +\prt{\uu\cdot \nabla \phi, \mu}\right) 
    + \frac{\nu_\ast}{2} \| \Delta \uu\|_{L^2(\Omega)}^2
    + \frac{\varepsilon}{2}\norm{\nabla \partial_t \phi}_{L^2(\Omega)}^2 
    \\
    &\quad 
    + \frac{\kappa}{2} \|\nabla \partial_t \di \|_{L^2(\Omega)}^2
    + \frac{1}{\varepsilon} \int_\Omega F''(\phi) \abs{\partial_t \phi}^2 \, \d x
    + 2\alpha \int_\Omega \left| \di \cdot \partial_t \di \right|^2 \, \d x 
    +\alpha \int_\Omega |\di |^2 |\partial_t \di|^2 \, \d x
    \\
    &\quad 
   + \beta \int_\Omega \left| \nabla \partial_t \phi \cdot \di\right|^2 \, \d x
   + \beta \int_{\Omega} |\partial_t \di \cdot \nabla \phi|^2 \, \d x
   \\
   &
   \leq  \varpi \|\partial_t \uu \|_{L^2(\Omega)}^2 +C \left( 1+ \|\phi \|_{H^2(\Omega)}^2 + \norm{\nabla \uu}_{L^2(\Omega)}^2 + \| \Delta \di \|_{L^2(\Omega)}^2 \right)
     \left( \norm{\nabla \uu}_{L^2(\Omega)}^2
    +  \norm{\nabla \mu}_{L^2(\Omega)}^2 +\| \partial_t \di\|_{L^2(\Omega)}^2 \right)
\\
&\quad 
 +C
\left(\norm{\vect{h}}_{L^2(\Omega)}^2+
\norm{\nabla \uu}_{L^2(\Omega)}^2 +
\norm{\Delta \di}_{L^2(\Omega)}^2 \right)
\left(1+ \norm{\Delta \di}_{L^2(\Omega)}^2\right).
\end{split}
\end{equation}
Next, repeating the computations in \eqref{E:g_z:3}-\eqref{E:g_z:8}, there exists $C>0$ (which depends on the assumption \eqref{Key-ass}) such that
\begin{equation*}
\norm{\Delta \phi}_{L^2(\Omega)}^2 
+ \norm{\Delta \di}_{L^2(\Omega)}^2
\leq C \| \nabla \mu\|_{L^2(\Omega)} 
+ C \| \hh\|_{L^2(\Omega)}^2
+C(\alpha, \varepsilon, \kappa, \Theta_0, \mathcal{D}_\infty, K_0, |\Omega|).
\end{equation*}
By using \eqref{W:regularity}, \eqref{GN} and \eqref{LADY}, we notice that
\begin{equation*}
\| (\uu \cdot \nabla) \di\|_{L^2(\Omega)}^2
\leq C \| \nabla \uu\|_{L^2(\Omega)} \| \di \|_{H^2(\Omega)} \leq C \| \nabla \uu\|_{L^2(\Omega)} \left( 1+ \| \Delta \di\|_{L^2(\Omega)}\right).
\end{equation*}
Since $\hh=\partial_t \di + (\uu \cdot \nabla) \di$, 
we conclude that
\begin{equation}
\label{est-H^2}
\norm{\Delta \phi}_{L^2(\Omega)}^2 
+ \norm{\Delta \di}_{L^2(\Omega)}^2
\leq C\left( 1 + \| \nabla \mu\|_{L^2(\Omega)}^2 + \| \partial_t \di \|_{L^2(\Omega)}^2 
+ \| \nabla \uu\|_{L^2(\Omega)}^2 
\right).
\end{equation}
Besides, testing \eqref{NSE} by $\partial_t \uu$ and exploiting once again \eqref{W:regularity}, \eqref{GN} and \eqref{LADY}, we find
\begin{equation*}
    \label{p_t_u-2}
    \begin{split}
    \norm{\partial_t \uu}_{L^2(\Omega)}^2 
    &\leq C \| \nabla \uu\|_{L^2(\Omega)}
    \| \Delta \uu\|_{L^2(\Omega)}^\frac12
    \norm{\partial_t \uu}_{L^2(\Omega)}
    + C \| \Delta \uu\|_{L^2(\Omega)} \norm{\partial_t \uu}_{L^2(\Omega)}
    \\
    & \quad 
    + C \| \phi\|_{H^2(\Omega)}^\frac12 
    \| \nabla \uu\|_{L^2(\Omega)}^\frac12
    \| \Delta \uu\|_{L^2(\Omega)}^\frac12
    \norm{\partial_t \uu}_{L^2(\Omega)}
   + C \| \nabla \mu\|_{L^2(\Omega)} 
   \norm{\partial_t \uu}_{L^2(\Omega)}
   \\
   &\quad 
   + C \| \di\|_{H^2(\Omega)}^\frac12 
   \| \hh \|_{L^2(\Omega)}^\frac12 
   \| \hh \|_{H^1(\Omega)}^\frac12 
   \norm{\partial_t \uu}_{L^2(\Omega)}, 
    \end{split}
\end{equation*}
which entails that
\begin{equation}
    \label{p_t_u-3}
    \begin{split}
    \norm{\partial_t \uu}_{L^2(\Omega)}^2 
    &\leq C \| \nabla \uu\|_{L^2(\Omega)}^2
    \| \Delta \uu\|_{L^2(\Omega)}
    + C \| \Delta \uu\|_{L^2(\Omega)}^2
    + C \| \phi\|_{H^2(\Omega)} 
    \| \nabla \uu\|_{L^2(\Omega)}
    \| \Delta \uu\|_{L^2(\Omega)}
    \\
    & \quad 
   + C \| \nabla \mu\|_{L^2(\Omega)}^2
   + C \| \di\|_{H^2(\Omega)}
   \| \hh \|_{L^2(\Omega)} 
   \| \hh \|_{H^1(\Omega)}.
    \end{split}
\end{equation}
Arguing as in \eqref{est-dhu}, we observe that
\begin{equation}
\label{h-H1}
  \begin{split}
      \| \hh \|_{H^1(\Omega)}
      &\leq \| \hh\|_{L^2(\Omega)} +  \| \nabla \partial_t \di \|_{L^2(\Omega)} + \| \nabla  \left( (\uu\cdot \nabla) \di \right) \|_{L^2(\Omega)}
      \\
      & \leq \| \hh\|_{L^2(\Omega)} +  \| \nabla \partial_t \di \|_{L^2(\Omega)} 
      + C \| \nabla \uu\|_{L^2(\Omega)}^\frac12 
      \| \Delta \uu\|_{L^2(\Omega)}^\frac12 \| \di\|_{H^2(\Omega)}^\frac12 
      + \| \Delta \uu\|_{L^2(\Omega)}^\frac12 \| \di \|_{H^2(\Omega)}.
  \end{split}  
\end{equation}
Thus, combining \eqref{p_t_u-3} with \eqref{est-H^2} and \eqref{h-H1}, we infer that
\begin{equation*}
    \label{p_t_u-4}
    \begin{split}
    \norm{\partial_t \uu}_{L^2(\Omega)}^2 
    &\leq 
    C \| \Delta \uu\|_{L^2(\Omega)}^2
    + C \left(\| \nabla \uu\|_{L^2(\Omega)}^2
    +  \| \phi\|_{H^2(\Omega)}^2 \right)
    \| \nabla \uu\|_{L^2(\Omega)}^2
    + C \| \nabla \mu\|_{L^2(\Omega)}^2
    \\
    & \quad 
   + C \| \di\|_{H^2(\Omega)}
   \| \hh \|_{L^2(\Omega)}^2
   + C \| \di\|_{H^2(\Omega)}
   \| \hh \|_{L^2(\Omega)} 
   \| \nabla \partial_t \di\|_{L^2(\Omega)}
   \\
   &\quad 
   + C \| \di\|_{H^2(\Omega)}^2
   \| \hh \|_{L^2(\Omega)}^\frac43 
   \| \nabla \uu\|_{L^2(\Omega)}^\frac23
   +C \| \di\|_{H^2(\Omega)}^\frac83
   \| \hh \|_{L^2(\Omega)}^\frac43
   \\
   &\leq 
   C \| \Delta \uu\|_{L^2(\Omega)}^2
    + C \left(\| \nabla \uu\|_{L^2(\Omega)}^2
    +  \| \phi\|_{H^2(\Omega)}^2 \right)
    \| \nabla \uu\|_{L^2(\Omega)}^2
    + C \| \nabla \mu\|_{L^2(\Omega)}^2
    \\
   &\quad 
   + C \| \di\|_{H^2(\Omega)}
   \| \hh \|_{L^2(\Omega)} 
   \| \nabla \partial_t \di\|_{L^2(\Omega)}
   + C \left( \| \di \|_{H^2(\Omega)}^2+ \| \hh\|_{L^2(\Omega)}^2+ \| \nabla \uu\|_{L^2(\Omega)}^2 \right)\left(1+ \| \Delta \di\|_{L^2(\Omega)}^2 \right)
   \\
   &\leq 
    C \| \Delta \uu\|_{L^2(\Omega)}^2
    + C \| \di\|_{H^2(\Omega)}
   \| \hh \|_{L^2(\Omega)} 
   \| \nabla \partial_t \di\|_{L^2(\Omega)}
    \\
   &\quad 
   + C \left( 1+ \| \phi\|_{H^2(\Omega)}^2 + \| \di \|_{H^2(\Omega)}^2+ \| \hh\|_{L^2(\Omega)}^2+ \| \nabla \uu\|_{L^2(\Omega)}^2 \right)
   \left(1+ 
   \| \nabla \mu\|_{L^2(\Omega)}^2 + \| \partial_t \di \|_{L^2(\Omega)}^2 
+ \| \nabla \uu\|_{L^2(\Omega)}^2 \right).
   \end{split}
\end{equation*}
Multiplying the above inequality by $\frac{\nu_\ast}{4C}$ and adding the resulting relation to \eqref{DI:2}, and setting $\varpi= \frac{\nu_\ast}{8C}$, we find
\begin{equation}
\label{DI:3}
\begin{split}
    &\dt\left( \frac12 \| \nabla \uu\|_{L^2(\Omega)}^2 + \frac{1}{2}\norm{\nabla \mu}_{L^2(\Omega)}^2 
    +\frac12 \| \partial_t \di\|_{L^2(\Omega)}^2
    +\prt{\uu\cdot \nabla \phi, \mu}\right) 
    + \frac{\nu_\ast}{4} \| \Delta \uu\|_{L^2(\Omega)}^2
    + \frac{\nu_\ast}{8C} \| \partial_t \uu\|_{L^2(\Omega)}^2
    \\
    &\quad 
    + \frac{\varepsilon}{2}\norm{\nabla \partial_t \phi}_{L^2(\Omega)}^2 
    + \frac{\kappa}{2} \|\nabla \partial_t \di \|_{L^2(\Omega)}^2
    + \frac{1}{\varepsilon} \int_\Omega F''(\phi) \abs{\partial_t \phi}^2 \, \d x
    + 2\alpha \int_\Omega \left| \di \cdot \partial_t \di \right|^2 \, \d x 
    \\
    &\quad 
   +\alpha \int_\Omega |\di |^2 |\partial_t \di|^2 \, \d x
   + \beta \int_\Omega \left| \nabla \partial_t \phi \cdot \di\right|^2 \, \d x
   + \beta \int_{\Omega} |\partial_t \di \cdot \nabla \phi|^2 \, \d x
   \\
   &
   \leq  
   C \left( 1+ \|\phi \|_{H^2(\Omega)}^2 + \norm{\nabla \uu}_{L^2(\Omega)}^2 + \| \Delta \di \|_{L^2(\Omega)}^2  + \norm{\vect{h}}_{L^2(\Omega)}^2 \right)
     \left( 1+ \norm{\nabla \uu}_{L^2(\Omega)}^2
    +  \norm{\nabla \mu}_{L^2(\Omega)}^2 +\| \partial_t \di\|_{L^2(\Omega)}^2 \right)
\\
&\quad 
 +C
\left(\norm{\vect{h}}_{L^2(\Omega)}^2+
\norm{\nabla \uu}_{L^2(\Omega)}^2 +
\norm{\Delta \di}_{L^2(\Omega)}^2 \right)
\left(1+ \norm{\Delta \di}_{L^2(\Omega)}^2\right)
+ C \| \di\|_{H^2(\Omega)}
   \| \hh \|_{L^2(\Omega)} 
   \| \nabla \partial_t \di\|_{L^2(\Omega)}.
\end{split}
\end{equation}
Finally, on account of \eqref{est-H^2}, we deduce that
\begin{equation}
\label{DI:4}
\begin{split}
    &\dt\left( \frac12 \| \nabla \uu\|_{L^2(\Omega)}^2 + \frac{1}{2}\norm{\nabla \mu}_{L^2(\Omega)}^2 
    +\frac12 \| \partial_t \di\|_{L^2(\Omega)}^2
    +\prt{\uu\cdot \nabla \phi, \mu}\right) 
    + \frac{\nu_\ast}{4} \| \Delta \uu\|_{L^2(\Omega)}^2
    + \frac{\nu_\ast}{8C} \| \partial_t \uu\|_{L^2(\Omega)}^2
    \\
    &\quad 
    + \frac{\varepsilon}{2}\norm{\nabla \partial_t \phi}_{L^2(\Omega)}^2 
    + \frac{\kappa}{4} \|\nabla \partial_t \di \|_{L^2(\Omega)}^2
    + \frac{1}{\varepsilon} \int_\Omega F''(\phi) \abs{\partial_t \phi}^2 \, \d x
    + 2\alpha \int_\Omega \left| \di \cdot \partial_t \di \right|^2 \, \d x 
    \\
    &\quad 
   +\alpha \int_\Omega |\di |^2 |\partial_t \di|^2 \, \d x
   + \beta \int_\Omega \left| \nabla \partial_t \phi \cdot \di\right|^2 \, \d x
   + \beta \int_{\Omega} |\partial_t \di \cdot \nabla \phi|^2 \, \d x
   \\
   &
   \leq  
   C \left( 1+ \|\phi \|_{H^2(\Omega)}^2 + \norm{\nabla \uu}_{L^2(\Omega)}^2 + \| \Delta \di \|_{L^2(\Omega)}^2  + \norm{\vect{h}}_{L^2(\Omega)}^2 \right)
     \left( 1+ \norm{\nabla \uu}_{L^2(\Omega)}^2
    +  \norm{\nabla \mu}_{L^2(\Omega)}^2 +\| \partial_t \di\|_{L^2(\Omega)}^2 \right).
\end{split}
\end{equation}

Let us now set
$$
\mathcal{K}= \frac12 \| \nabla \uu\|_{L^2(\Omega)}^2 + \frac{1}{2}\norm{\nabla \mu}_{L^2(\Omega)}^2 
    +\frac12 \| \partial_t \di\|_{L^2(\Omega)}^2
    +\prt{\uu\cdot \nabla \phi, \mu}.
$$
Arguing as in \cite[proof of Theorem 4.1]{giorgini2019uniqueness}, there exist two positive constants $C_1$ and $C_2$ such that
\begin{equation}
    \label{bel-abo}
   \frac14 \| \nabla \uu\|_{L^2(\Omega)}^2 + \frac{1}{4}\norm{\nabla \mu}_{L^2(\Omega)}^2 
    +\frac12 \| \partial_t \di\|_{L^2(\Omega)}^2 -C_1 \leq \mathcal{K} \leq C_2 \left( \frac12 \| \nabla \uu\|_{L^2(\Omega)}^2 + \frac{1}{2}\norm{\nabla \mu}_{L^2(\Omega)}^2 
    +\frac12 \| \partial_t \di\|_{L^2(\Omega)}^2 \right).
\end{equation}
Then, \eqref{DI:4} can be recast into the differential inequality
\begin{equation}
\label{DI:5}
\begin{split}
\dt \left( C_1 +\mathcal{K} \right) 
+ \mathcal{G}
\leq  \mathcal{R} + \mathcal{R}\left( C_1 +\mathcal{K} \right),
\end{split}
\end{equation}
where 
$$
    \mathcal{G}=\frac{\nu_\ast}{4} \| \Delta \uu\|_{L^2(\Omega)}^2
    + \frac{\nu_\ast}{8C} \| \partial_t \uu\|_{L^2(\Omega)}^2 
    + \frac{\varepsilon}{2}\norm{\nabla \partial_t \phi}_{L^2(\Omega)}^2 
    + \frac{\kappa}{4} \|\nabla \partial_t \di \|_{L^2(\Omega)}^2
$$
and
$$
\mathcal{R}=C \left( 1+ \|\phi \|_{H^2(\Omega)}^2 + \norm{\nabla \uu}_{L^2(\Omega)}^2 + \| \Delta \di \|_{L^2(\Omega)}^2  + \norm{\vect{h}}_{L^2(\Omega)}^2 \right).
$$
An application of the Gronwall lemmas (in both {\it standard} and {\it uniform} versions as in \cite{temam2012infinite}) entail
\begin{equation}
\label{DI:6}
\begin{split}
    &\sup_{t\in [0,\infty)} \left( \| \nabla \uu(t)\|_{L^2(\Omega)}^2 + \norm{\nabla \mu(t)}_{L^2(\Omega)}^2 
    +\| \partial_t \di(t)\|_{L^2(\Omega)}^2 \right)
    + \sup_{t\geq 0} \int_t^{t+1} \mathcal{G}(s)\, \d s
    \\
    &\leq C \left( C_1+\mathcal{K}(0)+
    \sup_{t\geq 0}\int_{t}^{t+1} \mathcal{K}(s)\, \d s
    +\sup_{t\geq 0}\int_{t}^{t+1} \mathcal{R}(s)\, \d s \right) \mathrm{exp} \left(
    \sup_{t\geq 0}\int_{t}^{t+1}  \mathcal{R}(s)\, \d s \right).
\end{split}
\end{equation}
Since $\mathcal{K}$ and $\mathcal{R}$ belong to $ L_{\rm{uloc}}^1(0,\infty)$  from \eqref{W:regularity}, in order to deduce global estimates from \eqref{DI:6} we are left to show that $\mathcal{K}(0)$ is finite. To this end, in light of \eqref{bel-abo}, we only need to take care of $\|\partial_t \di (0)\|_{L^2(\Omega)}$
and $\| \nabla \mu(0)\|_{L^2(\Omega)}$.
By using \eqref{var_d}, we first observe that 
$$
\begin{aligned}
    \|\partial_t \di (0)\|_{L^2(\Omega)}
        &\leq C \norm{\nabla \uu(0)}_{L^2(\Omega)} \norm{\di(0)}_{H^2(\Omega)} + C \norm{\Delta \di(0)}_{L^2(\Omega)} + C\left(\norm{\di(0)}_{L^\infty(\Omega)}\right)
        + \beta \norm{\di(0)}_{L^\infty(\Omega)} \norm{\nabla \phi(0)}^2_{L^4(\Omega)}\\
        &\leq C\left(\norm{\nabla \uu_0}_{L^2(\Omega)}, \norm{\di_0}_{H^2(\Omega)}, \norm{\nabla \phi_0}_{L^4(\Omega)}\right).
\end{aligned}
$$
Furthermore, 
we also have
\begin{equation*}
    \begin{split}
    \| \nabla \mu(0)\|_{L^2(\Omega)}
    &\leq C \norm{ \nabla \left(-\varepsilon \Delta \phi(0)+ \frac{1}{\varepsilon}\Psi'(\phi(0)) -\diver\left(\beta (\nabla \phi(0) \cdot \vect{d}(0)) \vect{d}(0)\right)\right)}_{L^2(\Omega)}+
    C \| \di(0)\|_{L^\infty(\Omega)} \|\nabla \di(0)\|_{L^2(\Omega)}
    \\
    &\leq 
    C \left\| \nabla \left(-\varepsilon \Delta \phi_0+ \frac{1}{\varepsilon}\Psi'(\phi_0) -\diver\left(\beta (\nabla \phi_0\cdot \vect{d}_0) \vect{d}_0\right)\right) \right\|_{L^2(\Omega)}+
    C\norm{\di_0}_{H^2(\Omega)}.
     \end{split}
\end{equation*}
Thus, $\mathcal{K}(0) <\infty$, and thereby we infer from \eqref{DI:6} that 
\begin{equation}
\label{eq:reg_1}
\begin{split}
&\|\uu\|_{L^\infty(0,\infty;\dot{H}^1(\Omega))}
    + \| \nabla \mu\|_{L^\infty(0,\infty;L^2(\Omega))}
    + \| \partial_t \di\|_{L^\infty(0,\infty; L^2(\Omega))} \leq C
    \\
    &\|\uu\|_{L^2_{\rm uloc}([0,\infty);\dot{H}^2(\Omega))} + \norm{\partial_t \uu}_{L^2_{\rm uloc}([0,\infty);\dot{L}^2(\Omega))} + \norm{\partial_t \phi}_{L^2_{\rm uloc}([0,\infty);H^1(\Omega))} + \norm{\partial_t \di}_{L^2_{\rm uloc}([0,\infty);H^1(\Omega))}\leq C,
    \end{split}
\end{equation}
where the constants $C$ only depends on the norm of the initial conditions and the parameters of the system. Here, we have used the conservation mass for the term involving $\partial_t \phi$. Moreover, from \eqref{est-H^2} and \eqref{eq:reg_1}, we deduce that
\begin{equation}
\label{eq:reg_2}
\norm{\phi}_{L^\infty(0,\infty;H^2(\Omega))} + \norm{\di}_{L^\infty(0,\infty;H^2(\Omega))} \leq C.
\end{equation}
Exploiting a similar argument as the one used in \eqref{media-mu-d} to deduce \eqref{est:vv:5}, 
 we derive that $\norm{\mu}_{H^1(\Omega)} \leq C(1 + \norm{\nabla \mu}_{L^2(\Omega)})$, which in turn gives that
\begin{equation}
\label{eq:reg_3}
     \mu \in L^\infty(0,\infty; H^1(\Omega)).
\end{equation}
On account of the estimate \eqref{F_k}, which can be easily repeated in the case $\gamma=0$, and making use of \eqref{eq:reg_3}, we find 
\begin{equation}
\label{eq:reg_4}
     F'(\phi) \in L^\infty(0,\infty; L^p(\Omega)), \quad \forall \, p \in [2,\infty).
\end{equation}
Lastly, owing to \eqref{eq:reg_1} and \eqref{eq:reg_2}, we notice that $\nabla((\uu \cdot \nabla )\di) \in L^2_{\rm uloc}([0,\infty);L^2(\Omega))$, which implies by comparison that $\nabla \hh \in L^2_{\rm uloc}([0,\infty);L^2(\Omega))$. Then, observing that $(\di \cdot \nabla \phi) \nabla \phi \in L_{\rm uloc}^2([0,\infty);L^p(\Omega))$ for any $p \in [2,\infty)$, we conclude that $\di \in L^2_{\rm uloc}([0,\infty); W^{2,p}(\Omega))$ for any $p \in [2,\infty)$.
\end{proof}

\section{Numerical results}
\label{sec:numerics}

In this section, we present numerical experiments on the dynamics of the liquid-crystalline emulsion. In order to focus on the interfacial and polarization dynamics together with the anchoring mechanism, we perform numerical simulations on the system \eqref{var_phi}-\eqref{var_d} with $\uu\equiv \mathbf{0}$. More precisely, we implement a finite element approximation of \eqref{var_phi}-\eqref{var_d} through the open source \texttt{FEniCS} \cite{logg2012automated}.
The weak-discrete problem reads as follows
\begin{align}
\label{eq:CH1_num}
    &\int_\Omega \frac{\phi_{n+1} -\phi_n}{\Delta t}q \, \d x + \int_\Omega \nabla \left[\prt{1-\vartheta} \mu_{n} + \vartheta\mu_{n+1}\right]\cdot \nabla q\, \d x = 0,\\
    \label{eq:CH2_num}
    &\int_\Omega \mu_{n+1} v\, \d x -\int_\Omega f'(\phi_{n+1}) v \, \d x -\int_\Omega \varepsilon \nabla \phi_{n+1} \cdot \nabla v\, \d x + \int_\Omega \frac{\alpha}{2} \abs{\vect{d}_{n+1}}^2 v\, \d x  -\beta \int_\Omega \abs{\vect{d}_{n+1}}^2 \nabla \phi_{n+1} \cdot \nabla v\, \d x = 0,\\
    \label{eq:d1_num}
    &\int_\Omega \frac{\vect{d}_{n+1} -\vect{d}_n}{\Delta t} \cdot \vect{\xi}\, \d x +\int_\Omega \left[\prt{1-\vartheta} \vect{h}_{n} + \vartheta\vect{h}_{n+1}\right]\cdot \vect{\xi}\, \d x = 0,\\
    \label{eq:d2_num}
    &\int_\Omega \vect{h}_{n+1}\cdot \vect{l}\, \d x - \kappa \int_{\Omega} \nabla \vect{d}_{n+1} \cdot \nabla \vect{l}\, \d x -\alpha \int_\Omega \abs{\vect{d}_{n+1}}^2 \vect{d}_{n+1}\cdot \vect{l}\, \d x + \alpha \int_\Omega (\phi_{n+1}-\phi_{\rm cr})\vect{d}_{n+1}\cdot\vect{l}\, \d x\\
    \notag
    &\qquad\qquad \qquad -\beta\int_\Omega \prt{\nabla \phi_{n+1}\cdot \vect{d}_{n+1}}\prt{\nabla \phi_{n+1} \cdot \vect{l}}\, \d x=0,
\end{align}
for any $ q \in V, v\in V, \vect{\xi} \in \vect{W},\vect{l} \in \vect{W}$, where $V$ is the space of continuous polynomials of degree $1$, while the space $\vect{W}$ is the set of continuous, piece-wise quadratic polynomials. 
In the discretized problem, the notation $f_n$ stands for the function evaluated at time $t_n$. To treat the time evolution, we employ the so-called {\em $\vartheta$-method} \cite{quarteroni2009numerical}. Thus, time derivatives are approximated by a difference quotient and the right-hand side is replaced with a linear combination of the value at time $t_n$ and the value at time $t_{n+1}$. Such linear combination depends on a real parameter $\vartheta$, which is chosen $\vartheta = 0.5$ (i.e. the {\em Crank-Nicolson} time discretization).
For simplicity, we employ the fourth-order polynomial approximation $f(\phi) = \frac{\phi^2 (1- \phi)^2}{\varepsilon}$ of the Flory-Huggins potential \eqref{pot-FH} with minima translated in $0$ and $1$. The chosen value for $\phi_{\rm cr}$ is $\frac12$ to be symmetric with respect to the minima. 

We consider $\Omega=[-1,1]^2 \subset \mathbb{R}^2$ partitioned into a triangulation (mesh) of non overlapping triangular cells. The mesh denoted by $Q_\delta$ contains $6561$ nodes and $12800$ cells. Such number of elements is necessary to capture phenomena at the interface which is assumed to be of size $\varepsilon \ll 1$.
We study the benchmark example concerning the shape deformation of a drop of polymer immersed in a liquid crystal (see also \cite{shen2014decoupled, sui2021second, zhao2016numerical}). To this end, we consider the following initial data
\begin{equation}
    \label{eq:val_d_0_phi_0_2test}
    \begin{aligned}
    &\phi_0 = 1- \left[\frac{1}{2}-\frac{1}{2}\tanh\left(\frac{\sqrt{x^2 + y^2} - 0.5}{0.01}\right)\right],  &&\di_0 = (0,0.95),&&&\mu_0 = 0,&&&&\vect{h}_0 = (0,0),
    \end{aligned}
\end{equation}
where we are obliged not to set $\left(\di_0\right)_y = 1$ for avoiding stagnation or divergence of the numerical code. We show in Figure \ref{fig:t_zero} the chosen initial configuration.
\begin{figure}[htbp]
		\centering
\includegraphics[width=0.85\textwidth]{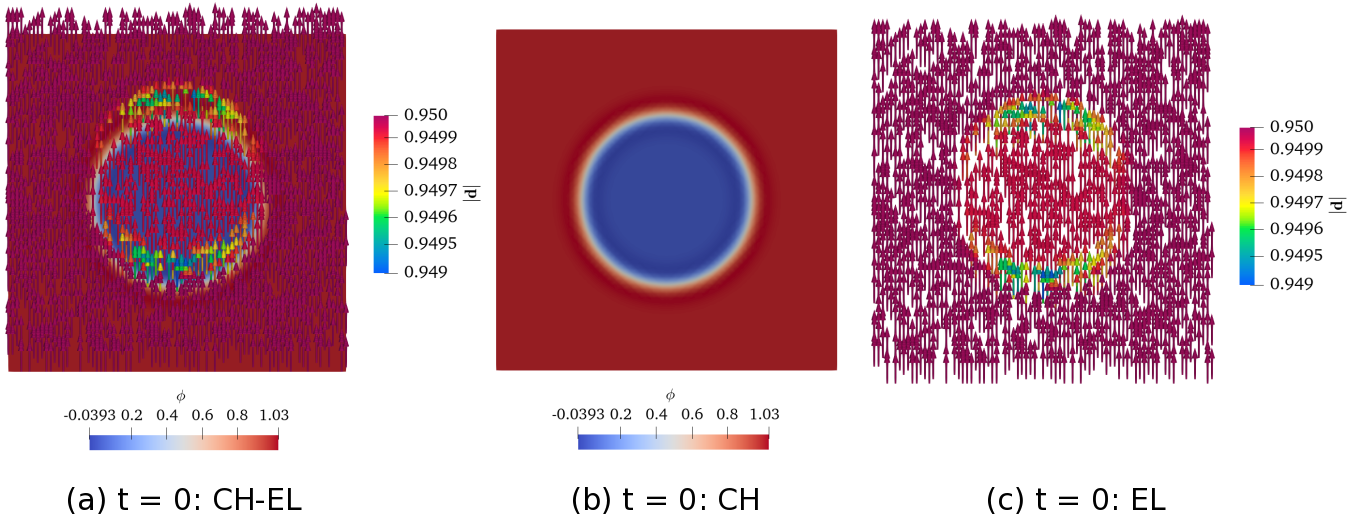}
		\caption{(a) Initial configuration for the entire discrete system CH-EL, (b) the initial concentration variable $\phi$, and (c) the initial distribution of the polarization $\di$.}
	\label{fig:t_zero}
\end{figure}

We choose the parameters of the problem to be as
\begin{align}
\label{eq:constants}
    &\varepsilon = 10^{-1},&&\alpha = 10,&&\beta = 1,&&\kappa =  10^{-1}, && \phi_{\rm cr}=\frac12.
\end{align}

\begin{figure}[h!]
\centering
\includegraphics[width=0.9\textwidth]{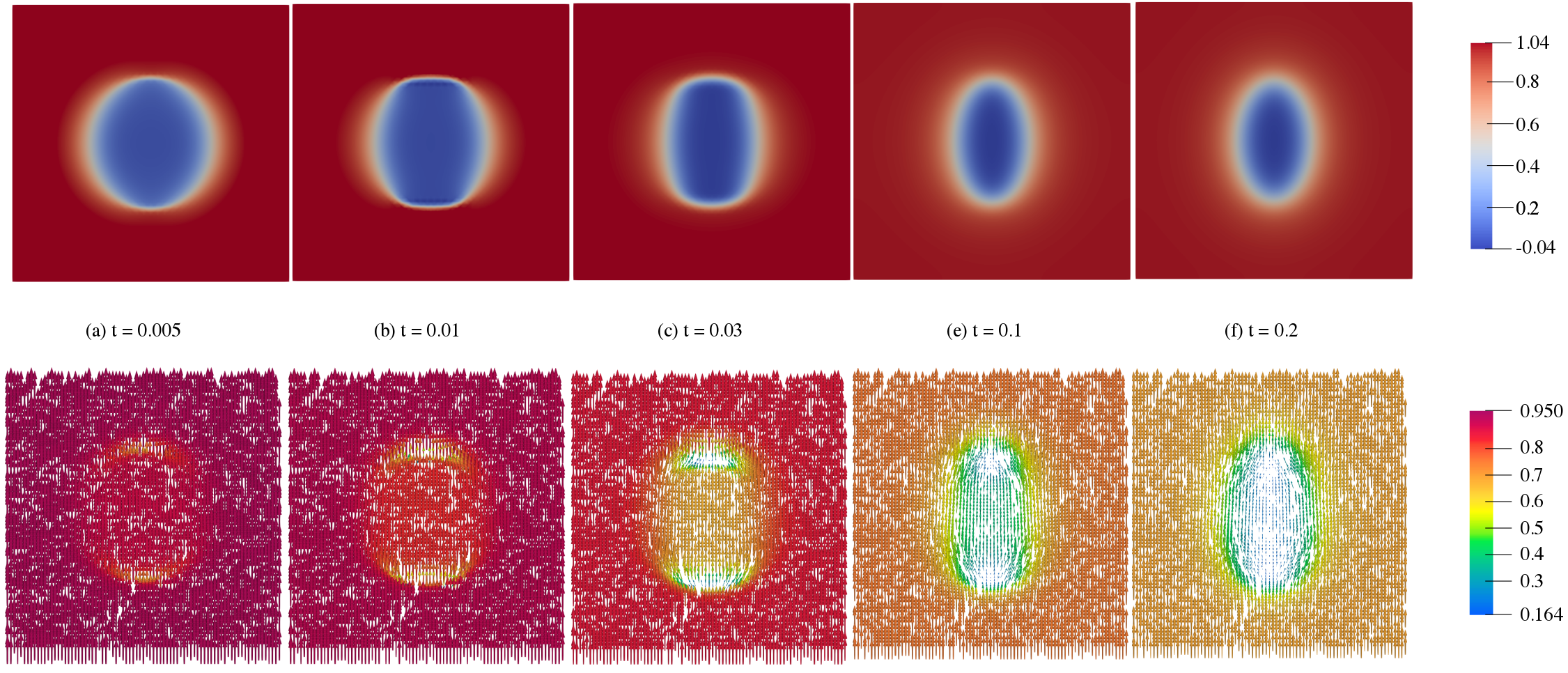}
	\caption{Different time steps of the solution of the discrete problem \eqref{eq:CH1_num}-\eqref{eq:d2_num}, having chosen as initial value for $\phi_0$ and $\di_0$ as in \eqref{eq:val_d_0_phi_0_2test} and for the constants given by \eqref{eq:constants}. First line contains the discrete concentration $\phi_n$, whereas the second line shows the magnitude of discrete polarization $|\di_n|$. 
 }
\label{fig:evolution}
\end{figure}

The discrete weak system \eqref{eq:CH1_num} - \eqref{eq:d2_num} results in a finite-dimensional quadratically nonlinear problem of total dimension 147606 degrees of freedom, of which 8281 for $\phi$ and 65522 unknowns for $\vect{d}$. We solve such a nonlinear problem with the \texttt{SNES} solver, part of the \texttt{PETSc} library \cite{petsc-user-ref}. 
To set a suitable stopping criterion, we consider the energy rate
$\mathcal{E} := \frac{E_{n+1} - E_n}{E_n},
$
where $n$ is the last time step, $n+1$ is the current time step and $E_i$ stands for the discretized total free energy 
$$
E_{i}=\int_{Q_{\delta}} \left(\frac{\varepsilon}{2} \abs{\nabla \phi_i}^2 + \frac{1}{\varepsilon}f(\phi_i) + \frac{\kappa}{2}\abs{\nabla \vect{d}_i}^2
+ \frac{\alpha}{4} \abs{\vect{d}_i}^4
-\frac{\alpha}{2}\left(\phi_i-\phi_{\rm cr}\right) \abs{\vect{d}_i}^2 + \frac{\beta}{2} \abs{\nabla \phi_i \cdot \vect{d}_i}^2\right) \, \d x.
$$
If $\mathcal{E} < 10^{-6}$ (numerical tolerance), we conclude that one of the equilibrium states has been reached.

We show in Figure \ref{fig:evolution} the dynamics at different time steps. We notice that the shape of the initial circular drop changes over time: it becomes more elongated in the $y$-direction, the one where the vector field $\di$ is initially maximum (cf. the first line of Figure \ref{fig:evolution}, which shows the discrete concentration $\phi_n$ at different time step). This fact is in accordance with the simulations obtained in \cite{sui2021second, zhao2016numerical}. In addition, 
a second interesting aspect concerns the direction of the vector field $\di$. As expected from the modeling viewpoint, the module of the polarization $|\di|$ becomes smaller inside the drop ($\phi \approx 0)$ as shown in the second line of Figure \ref{fig:evolution}. Instead, in the region where $\phi \approx 1$, the polarization $\di$ remains vertically oriented. In addition, the orientation of $\di$ is aligned in parallel direction to the diffuse interface of the drop. 
The latter behavior is more evident close to the equilibrium states as shown in Figure \ref{fig:equilibrium}. Here, we rescale the value of $\abs{\di}$ to better visualize the maximum value of the polarization in the region where $\phi \approx 1$. We observe that the value $\abs{\di} \approx 0.7$ is in accordance with the heuristics that $(\phi,\di)$ converges towards states with values equal to the critical points of 
$$
\widetilde{g}(s, w) := \frac{1}{\varepsilon}\left(s^2(1-s)^2\right) - \frac{\alpha}{2} (s - \phi_{\rm cr}) w^2 + \frac{\alpha}{4} w^4.
$$
In fact, the global minimum of $\widetilde{g}$ in the region $s>\frac12$ is approximately $w = 0.7$ for the choice of the physical parameters \eqref{eq:constants}.
\begin{figure}[h!]
\centering
\includegraphics[width=0.77\textwidth]{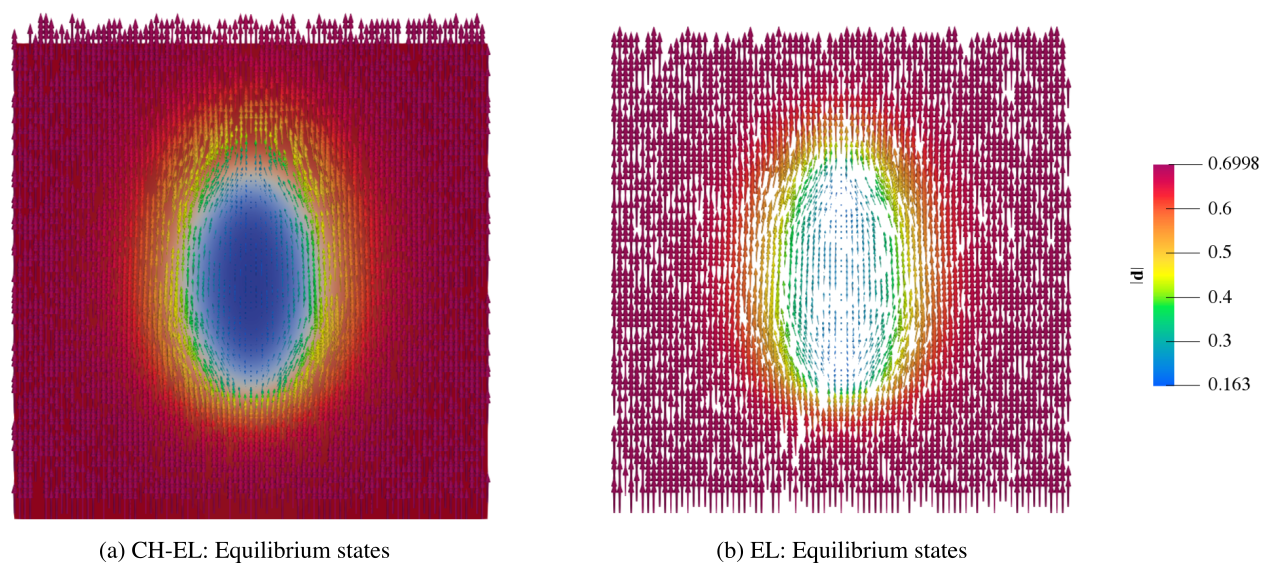}
	\caption{(a) Plot of the equilibrium states for both the discrete CH equations and discrete EL equations. (b) Visualization of the orientation of the discrete liquid crystal vector $\di$ at the equilibrium states.
 }
\label{fig:equilibrium}
\end{figure}

\begin{appendices}
\section{Weak solution to a viscous generalized Cahn-Hilliard equation}
\label{app:galerkin}
In this appendix, we show the existence of weak solutions to a viscous generalized Cahn-Hilliard equation driven by an incompressible vector field $\vv$ and an external polarization $\di$. The system reads as
\begin{equation}
\label{eq:VCH_1_app}
\left\{
\begin{aligned}
&\partial_t\phi + \vect{v} \cdot \nabla \phi = \Delta \mu,
\\
&\mu = \delta \,\partial_t \phi-\diver\left((\varepsilon + \gamma \abs{\nabla \phi}^2)\nabla \phi\right)+ \frac{1}{\varepsilon}\Psi'(\phi) -\frac{\alpha}{2}\abs{\dtilde}^2 -\diver\left(\beta (\nabla \phi \cdot \dtilde)\dtilde\right),
\end{aligned}
\right.
\quad \text{in } \Omega \times (0,T),
\end{equation}
where $T$ is a final time, which is equipped with periodic boundary conditions.  Without loss of generality, we assume that $\Psi(s) \geq 0$ for $s\in [-1,1]$, which is simply obtained by adding a suitable constant to the Flory-Huggins potential in \eqref{pot-FH}.

\begin{theorem}
\label{vp-CH:ws}
Let $\Omega= \mathbb{T}^n$, with $n=2,3$, and $T>0$. 
Assume that $\phi_0 \in W^{1,4}(\Omega)$ with $\| \phi_0\|_{L^\infty(\Omega)}\leq 1$ and $|\overline{\phi_0}|<1$. Given $\vv \in C([0,T];\vect{V}_m)$ and $\dtilde \in C([0,T];L^2(\Omega))\cap L^\infty(\Omega \times (0,T)) \cap L^2(0,T;W^{1,4}(\Omega))$ 
such that
\begin{equation}
    \label{App:key-ass}
    \| \dtilde\|_{L^\infty(\Omega \times (0,T))}\leq D_\infty, 
\end{equation}
there exists a weak solution $(\phi,\mu )$ to \eqref{eq:VCH_1_app} on $[0,T]$
such that 
\begin{align}
\label{A:WS:R1}
    & \phi \in L^\infty(0,T;W^{1,4}(\Omega))\cap L^2(0,T;H^2(\Omega)),\\
    \label{A:WS:R2}
    &\phi \in L^\infty(\Omega \times (0,T)) \, \text{with }
    |\phi(x,t)|<1 \, \text{a.e. in } \Omega \times (0,T),\\
    \label{A:WS:R3}
    &\partial_t \phi \in L^2(0,T; L^2(\Omega)), \quad 
    |\nabla \phi|^2 \nabla \phi \in L^2(0,T; H^1(\Omega)),
    \quad F'(\phi) \in L^2(0,T;L^2(\Omega)),\\
    \label{A:WS:R4}
    &\mu \in L^2(0,T;H^1(\Omega)),
\end{align}
which satisfies 
\begin{align}
\label{A:WS:1}
    &( \partial_t \phi, \xi )
    - (\phi \vv, \nabla \xi)
    + (\nabla \mu, \nabla \xi)
    =0, \quad \forall \, \xi \in H^1(\Omega), \text{ a.e. in }(0,T),
    \\
\label{A:WS:2}    
    &\mu=
    \delta \partial_t \phi
    - \diver \left( (\varepsilon+ \gamma |\nabla \phi|^2) \nabla \phi \right) 
    + \frac{1}{\varepsilon} \Psi'(\phi)
    -\frac{\alpha}{2}  |\dtilde|^2 
    - \beta \diver \left( (\nabla \phi \cdot \dtilde) \dtilde \right), \quad \text{a.e. in } \Omega \times (0,T),
\end{align}
as well as $\phi(\cdot,0)=\phi_0$ in $\Omega$. 
Furthermore, any weak solution fulfills 
\begin{equation}
\label{A:WS_satisfies}
\begin{split}
&\esssup_{t\in [0,T]}
\int_\Omega \abs{\nabla \phi(t)}^4 + \abs{\nabla \phi(t)}^2 + \Psi(\phi(t))\, \d x
+  \int_0^T \| \nabla \mu(s)\|_{L^2(\Omega)}^2 \, \d s
\\
& \quad 
+\int_0^T \| \partial_t \phi(s)\|_{L^2(\Omega)}^2 \, \d s
+  \int_0^T \| \Delta \phi(s)\|_{L^2(\Omega)}^2 \, \d s
+ \int_0^T \int_\Omega
\left| \nabla\left(\abs{\nabla \phi}^2\right) \right|^2 \, \d x
\\
&\leq
 \overline{C} \left(\int_\Omega \abs{\nabla \phi_0}^4 +  \abs{\nabla \phi_0}^2 + \Psi(\phi_0)\, \d x
+  \int_0^T \|\nabla \dtilde(s) \|_{L^4(\Omega)}^2 \, \d s + T\right)
\\
& \quad 
\times \mathrm{exp}\left(\overline{C}T + \overline{C}\int_0^T \| \nabla \dtilde(s)\|_{L^4(\Omega)}^2 + \|\vv(s) \|_{L^2(\Omega)}^2 \, \d s \right),
\end{split}
\end{equation}
where the positive constant $\overline{C}$ depends on $\alpha, \beta, \gamma, \delta, \varepsilon, \Theta_0, n, m, \overline{\phi_0}, D_\infty$, but is independent of $\dtilde$ and $\vv$.
\end{theorem}

\begin{proof}
Let us first assume that the initial datum $\phi_0 \in H^2(\Omega)$
such that $\| \phi_0\|_{L^\infty(\Omega)}\leq 1$ and $|\overline{\phi_0}|<1$.
For any $\lambda \in (0,1)$, we introduce a family of regular potential $\{\Psi_\lambda\}_{\lambda \in (0,1)}$ that approximates the Flory-Huggins potential $\Psi$ by defining
\begin{equation}
\label{Psi-L}
\Psi_\lambda(s) = F_\lambda(s) - {\color{black}\frac{\Theta_0}{2} s^2}, \quad \forall\, s \in \R  
\end{equation}
where 
\begin{equation}
\label{F-L}
F_\lambda(s) 
=
\begin{cases}
\sum_{j = 0}^2 \frac{1}{j!} F^{(j)}(1- \lambda) [s- (1- \lambda)]^j, &\forall\, s \geq 1- \lambda,\\[5pt]
F(s),&\forall\, s \in [-1+\lambda, 1-\lambda],\\[5pt]
\sum_{j = 0}^ 2 \frac{1}{j!} F^{(j)}(-1+ \lambda) {\color{black} [s - (-1 + \lambda)]^j}, & \forall\, s \leq -1 + \lambda.
\end{cases}
\end{equation}
Then, we approximate the problem \eqref{eq:VCH_1_app} by means of
\begin{empheq}[left=\empheqlbrace]{align}
\label{eq:VCH_1_app_eps}
&\partial_t\phi_\lambda + \vect{v} \cdot \nabla \phi_\lambda = \Delta \mu_\lambda,\\
\label{eq:VCH_2_app_eps}
    &\mu_\lambda = \delta \,\partial_t \phi_\lambda-\diver\left((\varepsilon + \gamma \abs{\nabla \phi_\lambda}^2)\nabla \phi_\lambda\right)+ \frac{1}{\varepsilon}\Psi_\lambda'(\phi_\lambda) -\frac{\alpha}{2}\abs{\dtilde}^2 -\diver\left(\beta (\nabla \phi_\lambda \cdot \dtilde)\dtilde\right).
\end{empheq}
We now introduce the Galerkin approximation. For $k \in \mathbb{N}$, we look for a pair of functions
\begin{equation}
    \label{eq:phi_gal_VpCH}
    \phi^k_\lambda(x,t) = \sum_{i = 1}^k a_i^k(t) \xi_i(x), 
    \qquad \mu_\lambda^k=\sum_{i = 1}^k b_i^k(t) \xi_i(x) \qquad \text{in } \Omega \times (0,T), 
\end{equation}
where $\xi_i$ are the eigenfunctions of the Laplacian operator $-\Delta$ and $V_k= \text{span} \lbrace \xi_1, \dots, \xi_k \rbrace$, which solves 
\begin{equation}
\label{eq:gal_VCH_1}
\prt{\partial_t \phi^k_\lambda, \xi_i} 
+ \prt{\vect{v} \cdot \nabla \phi^k_\lambda, \xi_i} = \prt{\mu^k_\lambda, \Delta \xi_i},
\end{equation}
and
\begin{equation}
\label{eq:gal_VCH_2}
\begin{split}
\prt{\mu^k_{\lambda}, \xi_i}&= 
\delta \left( \partial_t \phi^k_\lambda, \xi_i \right) 
+ \left( \left(\varepsilon + \gamma \abs{\nabla \phi^k_\lambda}^2\right)\nabla \phi^k_\lambda, \nabla \xi_i \right) 
+ \frac{1}{\varepsilon}\left(\Psi'_\lambda(\phi^k_\lambda), \xi_i \right) \\
&\quad
-\frac{\alpha}{2} \left( \abs{\dtilde}^2, \xi_i\right)
+\beta \left( (\nabla \phi^k_\lambda \cdot \dtilde)\dtilde, \nabla \xi_i \right),
\end{split}
\end{equation}
for all $i =1, \dots, k$. The system \eqref{eq:gal_VCH_1}-\eqref{eq:gal_VCH_2}
can be rewritten as the following system of ODEs 
\begin{equation}
\label{eq:CH_galerkin}
\left\{
\begin{aligned}
&\prt{1+\delta\mathbb{S}^k}  \frac{{\rm d} }{{\rm d}t} \mathbb{A}^k(t) + \mathbb{L}^k(t) \mathbb{A}^k (t) + \mathbb{G}^k(\mathbb{A}^k(t)) = \mathbb{F}^k(t),
  \\
  &\mathbb{A}^k(0)=\left( (\phi_0,\xi_1), \dots, (\phi_0,\xi_k) \right),
\end{aligned}\right.
\end{equation}
where
$$
\begin{aligned}
\mathbb{A}^k(t) &:= \left[a_1(t),\dots, a_k(t)\right]^{T},
\qquad \qquad 
\mathbb{S}^k := {\rm diag}\left[\lambda_1, \dots,\lambda_k\right],
\\
\mathbb{F}^k(t)&:=\left[ \prt{-\frac{\alpha}{2} \abs{\dtilde(t)}^2,\lambda_1 \xi_1},\dots, \prt{-\frac{\alpha}{2} \abs{\dtilde(t)}^2,\lambda_k\xi_k}\right],\\
\prt{\mathbb{L}^k(t)}_{j,k} &:= \int_\Omega \prt{\vect{v}(t)\cdot \nabla \xi_k}\xi_j\, \d x -\beta \int_\Omega \lambda_j \prt{\nabla \xi_k \cdot \dtilde(t)} \,\prt{\dtilde(t) \cdot \nabla \xi_j}\, \d x + \varepsilon\, \lambda_j^2 \delta_{j,k} 
-\frac{\Theta_0}{\varepsilon}\lambda_j \delta_{j,k}  , \\
\mathbb{G}^k(\mathbb{A}^k(t)) &:= \left[-\frac{1}{\varepsilon}\int_\Omega\lambda_1 \xi_1\, F_\lambda'\prt{\sum_{i=1}^k a_i(t) \xi_i}\, \d x,\dots, -\frac{1}{\varepsilon}\int_\Omega\lambda_k \xi_k\, F_\lambda'\prt{\sum_{i=1}^k a_i(t) \xi_i}\, \d x\right]^{T}\\
&\quad+ \left[\gamma \sum_{i=1}^k a_i(t) \int_\Omega\lambda_1  \abs{\sum_{l=1}^k a_l(t) \nabla \xi_l}^2 \nabla \xi_i\cdot \nabla \xi_1\, \d x,\dots, \gamma \sum_{i=1}^k a_i(t) \int_\Omega\lambda_k  \abs{\sum_{l=1}^k a_l(t) \nabla \xi_l}^2 \nabla \xi_i\cdot \nabla \xi_k\, \d x\right]^{T},
\end{aligned}
$$
where $\delta_{j,k}$ is the Kronecker delta.
Since $\vv \in C([0,T];\vect{V}_m)$ and $\vect{\dtilde}\in C([0,T];L^2(\Omega))$, it follows that $\mathbb{L}^k \in C([0,T]; \mathbb{R}^{k\times k})$ and $\mathbb{F}^k \in C([0,T];\mathbb{R}^k)$. By using \cite[Lemma 3.1]{giorgini2018cahn}, we get that the nonlinear part $\mathbb{G}^k(\mathbb{A}^k(t))$ is locally Lipschitz in $\mathbb{A}^k$. Then, we deduce that \eqref{eq:CH_galerkin} has a unique solution $\mathbb{A}^k(\cdot) \in C^1([0,T_0);\mathbb{R}^k)$ for some $T_0$ (possibly smaller than $T$).

Next, in order to show that $\mathbb{A}^k(t)$ is well defined on $[0,T]$ and to pass to the limit as $k\to \infty$, we carry out several uniform estimates in the approximation parameters ($k$ and $\lambda$). 
We start by observing that \eqref{eq:gal_VCH_1} with $i=1$ ($\xi_1=\frac{1}{(2\pi)^n}$) gives
\begin{equation}
\label{gal:cm}
\dt \int_\Omega \phi_\lambda^k(t) \, \d x 
=0 \quad \Rightarrow \quad
\overline{\phi_\lambda^k(t)}=
\frac{1}{|\Omega|}(\phi_0,1)
=\overline{\phi_0},\quad \forall \, t \in [0,T_0).
\end{equation}
We multiply \eqref{eq:gal_VCH_1} 
by $\mu_\lambda^k$. Observing that $\partial_t \phi_\lambda^k(t)\in V_m$, we find 
\begin{equation}
    \label{eq:stima_mu_n}
    \begin{aligned}
    &\dt \left[\int_\Omega \frac{\gamma}{4}\abs{\nabla \phi_\lambda^k}^4 + \frac{\varepsilon}{2}\abs{\nabla \phi_\lambda^k}^2 + \frac{1}{\varepsilon}  \Psi_\lambda(\phi_\lambda^k) \, \d x\right]
    +\norm{\nabla \mu_\lambda^k}_{L^2(\Omega)}^2
    + \delta \norm{\partial_t \phi_\lambda^k}_{L^2(\Omega)}^2\\
    &\quad 
    =\int_\Omega (\phi_\lambda^k -\overline{\phi_\lambda^k}) \vv \cdot \nabla \mu_\lambda^k\, \d x 
    +\frac{\alpha}{2}\int_\Omega \abs{\dtilde}^2 \partial_t \phi_\lambda^k\, \d x 
    + \beta \int_\Omega \diver\prt{(\nabla \phi_\lambda^k \cdot \dtilde)\dtilde}\partial_t \phi_\lambda^k \, \d x.
    \end{aligned}
\end{equation}
By using the Sobolev embedding, \eqref{H2:equiv}, \eqref{Rev-SI} and \eqref{App:key-ass}, we deduce that
\begin{equation}
\begin{split}
&
\left| \int_\Omega (\phi_\lambda^k-\overline{\phi_\lambda^k}) \vv  \cdot \nabla \mu_\lambda^k\, \d x 
    +\frac{\alpha}{2}\int_\Omega \abs{\dtilde}^2 \partial_t \phi_\lambda^k\, \d x 
    + \beta \int_\Omega \diver\prt{(\nabla \phi_\lambda^k \cdot \dtilde)\dtilde}\partial_t \phi_\lambda^k \, \d x \right|
 \\
& \quad
\leq \| \phi_\lambda^k-\overline{\phi_\lambda^k}\|_{L^6(\Omega)}
\| \vv\|_{L^3(\Omega)}
\| \nabla \mu_\lambda^k\|_{L^2(\Omega)}
+ \frac{\alpha}{2} \left\| \left|\dtilde\right|^2\right\|_{L^2(\Omega)}
\| \partial_t \phi_\lambda^k\|_{L^2(\Omega)}
\\
& \qquad
+ n^2\beta 
\| D^2 \phi_\lambda^k\|_{L^2(\Omega)}
\| \dtilde\|_{L^\infty(\Omega)}^2 
\|\partial_t \phi_\lambda^k \|_{L^2(\Omega)}
+ 2 n \beta \| \nabla \phi_\lambda^k \|_{L^4(\Omega)} \| \nabla \dtilde\|_{L^4(\Omega)}
\| \dtilde\|_{L^\infty(\Omega)} 
\| \partial_t \phi_\lambda^k\|_{L^2(\Omega)}
 \\
&\quad 
\leq 
\frac14 \| \nabla \mu_\lambda^k\|_{L^2(\Omega)}^2 
+ C(m)\| \vv\|_{L^2(\Omega)}^2\| \nabla \phi^k_\lambda\|_{L^2(\Omega)}^2
 + \frac{\delta}{4} \norm{\partial_t\phi_\lambda^k}_{L^2(\Omega)}^2 
+ C(\alpha, \delta, n) D_\infty^4
 \\
&\qquad
+C(\beta, \delta, n) D_\infty^4 \left( |\overline{\phi_0}|^2 + \|\Delta \phi_\lambda^k\|_{L^2(\Omega)}^2\right)  
+ C(n)D_\infty^2 \| \nabla \dtilde\|_{L^4(\Omega)}^2 
\norm{\nabla \phi_\lambda^k}_{L^4(\Omega)}^2
 \\
&\quad \leq \frac{1}{4}\norm{ \nabla \mu_\lambda^k}_{L^2(\Omega)}^2 
+ \frac{\delta}{4} \norm{\partial_t\phi_\lambda^k}_{L^2(\Omega)}^2
+ C_1 \left( \| \nabla \dtilde\|_{L^4(\Omega)}^2
+\norm{\vv}_{L^2(\Omega)}^2 \right)
\left( \norm{\nabla \phi_\lambda^k}_{L^4(\Omega)}^4
+ \norm{\nabla \phi_\lambda^k}_{L^2(\Omega)}^2 \right) \\
&\qquad
+ C_2 \norm{\Delta \phi_\lambda^k}_{L^2(\Omega)}^2 
+ C_3 \| \nabla \dtilde\|_{L^4(\Omega)}^2 + C_4,
\end{split}
\end{equation}
where the constants $C_1, \dots, C_4$ depend on $\alpha, \beta, \gamma, \delta, \varepsilon, n, m, \overline{\phi_0}, D_\infty$, but are independent of $k$ and $\lambda$. Hence, we arrive at
\begin{equation}
\label{A:est0}
    \begin{split}
    &\dt \left[\int_\Omega \frac{\gamma}{4}\abs{\nabla \phi_\lambda^k}^4 + \frac{\varepsilon}{2}\abs{\nabla \phi_\lambda^k}^2 + \frac{1}{\varepsilon} \Psi_\lambda(\phi_\lambda^k)\, \d x\right]
    +\frac{3}{4}\norm{\nabla \mu_\lambda^k}_{L^2(\Omega)}^2 + \frac{3\delta}{4} \norm{\partial_t \phi_\lambda^k}_{L^2(\Omega)}^2
    \\
    &\quad \leq 
   C_1 \left( \| \nabla \dtilde\|_{L^4(\Omega)}^2 
   + \| \vv\|_{L^2(\Omega)}^2 \right)\left[\int_\Omega \frac{\gamma}{4}\abs{\nabla \phi_\lambda^k}^4 + \frac{\varepsilon}{2}\abs{\nabla \phi_\lambda^k}^2 + \frac{1}{\varepsilon} \Psi_\lambda(\phi_\lambda^k)\, \d x\right]
   \\
   &\qquad 
    + C_2 \| \Delta \phi_\lambda^k\|_{L^2(\Omega)}^2 
    + C_3 \| \nabla \dtilde\|_{L^4(\Omega)}^2 + C_4.
    \end{split}
\end{equation}
Multiplying \eqref{eq:gal_VCH_2} by
$\lambda_i a_i$ and summing over $i$, we have
\begin{equation}
\label{A:est00}
\begin{split}
&
\frac{\delta}{2}\dt \| \nabla \phi^k_\lambda\|_{L^2(\Omega)}^2
+\varepsilon\norm{\Delta \phi_\lambda^k}_{L^2(\Omega)}^2 + \gamma \int_\Omega \diver \left( |\nabla \phi_\lambda^k |^2 \nabla \phi_\lambda^k \right) \Delta \phi_\lambda^k \, \d x + \underbrace{\frac{1}{\varepsilon} \int_\Omega F_{\lambda}''(\phi_\lambda^k)
\abs{\nabla \phi_\lambda^k}^2\, \d x}_{\geq 0}
\notag \\
&
= \int_\Omega \nabla \mu^k_\lambda \cdot \nabla \phi^k_\lambda \, \d x
+ \frac{\Theta_0}{\varepsilon} 
\int_{\Omega} |\nabla \phi_\lambda^k|^2 \, \d x
+\frac{\alpha}{2}\int_\Omega \Delta \phi_\lambda^k \abs{\dtilde}^2\, \d x -
\beta \int_\Omega \diver\prt{(\nabla \phi_\lambda^k \cdot \dtilde)\dtilde} \Delta \phi_\lambda^k \, \d x.
\end{split}
\end{equation}
Integrating by parts, the second term on the left-hand side is rewritten as
\begin{equation*}
\begin{split}
    \int_\Omega \diver\left(\abs{\nabla \phi_\lambda^k}^2 \nabla \phi_\lambda^k\right)\Delta \phi_\lambda^k\, \d x  &= \int_{\Omega} \partial_i\left(\abs{\nabla \phi_\lambda^k}^2 \partial_i\phi_\lambda^k\right)\partial_{jj}\phi_\lambda^k \, \d x
   \\ 
   &= \int_{\Omega} \partial_j(\abs{\nabla \phi_\lambda^k}^2 \partial_i \phi_\lambda^k) \partial_j\partial_i \phi_\lambda^k \, \d x
    \\
   &=\int_{\Omega} \left(\partial_j \partial_i\phi_\lambda^k \partial_j \partial_i \phi_\lambda^k \abs{\nabla \phi_\lambda^k}^2 + 2 \partial_j\partial_i \phi_\lambda^k \partial_i \phi_\lambda^k \partial_j\partial_l\phi_\lambda^k \partial_l\phi_\lambda^k \right)\, \d x 
    \\
   &=  \int_{\Omega} \abs{D^2 \phi_\lambda^k}^2 \abs{\nabla \phi_\lambda^k}^2 + \frac{1}{4} \left| \nabla\left(\abs{\nabla \phi_\lambda^k}^2\right) \right|^2 \, \d x.
   \label{Two:der:est}
   \end{split}
\end{equation*}
Also, we observe that
\begin{equation*}
\begin{split}
     \int_\Omega \diver\prt{(\nabla \phi_\lambda^k \cdot \dtilde)\dtilde} \Delta \phi_\lambda^k \, \d x 
    &=  \int_\Omega \partial_l \left( (\nabla \phi_\lambda^k \cdot \dtilde)\dtilde_j \right) \partial_l \partial_j \phi_\lambda^k \, \d x
    \\
    &= \int_\Omega \partial_l (\nabla \phi_\lambda^k \cdot \dtilde) \dtilde_j \partial_l \partial_j \phi_\lambda^k \, \d x
    + \int_\Omega (\nabla \phi_\lambda^k \cdot \dtilde) \partial_l \dtilde_j \partial_l \partial_j \phi_\lambda^k \, \d x
    \\
    &= \underbrace{\int_\Omega |\nabla(\nabla \phi_\lambda^k \cdot \dtilde)|^2 \, \d x}_{\geq 0}
    +  \int_\Omega (\nabla \phi_\lambda^k \cdot \dtilde) \nabla \dtilde : D^2 \phi^k_\lambda \, \d x
    -  \int_\Omega \nabla(\nabla \phi_\lambda^k \cdot \dtilde) \cdot (\nabla \dtilde^T \nabla \phi^k_\lambda) \, \d x.
    \end{split}
\end{equation*}
Then, we reach
\begin{equation}
\label{A:est1}
\begin{split}
& \frac{\delta}{2}\dt \| \nabla \phi^k_\lambda\|_{L^2(\Omega)}^2
+\varepsilon\norm{\Delta \phi_\lambda^k}_{L^2(\Omega)}^2 + \gamma 
\int_{\Omega} \abs{D^2 \phi_\lambda^k}^2 \abs{\nabla \phi_\lambda^k}^2 \, \d x
+\frac{\gamma}{4} \int_\Omega \left| \nabla\left(\abs{\nabla \phi_\lambda^k}^2\right) \right|^2 \, \d x
\\
&\leq 
\frac{\varpi}{4} \| \nabla \mu^k_\lambda\|_{L^2(\Omega)}^2 
+ \left( \frac{\Theta_0}{\varepsilon}+\frac{1}{\varpi}\right)\|\nabla \phi_\lambda^k\|_{L^2(\Omega)}^2 
+\frac{\alpha}{2} \left| \int_\Omega \Delta \phi_\lambda^k \abs{\dtilde}^2\, \d x \right|  + \left|\int_\Omega (\nabla \phi_\lambda^k \cdot \dtilde) \nabla \dtilde : D^2 \phi^k_\lambda \, \d x \right|
\\
&\quad + \left| \int_\Omega \nabla(\nabla \phi_\lambda^k \cdot \dtilde) \cdot (\nabla \dtilde \nabla \phi^k_\lambda) \, \d x\right|,
\end{split}
\end{equation}
where the positive parameter $\varpi$ will be determined later on.
It is easily seen that 
\begin{align}
\label{A:est2-2}
\left| \frac{\alpha}{2}\int_\Omega \Delta \phi_\lambda^k \abs{\dtilde}^2\, \d x \right| 
&\leq 
\frac{\varepsilon}{2}
\| \Delta \phi_\lambda^k\|_{L^2(\Omega)}^2 
+ \frac{\alpha^2 \abs{\Omega} C(n)}{8 \varepsilon} \| \dtilde\|_{L^\infty(\Omega)}^4
\leq \frac{\varepsilon}{2}
\| \Delta \phi_\lambda^k\|_{L^2(\Omega)}^2 
+ \frac{\alpha^2 \abs{\Omega} C(n)}{8 \varepsilon} D_\infty^4,
\end{align}
\begin{equation}
\label{A:est2-3}
    \begin{split}
       \left| \beta \int_\Omega (\nabla \phi_\lambda^k \cdot \dtilde) \nabla \dtilde : D^2 \phi^k_\lambda \, \d x \right|
       &\leq \frac{\gamma}{4} \int_\Omega \abs{D^2 \phi_\lambda^k}^2 \abs{\nabla \phi_\lambda^k}^2 \, \d x
       + \frac{\beta^2 C(n)}{\gamma} D_\infty^2 \| \nabla \dtilde\|_{L^2(\Omega)}^2 
    \end{split}
\end{equation}
and
\begin{equation}
\label{A:est2-4}
    \begin{split}
         \left| \beta  \int_\Omega \nabla(\nabla \phi_\lambda^k \cdot \dtilde) \cdot (\nabla \dtilde^T \nabla \phi^k_\lambda) \, \d x \right| 
         &=
         \beta \left| \int_\Omega \partial_l \partial_i \phi^k_\lambda \dtilde_i \partial_l \dtilde_j \partial_j \phi^k_\lambda 
         + \partial_i \phi^k_\lambda \partial_l \dtilde_i \partial_l \dtilde_j \partial_j \phi^k_\lambda \, \d x
         \right|
         \\
         &\leq \frac{\gamma}{4} \int_\Omega \abs{D^2 \phi_\lambda^k}^2 \abs{\nabla \phi_\lambda^k}^2 \, \d x
       + \frac{\beta^2 C(n)}{\gamma} D_\infty^2 \| \nabla \dtilde\|_{L^2(\Omega)}^2 
       \\
       &\quad + \beta \| \nabla \dtilde\|_{L^4(\Omega)}^2 \| \nabla \phi^k_\lambda\|_{L^4(\Omega)}^2.
    \end{split}
\end{equation}
Thus, we infer that 
\begin{equation}
\label{A:est2}
\begin{split}
& 
 \frac{\delta}{2}\dt \| \nabla \phi^k_\lambda\|_{L^2(\Omega)}^2
+\frac{\varepsilon}{2} \norm{\Delta \phi_\lambda^k}_{L^2(\Omega)}^2 + \frac{\gamma}{2} 
\int_{\Omega} \abs{D^2 \phi_\lambda^k}^2 \abs{\nabla \phi_\lambda^k}^2 \, \d x
+\frac{\gamma}{4} \int_\Omega \left| \nabla\left(\abs{\nabla \phi_\lambda^k}^2\right) \right|^2 \, \d x
\\
&\leq 
\frac{\varpi}{4} \| \nabla \mu^k_\lambda\|_{L^2(\Omega)}^2
+ \left( \frac{\Theta_0}{\varepsilon}+\frac{1}{\varpi}\right) \|\nabla \phi_\lambda^k\|_{L^2(\Omega)}^2 
+C_5 \| \nabla \dtilde\|_{L^4(\Omega)}^2 
+C_6 \| \nabla \dtilde\|_{L^4(\Omega)}^2 \| \nabla \phi^k_\lambda\|_{L^4(\Omega)}^4
+C_7, 
\end{split}
\end{equation}
where the positive constants $C_5, \dots, C_7$ depend only on $\beta$, $\gamma$, $n$ and $D_\infty$. 
Now, multiplying \eqref{A:est2} by $\frac{4C_2}{\varepsilon}$, adding it to \eqref{A:est0} and setting $\varpi= \frac{\varepsilon}{4 C_2}$, we end up with the differential inequality 
\begin{equation}
\label{A:est3}
    \begin{split}
    &\dt \left[\int_\Omega \frac{\gamma}{4}\abs{\nabla \phi_\lambda^k}^4 + \frac{\varepsilon}{2}\abs{\nabla \phi_\lambda^k}^2 + \frac{1}{\varepsilon} \Psi_\lambda(\phi_\lambda^k)\, \d x 
    + \frac{2\delta C_2}{\varepsilon} \| \nabla \phi^k_\lambda\|_{L^2(\Omega)}^2
     \right]
     \\
    &
    +\frac{1}{2}\norm{\nabla \mu_\lambda^k}_{L^2(\Omega)}^2 + \frac{\delta}{2} \norm{\partial_t \phi_\lambda^k}_{L^2(\Omega)}^2
    + C_2 \|\Delta \phi^k_\lambda\|_{L^2(\Omega)}^2 
    + \frac{\gamma C_2}{\varepsilon} \int_\Omega \left| \nabla\left(\abs{\nabla \phi_\lambda^k}^2\right) \right|^2 \, \d x
    \\
   &\quad  \leq 
   \overline{C} \left( 1+ \| \nabla \dtilde\|_{L^4(\Omega)}^2 
   + \| \vv\|_{L^2(\Omega)}^2 \right)\left[\int_\Omega \frac{\gamma}{4}\abs{\nabla \phi_\lambda^k}^4 + \frac{\varepsilon}{2}\abs{\nabla \phi_\lambda^k}^2 + \frac{1}{\varepsilon} \Psi_\lambda(\phi_\lambda^k)\, \d x\right]
    + \overline{C} \| \nabla \dtilde\|_{L^4(\Omega)}^2 + \overline{C},
    \end{split}
\end{equation}
where the positive constant $\overline{C}$ depends on 
$\alpha, \beta, \gamma, \delta, \varepsilon, \Theta_0, n, m, \overline{\phi_0}, D_\infty$, but are independent of $k$ and $\lambda$. By the Gronwall lemma, 
we immediately conclude that $\max_{t\in [0,T_0)} |\mathbb{A}^k(t)|$ is bounded, thereby the continuation principle for systems of ODEs yields that $\mathbb{A}$ is defined on $[0,T]$. Moreover, we deduce the estimate
\begin{equation}
\label{A:est7}
\begin{split}
&\max_{t\in [0,T]}
\int_\Omega \frac{\gamma}{4}\abs{\nabla \phi_\lambda^k(t)}^4 + \left( \frac{\varepsilon}{2} + \frac{2\delta C_2}{\varepsilon} \right) \abs{\nabla \phi_\lambda^k(t)}^2 + \frac{1}{\varepsilon} \Psi_\lambda(\phi_\lambda^k(t))\, \d x
+ \frac12 \int_0^T \| \nabla \mu_\lambda^k(s)\|_{L^2(\Omega)}^2 \, \d s
\\
& \quad 
+\frac{\delta}{2}\int_0^T \| \partial_t \phi_\lambda^k(s)\|_{L^2(\Omega)}^2 \, \d s
+ C_2 \int_0^T \| \Delta \phi^k_\lambda(s)\|_{L^2(\Omega)}^2 \, \d s
+ \frac{\gamma C_2}{\varepsilon} \int_0^T \int_\Omega
\left| \nabla\left(\abs{\nabla \phi_\lambda^k}^2\right) \right|^2 \, \d x
\\
&\leq
2\left(\int_\Omega \frac{\gamma}{4}\abs{\nabla \phi_\lambda^k(0)}^4 + \left( \frac{\varepsilon}{2} + \frac{2\delta C_2}{\varepsilon} \right) \abs{\nabla \phi_\lambda^k(0)}^2 + \frac{1}{\varepsilon} \Psi_\lambda(\phi_\lambda^k(0))\, \d x
+ \overline{C} \int_0^T \|\nabla \dtilde(s) \|_{L^4(\Omega)}^2 \, \d s + \overline{C}T\right)
\\
& \quad 
\times \mathrm{exp}\left(\overline{C}T + \overline{C}\int_0^T \| \nabla \dtilde(s)\|_{L^4(\Omega)}^2 + \|\vv(s) \|_{L^2(\Omega)}^2 \, \d s \right).
\end{split}
\end{equation}
We observe that 
$$
\| \nabla \phi_\lambda^k(0)\|_{L^4(\Omega)}
\leq C_S \| \phi_\lambda^k(0)\|_{H^2(\Omega)}
=C_S \| \mathbb{P}^k \phi_0\|_{H^2(\Omega)}
\leq 
C_S \| \phi_0\|_{H^2(\Omega)},
$$
where $C_S$ is the constant of the Sobolev embedding. Besides, from the definitions \eqref{Psi-L} and \eqref{F-L}, there exists two positive constants $\widetilde{C}^1_\lambda$ and $\widetilde{C}^2_\lambda$ such that $\Psi_\lambda(s) \leq \widetilde{C}^1_\lambda + \widetilde{C}^2_\lambda s^2$ for all $s \in \mathbb{R}$. Hence, it follows from \eqref{A:est7} that
\begin{equation}
\label{A:est8}
\begin{split}
&\max_{t\in [0,T]}
\int_\Omega \frac{\gamma}{4}\abs{\nabla \phi_\lambda^k(t)}^4 + \left( \frac{\varepsilon}{2} + \frac{2\delta C_2}{\varepsilon} \right) \abs{\nabla \phi_\lambda^k(t)}^2 + \frac{1}{\varepsilon} \Psi_\lambda(\phi_\lambda^k(t))\, \d x
+ \frac12 \int_0^T \| \nabla \mu_\lambda^k(s)\|_{L^2(\Omega)}^2 \, \d s
\\
& \quad 
+\frac{\delta}{2}\int_0^T \| \partial_t \phi_\lambda^k(s)\|_{L^2(\Omega)}^2 \, \d s
+ C_2 \int_0^T \| \Delta \phi^k_\lambda(s)\|_{L^2(\Omega)}^2 \, \d s
+ \frac{\gamma C_2}{\varepsilon} \int_0^T \int_\Omega
\left| \nabla\left(\abs{\nabla \phi_\lambda^k}^2\right) \right|^2 \, \d x
\\
& \leq
2\left( \frac{\gamma C_S^4}{4} \| \phi_0\|_{H^2(\Omega)}^4 + \left( \frac{\varepsilon}{2} + \frac{2\delta C_2}{\varepsilon} \right) \| \phi_0\|_{H^1(\Omega)}^2 +\frac{1}{\varepsilon} \left( \widetilde{C}^1_\lambda |\Omega| + \widetilde{C}^2_\lambda \| \phi_0\|_{L^2(\Omega)}^2 \right)\right.
\\
&\quad \left.
+ \overline{C} \int_0^T \|\nabla \dtilde(s) \|_{L^4(\Omega)}^2 \, \d s + \overline{C}T 
\right)
 \mathrm{exp} \left(\overline{C}T + \overline{C}\int_0^T \| \nabla \dtilde(s)\|_{L^4(\Omega)}^2 + \|\vv(s) \|_{L^2(\Omega)}^2 \, \d s\right)
=:R_1(\lambda).
\end{split}
\end{equation}
Furthermore, exploiting again \eqref{Psi-L} and \eqref{F-L}, there exists two positive constants $\widetilde{C}^3_\lambda$ and $\widetilde{C}^4_\lambda$ such that $|\Psi'_\lambda(s)| \leq \widetilde{C}^3_\lambda + \widetilde{C}^4_\lambda |s|$ for all $s \in \mathbb{R}$. 
In light of \eqref{H1:equiv}, \eqref{gal:cm} and \eqref{A:est8}, we obtain 
\begin{equation}
\label{A:est9}
\| \Psi'_\lambda(\phi^k_\lambda) \|_{L^2(0,T;L^2(\Omega))}
\leq \widetilde{C}^3_\lambda T^\frac12 |\Omega|^\frac12 + \sqrt{2} \widetilde{C}^4_\lambda  |\overline{\phi_0}| |\Omega|^\frac12 T^\frac12+ 2 \widetilde{C}^4_\lambda \frac{R_1^\frac12}{\varepsilon^\frac12}T^\frac12:=R_2(\lambda).
\end{equation}
Next, taking $i=1$ in \eqref{eq:gal_VCH_2} and observing that $\overline{\partial_t \phi^k_\lambda}= 0$, we notice that
$
\overline{\mu}= \frac{1}{\varepsilon} \overline{\Psi'_\lambda(\phi^k_\lambda)}- \frac{\alpha}{2|\Omega|}
\|\dtilde\|_{L^2(\Omega)}^2.
$
Using \eqref{A:est8} and \eqref{A:est9}, we get
\begin{equation}
\label{A:est10}
\| \overline{\mu}\|_{L^2(0,T)}
\leq 
\frac{1}{\varepsilon|\Omega|^\frac12} R_2(\lambda)
+
\frac{\alpha}{2|\Omega|} \| \dtilde\|_{L^4(0,T;L^2(\Omega))}^2=:R_3(\lambda).
\end{equation}
Then, \eqref{A:est8} and \eqref{A:est10} entail that
\begin{equation}
    \label{A:est11}
    \| \mu\|_{L^2(0,T;H^1(\Omega))} 
    \leq 2 R_1(\lambda)^\frac12 + \sqrt{2}|\Omega|^\frac12 R_3(\lambda).
\end{equation}
As an immediate consequence of \eqref{A:est8}, \eqref{A:est9} and \eqref{A:est11}, we find that (up to subsequence)
\begin{equation}
\label{A:lim1}
\begin{split}
\begin{aligned}
&\phi_\lambda^k \rightharpoonup \phi_\lambda \ &&\text{weak-star in } L^\infty(0,T;W^{1,4}(\Omega)), 
\quad 
&&&\phi_\lambda^k \rightharpoonup \phi_\lambda \ &&&&\text{weakly in } 
L^2(0,T;H^2(\Omega)),\\
&\partial_t \phi_\lambda^k \rightharpoonup \partial_t \phi_\lambda \  &&\text{weakly in } 
L^2(0,T;L^2(\Omega)),
\quad
&&&|\nabla \phi_\lambda^k|^2 \rightharpoonup g_\lambda \  &&&&\text{weakly in } 
L^2(0,T;H^1(\Omega)),\\
&\mu_\lambda^k  \rightharpoonup \mu_\lambda \   &&\text{weakly in } L^2(0,T;H^1(\Omega)).
\end{aligned}
\end{split}
\end{equation}
In addition, the Aubin-Lions compactness theorem implies that 
\begin{equation}
\label{A:lim2}
\phi^k_\lambda \rightarrow \phi_\lambda \quad \text{strongly in } L^2(0,T; W^{1,p}(\Omega)) \text{ where $p<6$ if $d=3$ and any $p \in [2,\infty)$ if $d=2$}.
\end{equation}
Then, it is easily seen from \eqref{A:lim1} and \eqref{A:lim2} that $g_\lambda= |\nabla \phi_\lambda|^2$ and
\begin{equation}
    \label{A:lim3}
    \Psi_\lambda(\phi^k_\lambda) \rightarrow \Psi_\lambda(\phi_\lambda) \quad \text{strongly in } L^2(0,T; L^2(\Omega)).
\end{equation}
Thanks to the above convergence results, we are in the position to pass to the limit as $k\rightarrow \infty$ in \eqref{eq:gal_VCH_1}-\eqref{eq:gal_VCH_2} obtaining 
\begin{align}
\label{A:eq:lam:1}
    &( \partial_t \phi_\lambda, \xi ) 
    - \left(\phi_\lambda \vv, \nabla \xi\right)
    + \left(\nabla \mu_\lambda, \nabla \xi\right)
    =0,
    \\
\label{A:eq:lam:2}
    &(\mu_\lambda, \zeta)
    =\delta(\partial_t \phi_\lambda, \zeta)
    + \left( (\varepsilon+ \gamma |\nabla \phi_\lambda|^2) \nabla \phi_\lambda, \nabla \zeta \right)
    + \frac{1}{\varepsilon} (\Psi_\lambda'(\phi_\lambda),\zeta)
    -\frac{\alpha}{2} \left( |\dtilde|^2, \zeta \right) 
    + \beta \left( (\nabla \phi_\lambda \cdot \dtilde) \dtilde, \nabla \zeta \right),
\end{align}
for any $\xi, \zeta \in H^1(\Omega)$ and almost every $t\in (0,T)$, as well as $\phi(\cdot,0)=\phi_0$ in $\Omega$.
At this point, we rewrite \eqref{A:eq:lam:2} as
$\left( (\varepsilon+ \gamma |\nabla \phi_\lambda|^2) \nabla \phi_\lambda, \nabla \zeta \right) = 
(f_\lambda, \zeta)$, where $f_\lambda= \mu_\lambda -\delta \partial_t \phi_\lambda - \frac{1}{\varepsilon} \Psi'_\lambda(\phi_\lambda)+ \frac{\alpha}{2}|\dtilde|^2 +\beta \diver \left( (\nabla \phi_\lambda \cdot \dtilde) \dtilde \right).
$
It clearly follows from the assumptions on $\dtilde$ and \eqref{A:lim1} that $f_\lambda \in L^2(0,T; L^2(\Omega))$. Therefore, an application of \cite[Theorem 2.6]{CM2019} entails that 
$\gamma |\nabla \phi_\lambda|^2 \nabla \phi_\lambda \in L^2(0,T; H^1(\Omega;\R^n))$,
which, in turn, allows us to conclude that 
\begin{equation}
\label{A:eq:lam:3}
- \diver \left( \left(\varepsilon+ \gamma |\nabla \phi_\lambda|^2\right) \nabla \phi_\lambda \right) = 
\mu_\lambda -\delta \partial_t \phi_\lambda - \frac{1}{\varepsilon} \Psi'_\lambda(\phi_\lambda)+ \frac{\alpha}{2}|\dtilde|^2 +\beta \diver \left( (\nabla \phi_\lambda \cdot \dtilde) \dtilde \right) \quad \text{a.e. in } \Omega \times (0,T).
\end{equation}
Furthermore, by the lower-semicontinuity of the norm with respect to the weak convergence and by noticing that $\Psi_\lambda(\phi^k_\lambda(0))\rightarrow \Psi_\lambda(\phi_0)$ in $L^2(\Omega)$, we infer from letting $k\rightarrow \infty$ in \eqref{A:est7} that 
\begin{equation}
\label{A:est12}
\begin{split}
&\esssup_{t\in [0,T]}
\int_\Omega \frac{\gamma}{4}\abs{\nabla \phi_\lambda(t)}^4 + \left( \frac{\varepsilon}{2} + \frac{2\delta C_2}{\varepsilon} \right) \abs{\nabla \phi_\lambda(t)}^2 + \frac{1}{\varepsilon} \Psi_\lambda(\phi_\lambda(t))\, \d x
+ \frac12 \int_0^T \| \nabla \mu_\lambda(s)\|_{L^2(\Omega)}^2 \, \d s
\\
& \quad 
+\frac{\delta}{2}\int_0^T \| \partial_t \phi_\lambda(s)\|_{L^2(\Omega)}^2 \, \d s
+ C_2 \int_0^T \| \Delta \phi_\lambda(s)\|_{L^2(\Omega)}^2 \, \d s
+ \frac{\gamma C_2}{\varepsilon} \int_0^T \int_\Omega
\left| \nabla\left(\abs{\nabla \phi_\lambda}^2\right) \right|^2 \, \d x
\\
&\leq
2\left(\int_\Omega \frac{\gamma}{4}\abs{\nabla \phi_0}^4 + \left( \frac{\varepsilon}{2} + \frac{2\delta C_2}{\varepsilon} \right) \abs{\nabla \phi_0}^2 + \frac{1}{\varepsilon} \Psi_\lambda(\phi_0)\, \d x
+ \overline{C} \int_0^T \|\nabla \dtilde(s) \|_{L^4(\Omega)}^2 \, \d s + \overline{C}T\right)
\\
& \quad 
\times \mathrm{exp}\left(\overline{C}T + \overline{C}\int_0^T \| \nabla \dtilde(s)\|_{L^4(\Omega)}^2 + \|\vv(s) \|_{L^2(\Omega)}^2 \, \d s \right)
\\
&\leq
2\left(\int_\Omega \frac{\gamma}{4}\abs{\nabla \phi_0}^4 + \left( \frac{\varepsilon}{2} + \frac{2\delta C_2}{\varepsilon} \right) \abs{\nabla \phi_0}^2 + \frac{1}{\varepsilon} \Psi(\phi_0)\, \d x
+ \overline{C} \int_0^T \|\nabla \dtilde(s) \|_{L^4(\Omega)}^2 \, \d s + \overline{C}T\right)
\\
& \quad 
\times \mathrm{exp}\left(\overline{C}T + \overline{C}\int_0^T \| \nabla \dtilde(s)\|_{L^4(\Omega)}^2 + \|\vv(s) \|_{L^2(\Omega)}^2 \, \d s \right)=:\widetilde{R_0}.
\end{split}
\end{equation}
Here we have used that $\Psi_\lambda(s)\leq \Psi(s)$ for $|s|\leq 1$ and $\|\phi_0\|_{L^\infty(\Omega)}\leq 1$.
Now we need to recover an estimate on $\mu_\lambda$ in $H^1(\Omega)$, which is uniform in $\lambda$.
We recall that 
$$
{\color{black}
\overline{\mu_\lambda}=
\frac{1}{\varepsilon}  \overline{\Psi_\lambda'(\phi_\lambda)}  - \frac{\alpha}{2|\Omega|} \| \dtilde\|_{L^2(\Omega)}^2.
}
$$
Multiplying \eqref{A:eq:lam:3} by $\phi_\lambda-\overline{\phi_\lambda}$
and integrating over $\Omega$, and exploiting \eqref{A:est12}, we have
$$
\begin{aligned}
&\int_{\Omega} \varepsilon |\nabla \phi_\lambda|^2 + \gamma | \nabla \phi_\lambda|^4 + \beta |\nabla \phi_\lambda\cdot \dtilde|^2 \, \d x + {\color{black} \frac{1}{\varepsilon}}\int_{\Omega} F_\lambda'(\phi_\lambda) (\phi_\lambda-\overline{\phi_\lambda}) \, \d x \\
&=
-\delta \int_\Omega \partial_t \phi_\lambda 
(\phi_\lambda-\overline{\phi_\lambda}) \, \d x
+\int_{\Omega} \mu_\lambda (\phi_\lambda-\overline{\phi_\lambda}) \, \d x 
+ \Theta_0 \int_{\Omega} \phi_\lambda (\phi_\lambda-\overline{\phi_\lambda}) \, \d x
- \frac{\alpha}{2} \int_{\Omega} |\dtilde|^2 (\phi_\lambda-\overline{\phi_\lambda})\, \d x
\\
&=
-\delta \int_\Omega \partial_t \phi_\lambda 
(\phi_\lambda-\overline{\phi_\lambda}) \, \d x
+\int_{\Omega} (\mu_\lambda-\overline{\mu_\lambda}) (\phi_\lambda-\overline{\phi_\lambda}) \, \d x 
+ \Theta_0 \int_{\Omega} |\phi_\lambda-\overline{\phi_\lambda}|^2 \, \d x
- \frac{\alpha}{2} \int_{\Omega} |\dtilde|^2 (\phi_\lambda-\overline{\phi_\lambda})\, \d x
\\
&\leq 
\delta \|\partial_t \phi_\lambda \|_{L^2(\Omega)}
\| \nabla \phi_\lambda\|_{L^2(\Omega)}
+
\| \nabla \mu_\lambda\|_{L^2(\Omega)}
\| \nabla \phi_\lambda\|_{L^2(\Omega)}
+ \Theta_0
\| \nabla \phi_\lambda\|_{L^2(\Omega)}^2
+\frac{\alpha}{2} \| \dtilde\|_{L^4(\Omega)}^2 \| \nabla \phi_\lambda\|_{L^2(\Omega)}
\\
&\leq 
(1+\delta) \sqrt{\frac{2\widetilde{R_0}}{\varepsilon}}\left( \|\partial_t \phi_\lambda \|_{L^2(\Omega)}
+ \| \nabla \mu_\lambda\|_{L^2(\Omega)}\right)
+ \frac{2 \Theta_0 \widetilde{R_0}}{\varepsilon}+\frac{\alpha}{2} \sqrt{\frac{2 \widetilde{R_0}}{\varepsilon}}\| \dtilde\|_{L^4(\Omega)}^2.
\end{aligned}
$$
Taking the square, integrating over $[0,T]$, and using \eqref{A:est12} once again, we arrive at
\begin{align*}
\int_0^T {\color{black}\frac{1}{\varepsilon^2}}\left(\int_{\Omega} F_\lambda'(\phi_\lambda) (\phi_\lambda-\overline{\phi_\lambda}) \, \d x\right)^2 \, \d s
\leq 
4^3 \frac{\widetilde{R_0}^2}{\varepsilon} 
\left(\frac{1}{\delta}+1\right)
+4^2 \frac{\Theta_0^2\widetilde{R_0}^2}{\varepsilon^2}
+2 \frac{\alpha^2\widetilde{R_0}^2}{\varepsilon}
\| \dtilde\|_{L^4(0,T;L^4(\Omega))}^4=:\widetilde{R_1}.
\end{align*}
In light of the inequality (see \cite{frigeri2012nonlocal, miranville2004robust})
$$
\int_{\Omega} |F_\lambda'(\phi_\lambda)|\, \d x
\leq C^1_F \int_{\Omega} F_\lambda'(\phi_\lambda) (\phi_\lambda-\overline{\phi_\lambda}) \, \d x+ C^2_F,
$$
where $C^1_F$ and $C^2_F$ are positive constants depending only on $\overline{\phi_0}$, we deduce that
\begin{equation}
    \label{A:est15}
\| F_\lambda'(\phi_\lambda)\|_{L^2(0,T;L^1(\Omega))}
\leq \left( 2 (C_F^1)^2 {\color{black}\varepsilon^2 }
\widetilde{R_1}+ 2(C_F^2)^2 T \right)^\frac12 :=\widetilde{R_2},
\end{equation}
which, in turn, entails that 
$$
\|\overline{\mu_\lambda}\|_{L^2(0,T)}
\leq 
\left( \frac{4\widetilde{R_2}^2}{\varepsilon^2 |\Omega|^2} +\frac{4\Theta_0^2}{\varepsilon^2}|\overline{\phi_0}|^2 +
\frac{\alpha^2}{|\Omega|} \| \dtilde\|_{L^4(0,T;L^4(\Omega))}^4 \right)^\frac12
=: 
\widetilde{R_3}.
$$
As a consequence, we deduce that 
\begin{equation}
    \label{A:est16}
\| \mu_\lambda\|_{L^2(0,T; H^1(\Omega))}
\leq 2 \widetilde{R_0}^\frac12 + \sqrt{2}|\Omega|^\frac12 \widetilde{R_3}=:\widetilde{R_4}.
\end{equation}
Lastly, recalling that $|\Psi'_\lambda(s)| \leq \widetilde{C}^3_\lambda + \widetilde{C}^4_\lambda |s|$ and $F_\lambda''(s)\leq \widetilde{C}^5_\lambda$ for all $s \in \mathbb{R}$, as well as $\phi_\lambda \in L^\infty(0,T; W^{1,4}(\Omega))$, 
we observe that $F'_\lambda(\phi_\lambda(t)) \in H^1(\Omega)$ for almost every $t\in (0,T)$. Then, taking $\zeta= F'_\lambda(\phi_\lambda)$ in \eqref{A:eq:lam:2}, we find
\begin{align*}
    \int_{\Omega} &F_\lambda''(\phi_\lambda)\left[\varepsilon \abs{\nabla \phi_\lambda}^2 + \gamma \abs{\nabla \phi_\lambda}^4 + \beta \abs{\nabla \phi_\lambda \cdot \dtilde}^2\right]\, \d x + \frac{1}{2\varepsilon}  \int_\Omega |F'_\lambda(\phi_\lambda)|^2 \, \d x \\
    &\quad \leq 2 \varepsilon \| \mu_\lambda\|_{L^2(\Omega)}^2
    + 2\varepsilon \delta^2 \| \partial_t \phi_\lambda\|_{L^2(\Omega)}^2 
    + 2 \varepsilon\Theta_0^2 \| \phi_\lambda\|_{L^2(\Omega)}^2
    + \frac{\varepsilon \alpha^2}{2} \| \dtilde\|_{L^4(\Omega)}^4.
\end{align*}
Therefore, integrating in time over $[0,T]$ and recalling that $\delta \in (0,1]$, we conclude that 
\begin{equation}
    \label{A:est17}
\| F'_\lambda(\phi_\lambda) \|_{L^2(0,T; L^2(\Omega))}
\leq  \left( 4\varepsilon^2 \widetilde{R_4}^2
+8\varepsilon^2 \widetilde{R_0} 
+ 4 \varepsilon^2 \Theta_0^2 
\left( \frac{4}{\varepsilon} T \widetilde{R_0}+2T|\Omega| |\overline{\phi_0}|^2 \right)
+ \varepsilon^2\alpha^2 \| \dtilde\|_{L^4(0,T;L^4(\Omega))}^4\right)^\frac12=:
\widetilde{R_5}.
\end{equation}
Since the estimates obtained in
\eqref{A:est12}, \eqref{A:est15}, \eqref{A:est16} and \eqref{A:est17} are independent of $\lambda$, by classical compactness results we derive (up to a subsequence) that 
\begin{equation}
\label{A:lim4}
\begin{split}
\begin{aligned}
&\phi_\lambda \rightharpoonup \phi \  &&\text{weak-star in } L^\infty(0,T;W^{1,4}(\Omega)),\quad
&&&\phi_\lambda \rightharpoonup \phi \ &&&&\text{weakly in } L^2(0,T;H^2(\Omega)),\\
&\phi_\lambda \rightarrow \phi \
&&\text{strongly in }
L^2(0,T; W^{1,p}(\Omega)),\quad
&&&\partial_t \phi_\lambda \rightharpoonup \partial_t \phi \  &&&&\text{weakly in } 
L^2(0,T;L^2(\Omega)),\\
&|\nabla \phi_\lambda|^2 \rightharpoonup |\nabla \phi|^2 \  &&\text{weakly in } 
L^2(0,T;H^1(\Omega)),\quad
&&&\mu_\lambda \rightharpoonup \mu \  &&&&\text{weakly in } L^2(0,T;H^1(\Omega)),
\end{aligned}
\end{split}
\end{equation}
for $p<6$ if $d=3$ and any $p \in [2,\infty)$ if $d=2$. In order to pass to the limit in the nonlinear potential $\Psi_\lambda$, we argue exactly as in \cite[Step 2, Section 3.3]{giorgini2018cahn} to deduce that
\begin{equation}
    \phi \in L^\infty(\Omega \times (0,T)) \text{ such that } |\phi(x,t)|<1 \text{ almost everywhere in } \Omega \times (0,T).
\end{equation}
Owing to this, since (up to a subsequence) $\phi_\lambda \rightarrow \phi$ almost everywhere in $\Omega \times (0,T)$, we also obtain that $\Psi'_\lambda(\phi_\lambda) \rightarrow \Psi'(\phi)$ almost everywhere in $\Omega \times (0,T)$. By Fatou's lemma, we infer from \eqref{A:est17} that $\|F'(\phi)\|_{L^2(0,T;L^2(\Omega))} < \infty$.
Owing to this, we deduce that
$$
F'_\lambda(\phi_\lambda) \rightharpoonup F'(\phi) \quad  \text{weakly in } 
L^2(0,T;L^2(\Omega)).
$$
We are now in the position to pass to the limit as $\lambda \to 0$ in \eqref{A:eq:lam:1}-\eqref{A:eq:lam:2}. 
As a result, we find
\begin{align}
\label{A:eq:ws:1}
    &( \partial_t \phi, \xi ) 
    - \left(\phi \vv, \nabla \xi\right)
    + \left(\nabla \mu, \nabla \xi\right)
    =0,
    \\
\label{A:eq:ws:2}
    &(\mu, \zeta)
    =\delta(\partial_t \phi, \zeta)
    + \left( (\varepsilon+ \gamma |\nabla \phi|^2) \nabla \phi, \nabla \zeta \right)
    + \frac{1}{\varepsilon} (\Psi'(\phi),\zeta)
    -\frac{\alpha}{2} \left( |\dtilde|^2, \zeta \right) 
    + \beta \left( (\nabla \phi \cdot \dtilde) \dtilde, \nabla \zeta \right),
\end{align}
for any $\xi, \zeta \in H^1(\Omega)$ and almost every $t\in (0,T)$. Rewriting \eqref{A:eq:ws:2} as
$$
\left( (\varepsilon+ \gamma |\nabla \phi|^2) \nabla \phi, \nabla \zeta \right) = 
(f, \zeta), \quad \text{where}\quad f= \mu -\delta \partial_t \phi- \frac{1}{\varepsilon} \Psi'(\phi)+ \frac{\alpha}{2}|\dtilde|^2 +\beta \diver \left( (\nabla \phi \cdot \dtilde) \dtilde \right) \in L^2(0,T;L^2(\Omega)),
$$
a further application of \cite[Theorem 2.6]{CM2019} gives that 
$\gamma |\nabla \phi|^2 \nabla \phi \in L^2(0,T; H^1(\Omega;\R^n))$, from which we deduce that \eqref{A:WS:2} holds.
Furthermore, by the lower semicontinuity of the norm with respect to the weak convergence, the limit functions $\phi$ and $\mu$ fulfil \eqref{A:WS_satisfies} and
\begin{equation}
    \label{A:est:WS}
    \begin{aligned}
&\esssup_{t\in [0,T]}
\| \phi(t)\|_{W^{1,4}(\Omega)}^4 
+ \int_0^T \| \mu(s)\|_{H^1(\Omega)}^2 \, \d s
+\int_0^T \| \partial_t \phi(s)\|_{L^2(\Omega)}^2 \, \d s
\\
& \quad 
+ \int_0^T \| \phi(s)\|_{H^2(\Omega)}^2 \, \d s
+ \int_0^T \int_\Omega
\left| \nabla\left(\abs{\nabla \phi}^2\right) \right|^2 \, \d x
+  \int_0^T\norm{F'(\phi(s))}_{L^2(\Omega)}^2\, \d s
\leq \widetilde{R_6}
    \end{aligned}
\end{equation}
holds for some positive constant $\widetilde{R_6}$ that is a continuous monotone function of $\widetilde{R_0}, \dots, \widetilde{R_5}$ and the parameters of the system.

In order to conclude the proof, we are left to consider the general case in which $\phi_0 \in W^{1,4}(\Omega)$ such that $\| \phi_0\|_{L^\infty(\Omega)}\leq 1$ and {\color{black} $|\overline{\phi_0}|<1$}. To this end, let $\varphi_\eta=\frac{1}{\eta^n}\varphi(\frac{\cdot}{\eta})$ for $\eta\in (0,1)$, where $\varphi$ is a radial, positive, smooth function supported in $B_{\R^n}(1)$ with $\int_{\R^n} \varphi \, \d x=1$. We define $\phi_0^\eta= \phi_0 \ast \varphi_\eta$. Clearly, $\phi_0^\eta \in C^\infty(\Omega)$ with $\| \phi_0^\eta\|_{L^\infty(\Omega)}\leq 1$ and $|\overline{\phi_0^\eta}|=|\overline{\phi_0}|<1$. Thanks to the previous analysis, for any $\eta \in (0,1)$ there exists a pair of functions $\phi_\eta, \mu_\eta : \Omega \times (0,T) \rightarrow \R$ such that 
\begin{align*}
     & \phi_\eta \in L^\infty(0,T;W^{1,4}(\Omega))\cap L^2(0,T;H^2(\Omega)),\\
    &\phi_\eta \in L^\infty(\Omega \times (0,T)) \, \text{with }
    |\phi_\eta(x,t)|<1 \, \text{a.e. in } \Omega \times (0,T),\\
    &\partial_t \phi_\eta \in L^2(0,T; L^2(\Omega)), 
    \quad F'(\phi_\eta) \in L^2(0,T;L^2(\Omega)),\\
    &\mu_\eta \in L^2(0,T;H^1(\Omega)),
\end{align*}
which satisfies
\begin{align}
\label{A:eq:eps:1}
    &( \partial_t \phi_\eta, \xi ) 
    - \left(\phi_\eta \vv, \nabla \xi\right)
    + \left(\nabla \mu_\eta, \nabla \xi\right)
    =0, \quad \forall \, \xi \in H^1(\Omega), \text{a.e. in } (0,T),
    \\
\label{A:eq:eps:2}
   &\delta \partial_t \phi_\eta- \diver \left( \left(\varepsilon+ \gamma |\nabla \phi_\eta|^2\right) \nabla \phi_\eta \right) 
+ \frac{1}{\varepsilon} \Psi'(\phi_\eta)- \frac{\alpha}{2}|\dtilde|^2 -\beta \diver \left( (\nabla \phi_\eta \cdot \dtilde) \dtilde \right)
= 
\mu_\eta  \quad \text{a.e. in } \Omega \times (0,T),
\end{align}
as well as $\phi_\eta(\cdot, 0)=\phi_0^\eta(\cdot)$ in $\Omega$. In light of \eqref{A:est:WS}, we infer that
\begin{equation}
    \label{A:est:WS_eta}
    \begin{aligned}
&\esssup_{t\in [0,T]}
\| \phi_\eta(t)\|_{W^{1,4}(\Omega)}^4 
+ \int_0^T \| \mu_\eta(s)\|_{H^1(\Omega)}^2 \, \d s
+\int_0^T \| \partial_t \phi_\eta(s)\|_{L^2(\Omega)}^2 \, \d s
\\
& \quad 
+ \int_0^T \| \phi_\eta(s)\|_{H^2(\Omega)}^2 \, \d s
+ \int_0^T \int_\Omega
\left| \nabla\left(\abs{\nabla \phi_\eta}^2\right) \right|^2 \, \d x
+  \int_0^T\norm{F'(\phi_\eta(s))}_{L^2(\Omega)}^2\, \d s
\leq \widetilde{R_5}.
    \end{aligned}
\end{equation}
Notice that the right-hand side $\widetilde{R_5}$ depends on $\phi_0^\eta$ only through $\| \phi_0^\eta\|_{W^{1,4}(\Omega)}$, $\| \phi_0^\eta\|_{L^\infty(\Omega)}$ and $|\overline{\phi_0^\eta}|$. Since $\| \phi_0^\eta\|_{W^{1,4}(\Omega)}$ is bounded as $\eta \rightarrow 0$, and $\| \phi_0^\eta\|_{L^\infty(\Omega)}\leq 1$ and 
$|\overline{\phi_0^\eta}|=|\overline{\phi_0}|<1$
for any $\eta \in (0,1)$, we conclude that all the norms on the left-hand side in \eqref{A:est:WS_eta} are bounded uniformly in $\eta$. A similar argument as the one performed above leads to the existence of a limit pair $(\phi,\mu)$ satisfying \eqref{A:WS:R1}-\eqref{A:WS:R4}, the variational formulation \eqref{A:WS:1}-\eqref{A:WS:2} and the initial condition $\phi(\cdot, 0)=\phi_0(\cdot)$ in $\Omega$, as well as the global estimate \eqref{A:WS_satisfies}. The proof is completed.
\end{proof}

\end{appendices}

\section*{Acknowledgements}
The authors wish to thank Francesco Ballarin, Andrea Cianchi and Cristiana De Filippis for fruitful discussions, {\color{black} and the two reviewers for helpful comments that improved the quality of our paper}.\\
GB is supported by the European Research Council (ERC), under the European Union's Horizon 2020 research and innovation program, through the project ERC VAREG - {\em Variational approach to the regularity of the free boundaries} (grant agreement No. 853404) and by Gruppo Nazionale per l'Analisi Matematica, la Probabilit\`a e le loro Applicazioni (GNAMPA) of Istituto Nazionale di Alta Matematica (INdAM) through the INdAM-GNAMPA project 2024 CUP E53C23001670001. GB acknowledges the MIUR Excellence Department Project awarded to the Department of Mathematics, University of Pisa, CUP I57G22000700001.\\
AG is supported by the MUR grant Dipartimento di Eccellenza 2023-2027 of Dipartimento di Matematica, Politecnico di Milano, and by Gruppo Nazionale per l'Analisi Matematica, la Probabilit\`a e le loro Applicazioni (GNAMPA) of Istituto Nazionale per l'Alta Matematica (INdAM).

\section*{Competing interests and funding}
The authors do not have any financial or non-financial interests that are directly or indirectly related to the work submitted for publication.

\section*{Data availability statement}
No further data is used in this manuscript.

\printbibliography

\end{document}